\documentclass{amsart}
\usepackage[english]{babel}

\usepackage{amssymb, amsmath,mathtools,mathabx}
\usepackage{mathrsfs}%stix}
\usepackage{amscd}
\usepackage{verbatim}
\usepackage{stmaryrd}
\usepackage[dvipsnames]{xcolor}

\usepackage{enumerate}

\usepackage[colorlinks,linkcolor={blue},citecolor={blue},urlcolor={purple},]{hyperref}

\usepackage[colorinlistoftodos,prependcaption,textsize=tiny]{todonotes}

\usepackage{forest}

\usepackage{tabularx}
\newcolumntype{L}{>{\arraybackslash}X}
\usepackage{multirow}

\theoremstyle{plain}
\newtheorem{theorem}{Theorem}[section]
\theoremstyle{remark}
\newtheorem{remark}[theorem]{Remark}

\theoremstyle{plain}

\newtheorem{lemma}[theorem]{Lemma}
\newtheorem{proposition}[theorem]{Proposition}
\newtheorem{definition}[theorem]{Definition}

\newtheorem{assumption}[theorem]{Assumption}

\numberwithin{equation}{section}

\def\N{{\mathbb N}}
\def\Z{{\mathbb Z}}
\def\Q{{\mathbb Q}}
\def\R{{\mathbb R}}

\newcommand{\E}{{\mathbf E}}
\renewcommand{\P}{{\mathbf P}}
\newcommand{\F}{{\mathscr F}}

\newcommand{\G}{{\mathrm G}}

\newcommand{\om}{\omega}
\renewcommand{\O}{\Omega}

\renewcommand{\a}{\kappa}

\renewcommand{\b}{\mathcal{B}}
\newcommand{\bb}{\mathcal{B}}

\newcommand{\Dom}{\mathscr{O}}

\newcommand{\Tor}{\mathbb{T}}
\newcommand{\W}{\mathcal{O}}

\newcommand{\A}{{\mathcal A}}

\newcommand{\X}{\mathscr{X}}

\newcommand{\loc}{\mathrm{loc}}

\newcommand{\calL}{{\mathscr L}}

\newcommand{\T}{\mathbb{T}}

\newcommand{\hz}{\prescript{}{0}{H}}

\newcommand{\hp}{c}
\newcommand{\D}{\mathscr{D}}

\newcommand{\Ls}{\mathbb{L}}

\newcommand{\Hs}{\mathbb{H}}
\newcommand{\Bs}{\mathbb{B}}
\newcommand{\J}{\mathbb{J}}
\newcommand{\q}{\mathbb{Q}}

\newcommand{\e}{\mathrm{E}}

\newcommand{\Sol}{\mathscr{R}}

\newcommand{\tT}{T^{*}}

\newcommand{\wt}{\widetilde}

\usepackage{stmaryrd}

\newcommand{\MRtas}{\mathcal{SMR}_{p,\a}^{\bullet}(s,T)}

\newcommand{\MRtatz}{\mathcal{SMR}_{p,\a}(t,T)}

\newcommand{\Tr}{\mathsf{Tr}}
\newcommand{\Xap}{X^{\mathsf{Tr}}_{\a,p}}

\newcommand{\Yr}{Y^{\mathsf{Tr}}_{r}}

\newcommand{\one}{{{\bf 1}}}
\newcommand{\embed}{\hookrightarrow}
\newcommand{\s}{\delta}
\newcommand{\lb}{\langle}
\newcommand{\rb}{\rangle}

\renewcommand{\div}{\mathrm{div}}
\newcommand{\B}{B}
\newcommand{\Stok}{\mathbb{S}}

\newcommand{\p}{\mathbb{P}}

\newcommand{\supp}{\text{supp}\,}
\DeclareMathOperator*{\esssup}{ess\,sup}

\renewcommand{\l}{\langle}
\renewcommand{\r}{\rangle}

\newcommand{\crit}{\mathrm{c}}

\newcommand{\norm}[1]{{\left\vert\kern-0.25ex\left\vert\kern-0.25ex\left\vert #1
    \right\vert\kern-0.25ex\right\vert\kern-0.25ex\right\vert}}

\renewcommand{\emptyset}{\varnothing}

\newcommand{\Ff}{\mathrm{F}}

\newcommand{\Progress}{\mathscr{P}}
\newcommand{\MeasurableP}{\mathscr{A}}

\newcommand{\btwod}{b}

\newcommand{\Fd}{f}
\newcommand{\Gforce}{g}

\newcommand{\Borel}{\mathscr{B}}
\newcommand{\Constant}{R}

\newcommand{\ellip}{\nu}

\newcommand{\reg}{\delta}
\newcommand{\wh}{\widehat}

\newcommand{\Set}{\mathcal{S}}
\newcommand{\partition}{\pi}

\newcommand{\vone}{v}
\newcommand{\vtwo}{v'}

\newcommand{\noise}{\wt{f}}

\newcommand{\wtwo}{\ell}

\newcommand{\usigma}{u_{\stopp}}
\newcommand{\stopp}{\varrho}
\newcommand{\pp}{\mathcal{P}}
\newcommand{\Lc}{\mathcal{L}}

\newcommand{\dd}{\mathrm{d}}

\newcommand{\smallpar}{\chi}

\allowdisplaybreaks

\begin{document}

\author{Antonio Agresti}
\address{Department of Mathematics,
	TU Kaiserslautern, Paul-Ehrlich-Stra{\ss}e 31,
	67663 Kaiserslautern, Germany}
\curraddr{Delft Institute of Applied Mathematics\\
Delft University of Technology \\ P.O. Box 5031\\ 2600 GA Delft\\The
Netherlands}
\email{antonio.agresti92@gmail.com}

\author{Mark Veraar}
\address{Delft Institute of Applied Mathematics\\
Delft University of Technology \\ P.O. Box 5031\\ 2600 GA Delft\\The
Netherlands} \email{M.C.Veraar@tudelft.nl}

\thanks{
The first author has been partially supported by the Nachwuchsring -- Network for the promotion of young scientists -- at TU Kaiserslautern.
The second author is supported by the VIDI subsidy 639.032.427 of the Netherlands Organisation for Scientific Research (NWO). Both authors were also supported by the VICI subsidy VI.C.212.027 of the Netherlands Organisation for Scientific Research (NWO)}

\date\today

\title[Stochastic Navier-Stokes equations in critical spaces]{Stochastic Navier-Stokes equations for\\ turbulent flows in critical spaces}



\begin{abstract}
In this paper we study the stochastic Navier-Stokes equations on the $d$-dimensional torus with transport noise, which  arise in the study of turbulent flows.
Under very weak smoothness assumptions on the data we prove local well-posedness in the critical case $\Bs^{d/q-1}_{q,p}$ for $q\in [2,2d)$ and $p$ large enough. Moreover, we obtain new regularization results for solutions, and new blow-up criteria which can be seen as a stochastic version of the Serrin blow-up criteria. The latter is used to prove global well-posedness with high probability for small initial data in critical spaces in any dimensions $d\geq 2$. Moreover, for $d=2$, we obtain new global well-posedness results and regularization phenomena which unify and extend several earlier results.
\end{abstract}

\maketitle
\tableofcontents

\section{Introduction}
\label{s:intro}
In this paper we provide a comprehensive treatment of (analytically and probabilistically) \emph{strong} solutions to stochastic Navier-Stokes equations for turbulent incompressible flows on the $d$-dimensional torus $\T^d$:
\begin{equation}
\label{eq:Navier_Stokes}
\left\{
\begin{aligned}
\dd u &=\big[\nu \Delta u -(u\cdot \nabla)u -\nabla p  \big]\, \dd t
+\sum_{n\geq 1}\big[(\btwod_{n}\cdot\nabla) u -\nabla \wt{p}_n+g_n(u)\big] \,\dd{w}_t^n,
\\
\div \,u&=0,
\\ u(\cdot,0)&=u_0.
\end{aligned}\right.
\end{equation}
Here $u:=(u^k)_{k=1}^d:[0,\infty)\times \O\times \Tor^d\to \R^d$ denotes the unknown velocity field, $p,\wt{p}_n:[0,\infty)\times \O\times \Tor^d\to \R$ the unknown pressures, $(w_t^n:t\geq 0)_{n\geq 1}$ a given sequence of independent standard Brownian motions and
\begin{equation*}
(\btwod_{n}\cdot\nabla) u:=\Big(\sum_{j=1}^d \btwod_n^j \partial_j u^k\Big)_{k=1}^d,
\qquad (u\cdot \nabla ) u:=\Big(\sum_{j=1}^d u^j \partial_j u^k\Big)_{k=1}^d.
\end{equation*}
Actually, we will consider \eqref{eq:Navier_Stokes} in a slightly more general form (see \eqref{eq:Navier_Stokes_generalized} below). For instance we replace $\nu$ by variable viscosity in divergence form, and we allow an additional deterministic forcing term.

The relation between Navier-Stokes equations and hydrodynamic turbulence is one of the most challenging problems in fluid mechanics. It is believed that the onset of the turbulence is related to the randomness of background movement. The mathematical study of stochastic perturbations of Navier-Stokes equations began with the work of Bensoussan and Temam in \cite{BT73} and it was later extended in several directions, see e.g.\ \cite{BCF91,BCF92,BF17,BP00,CP01,C78,FG95,IF79,WS00} and the references therein.  In the mathematical literature, usually, the random perturbation of the Navier-Stokes equations is postulated and does not have a clear physical motivation. In \cite{MR01,MiRo04}, Mikulevicius and Rozovskii found a new approach to derive the stochastic Navier-Stokes equations. Indeed, there it is assumed that the dynamics of the fluid particles is given by the \textit{stochastic flow}
\begin{equation}
\label{eq:NS_stochastic_flow}
\left\{
\begin{aligned}
\dd \mathcal{X}(t,x)& =u(t,\mathcal{X}(t,x))\,\dd t + \btwod(t,\mathcal{X}(t,x))\circ \dd W,\\
\mathcal{X}(0,x)& =x,\qquad x\in \Tor^d,
\end{aligned}\right.
\end{equation}
where $W$ and $\circ$ denote a white-noise and the Stratonovich integration, respectively. Using Newton's second law, in \cite{MiRo04} the authors derived the stochastic Euler and Navier-Stokes equations ``for turbulent flows'', or with transport noise.
Roughly speaking, \eqref{eq:NS_stochastic_flow} says that the velocity field splits into a sum of slow oscillating part $u\,\dd t$ and a fast oscillating part $\btwod \circ \dd W$. As noticed in \cite{MiRo04}, such splitting can be traced back to the work of Reynolds in 1880. Closely related models were proposed by Kraichnan \cite{K68,K94} in the study of turbulence transportation of passive scalars and then developed further by other authors, see e.g.\ \cite{GK96,GV00}. 
The corresponding theory of turbulence transportation is called Kraichnan's theory, and in the latter, one typically  assumes $\div \,b_n=0$ in $\D'(\T^d)$ for all $n\geq 1$, and for some small $\gamma_0>0$,
\begin{equation}
\label{eq:regularity_btwod}
\btwod=((\btwod^j_n)_{j=1}^d)_{n\geq 1}\in
C^{\gamma}(\Tor^d;\ell^2(\N_{\geq 1};\R^d)) \  \text{ for all }\  \gamma<\gamma_0,
\end{equation}
see e.g.\ \cite[Section 1]{MR05} and \cite[pp. 426-427 and 436]{MajKra99}. The reader is also referred to \cite[Section 5]{GY21} for details on the Kraichnan model. The case $\gamma_0=\frac{2}{3}$ in \eqref{eq:regularity_btwod} is related to the Kolmogorov spectrum of turbulence, see \cite[Subsection 2.1]{agresti2023primitive} and \cite[pp. 426]{MajKra99}.
The reader is referred to \cite{BF20,F_intro,MR01} for additional motivations. 
With this application in mind, in the current work, we consider the more general case of $b\in H^{\eta,\xi}(\T^d;\ell^2)$ where $H$ is the Bessel potential space and $\eta>0$, $\xi\in [2,\infty)$ satisfy $\eta>\frac{d}{\xi}$.
In light of this, the vorticity formulation of \eqref{eq:Navier_Stokes} (i.e.\ for the unknown $\text{curl}\, u$) is in general not available.
The reader interested in the Navier-Stokes equations in vorticity form may consult \cite{F15_lectures_saint_flour,F15_open,FL19,HLN21} and the references therein.

The problem \eqref{eq:Navier_Stokes} will be considered in the It\^o formulation. In the literature the Stratonovich formulation of \eqref{eq:Navier_Stokes} is also used frequently, see e.g.\ \cite{DP22_two_scale,FP22_2Dfluids}
 where the Navier-Stokes with transport noise in Stratonovich form is derived based on two-scale type approximations.
In the latter case for the transport term one replaces $\sum_{n\geq 1}[ (b_n\cdot\nabla) u-\nabla \wt{p}_n]\,\dd w_t^n$ by $\sum_{n\geq 1}[ (b_n\cdot\nabla) u-\nabla \wt{p}_n]\circ \dd w_t^n$. In the case $g\equiv 0$, $b_n$ is $(t,\om)$-independent and $\div\,b_n=0$, at least formally,
\begin{align}
\label{eq:stratonovich_correction}
\sum_{n\geq 1}\big[(b_n\cdot\nabla) u- \nabla \wt{p}_n\big]\circ \dd w_t^n
&= \sum_{n\geq 1}\big[
(b_n\cdot\nabla) u- \nabla \wt{p}_n\big]\,\dd w_t^n\\
\nonumber
+\Big[  \div( a_b\cdot \nabla u) & - \frac{1}{2}\sum_{n\geq 1} \div(b_n\otimes  \nabla\wt{p}_n) -\nabla q  \Big]\, \dd t,
\end{align}
where $a_b:=(\frac{1}{2}\sum_{n\geq 1} b^j_n b^i_n)_{i,j=1}^d$, $q$ is a scalar function which can be absorbed into the (deterministic) pressure $p$ and $\div(b_n\otimes  \nabla\wt{p}_n)=(\sum_{j=1}\partial_j (b_n^j \partial_k \wt{p}_n))_{k=1}^d$. The reader is referred to \cite[Subsection 1.3]{FL19} for details (see also \cite[Section 3.1]{FGL21} in absence of pressure terms). The additional deterministic term on the RHS\eqref{eq:stratonovich_correction} is referred to as the It\^o correction.
As recalled before \eqref{eq:regularity_btwod}, the divergence free condition on $b_n$ is physically justified. Since in  \eqref{eq:Navier_Stokes_generalized} we will also consider the differential operators appearing in the deterministic part of RHS\eqref{eq:stratonovich_correction}, our results also cover the case of transport noise in Stratonovich form.

It is not possible to discuss all relevant papers on stochastic Navier-Stokes equations here, and we mostly restrict ourselves to the ones where transport noise is included. There are several important papers concerned with martingale solutions to \eqref{eq:Navier_Stokes}, see e.g.\ \cite{BM13_unbounded,FG95,MR05,MR99} and the references therein. In these papers global existence is the main point, and uniqueness is open unless $d=2$. See also \cite{HLN19} for a rough path theory approach.

\subsection{Scaling and criticality}
\label{ss:scaling}
Before we discuss some of our main results in more detail, we will discuss scaling and criticality in case \eqref{eq:Navier_Stokes} is considered on $\R^d$ instead of $\T^d$ as the scaling is easier in the $\R^d$-case. This will also motivate the function spaces and the growth condition on $g$ used in the manuscript. In the deterministic setting (e.g.\ $\btwod\equiv 0$ and $\Gforce\equiv 0$), it is well-known that (local smooth) solutions to  \eqref{eq:Navier_Stokes} are (roughly) invariant under the map $u\mapsto u_{\lambda}$ where $\lambda >0$ and
\begin{equation}
\label{eq:NS_scaling_map}
u_{\lambda}(t,x)=\lambda^{1/2}u(\lambda t,\lambda^{1/2} x),\qquad
 (t,x)\in \R_+\times\R^d.
\end{equation}
In the PDE literature (see e.g.\ \cite{Can04,LePi,PW18,Trie13}), Banach spaces of functions (locally) invariant under the induced map on the initial data, i.e.\
$$
u_0\mapsto u_{0,\lambda} \quad \text{ where } \quad u_{0,\lambda} :=\lambda^{1/2} u_{0}(\lambda^{1/2}\cdot ),
$$ 
are called \textit{critical} for \eqref{eq:Navier_Stokes}. Examples of such critical spaces are the Besov space
$
B^{d/q-1}_{q,p}(\R^d;\R^d)
$
and the Lebesgue space
$
L^d(\R^d;\R^d)
$. Indeed, the corresponding homogeneous spaces satisfy
$$
\|u_{0,\lambda}\|_{\dot{B}^{d/q-1}_{q,p}(\R^d;\R^d)}\eqsim \|u_0\|_{\dot{B}^{d/q-1}_{q,p}(\R^d;\R^d)}, \qquad
\|u_{0,\lambda}\|_{L^d(\R^d;\R^d)}\eqsim \|u_0\|_{L^d(\R^d;\R^d)},
$$
where the implicit constants do not depend on $\lambda>0$.

In the stochastic case a similar behavior appears. More precisely, one can check that if $u$ is a (local smooth) solution to \eqref{eq:Navier_Stokes} on $\R^d$, then $u_{\lambda}$ is a (local smooth) solution to \eqref{eq:Navier_Stokes} on $\R^d$, where $(w^n)_{n\geq 1}$ is replaced by the \emph{scaled} noise $(\beta_{\cdot,\lambda}^n)_{n\geq 1}$ where $\beta_{t,\lambda}^n:=\lambda^{-1/2}w_{\lambda t}^n$ for $t\geq 0$ and $n\geq 1$.  Moreover, the scaling  \eqref{eq:NS_scaling_map} also suggests that the nonlinearity $g_n(u)$ has to growth quadratically in $u$ in order to keep the same scaling. Indeed, suppose that $g_n(u)=G_n(u,u)$ where $G_n:\R^d\times \R^d\to \R^d$ is bilinear. In such a case, roughly speaking, for all $\lambda>0$, $x\in \R^d$ and $n\geq 1$,  it holds that 
\begin{align*}
\int_0^{t/\lambda} g_n(u_{\lambda}(s,x)) \,\dd \beta_{s,\lambda}^n
&=
\int_0^{t/\lambda} \lambda G_n(u_{\lambda}(\lambda s,\lambda^{1/2}x),u_{\lambda}(\lambda s,\lambda^{1/2}x))
\,\dd \beta_{s,\lambda}^n \\
&=\lambda^{1/2}
\int_0^{t}  G_n(u_{\lambda}(s,\lambda^{1/2}x),u_{\lambda}(s,\lambda^{1/2}x))\, \dd w_{s}^n,
\end{align*}
which agrees with the scaling of the deterministic nonlinearity
\begin{align*}
\int_0^{t/\lambda}\big(u_{\lambda}(s,x)\cdot \nabla\big)u_{\lambda}(s,x)\, \dd s
&= \lambda^{1/2}\int_0^{t} \big(u(s,\lambda^{1/2}x)\cdot \nabla \big)u(s,\lambda^{1/2}x) \,\dd s.
\end{align*}
A similar scaling argument also holds for the remaining deterministic and stochastic integrals. In particular, the $b$-term has the correct scaling under the map \eqref{eq:NS_scaling_map}.

The above discussion shows that the stochastic perturbations of gradient/transport type and/or of quadratic type in $u$ respect the scaling of the deterministic Navier-Stokes equations. In the present manuscript we will consider both situations. Furthermore, for our results to hold, we do not need $g_n$ to be bilinear. A quadratic growth condition
will be sufficient.

Finally, we mention that our methods can be also extended to the case where $g_n$ grows more than quadratically. However, the above scaling argument fails, and the corresponding critical spaces change accordingly to the growth of $g_n$ (see \cite[Subsection 5.2]{AV19_QSEE_1} for related results). Since we are not aware of physical motivations for such behaviour of $g_n$, we leave the details to the interested reader.

\subsection{Overview of the main results}
The aim of this paper is to provide a first systematic treatment of the stochastic Navier-Stokes equations \eqref{eq:Navier_Stokes} arising in study of turbulent flows in the \emph{strong} setting (more precisely $L^q(H^{s,q})$-setting).
Below we briefly list our main results.
\begin{enumerate}[(a)]
\item\label{it:local_intro} Local well-posedness in \emph{critical} spaces (Theorem \ref{t:NS_critical_local}).
\item\label{it:regularization} Regularization with \emph{optimal} gain in the Sobolev scale (Theorems \ref{t:NS_critical_local} and \ref{t:NS_high_order_regularity}).
\item\label{it:serrin_intro} Stochastic \emph{Serrin}'s blow-up criteria (Theorem \ref{t:NS_serrin}).
\item\label{it:global_intro_small_data} Global well-posedness for small data in \emph{critical} spaces (Theorem \ref{t:global_small_data}).
\item\label{it:global_intro} Global well-posed in case $d=2$ with \emph{rough} data (Theorem \ref{t:NS_initial_data_large_two_dimensional}).
 \end{enumerate}

\eqref{it:local_intro}: In Theorem \ref{t:NS_critical_local} (see also Theorem \ref{t:NS_critical_localgeneralkappa})
we show existence of \emph{strong unique} solutions to \eqref{eq:Navier_Stokes} for initial values $u_0$ belonging to the critical Besov space $B^{d/q-1}_{q,p}(\Tor^d;\R^d)$, where $q$ varies in a certain range (depending on the smoothness of $b_n$) and $p<\infty$ can be chosen as large as one wants. The embedding $L^d(\Tor^d;\R^d)\embed B^0_{d,p}(\Tor^d;\R^d)$ for $p\geq d$ shows that our results also hold for the critical space $L^d(\Tor^d;\R^d)$. Let us remark that, in all cases, we have the restriction $q<2d$ and therefore we allow critical spaces with smoothness up to $-\frac{1}{2}$. We are not aware of the optimality of this threshold. Note that this behavior differs from the deterministic case where $q$ can be chosen to be large.
Next we compare the result of Theorem \ref{t:NS_critical_local} to the existing literature.
In \cite{MiRo04} the authors showed (local) existence of strong solutions to \eqref{eq:Navier_Stokes} on $\R^d$ for divergence free vector field for initial values from the (non critical) space $W^{1,q}(\R^d;\R^d)$ with $q>d$ under a restrictive assumption on the regularity of $b$. In particular, the smoothness assumption there does not satisfy \eqref{eq:regularity_btwod} if $\gamma_0$ is small.
Existence of strong solutions to \eqref{eq:Navier_Stokes} in an $L^p(L^q)$-setting has been also considered in
\cite{DZ20_critical_NS,W20}. However, in the latter references, the Kraichnan term $(b_n\cdot \nabla) u$ is not included.

\eqref{it:regularization}: One question which will be answered in this paper is how the regularity of $b_n$ affects the regularity of $u$. The term $(b_n\cdot\nabla )u$ suggests that $u$ can gain at most one order of smoothness compared to $b_n$. In Theorem \ref{t:NS_high_order_regularity} we prove an \emph{optimal} gain of regularity in the Sobolev scale, i.e.\ for all $\eta>0$ and $\xi\in [2,\infty)$ we have
$$
(b_n^j)_{n\geq 1}\in H^{\eta,\xi}(\Tor^d;\ell^2)
\ \
\Longrightarrow \ \
u\in L^r_{\loc}(0,\sigma;H^{\eta+1,\xi}(\Tor^d;\R^d))\text{ a.s. }
$$
where $r\in (2,\infty)$ is arbitrary and $\sigma$ denotes the explosion time of $u$. In the H\"{o}lder scale, by Remark \ref{r:ass_NS} and Theorem \ref{t:NS_high_order_regularity}, we prove a \emph{sub-optimal} gain of regularity, i.e.\ for all $\eta>0$ and $\varepsilon>0$ and $\theta\in (0,1/2)$ we have
$$
(b_n^j)_{n\geq 1}\in C^{\eta}(\Tor^d;\ell^2)
\ \
\Longrightarrow \ \
u(t)\in C^{\theta,1+\eta-\varepsilon}_{\rm loc}((0,\sigma)\times \Tor^d;\R^d)\text{ a.s.}
$$
The above results are \emph{new} even in the important case $d=2$ where $\sigma=\infty$ a.s. Indeed, for $(b_n^j)_{n\geq 1}\in C^{\eta}(\Tor^d;\ell^2)$ with $\eta>0$ small, the regularization cannot be proved with classical bootstrap methods, see the text below Theorem \ref{t:NS_initial_data_large_two_dimensional}. However, in the case $b_n$ is regular enough, (some) gain of regularity can be proved by standard methods (see e.g.\ \cite{M09}) or via the previously mentioned vorticity formulation of \eqref{eq:Navier_Stokes}. But for rough   $(b_n^j)_{n\geq 1}$ these approaches are not applicable.

\eqref{it:serrin_intro}: To the best of our knowledge, blow-up criteria for stochastic Navier-Stokes equations have not been studied so far in the literature. In
Theorem \ref{t:NS_serrin} we provide an extension to the classical Serrin's criteria (see e.g.\ \cite[Theorem 11.2]{LePi}).

\eqref{it:global_intro_small_data}: In Theorem \ref{t:global_small_data} we prove existence for large times in probability of smooth solutions to \eqref{eq:Navier_Stokes} under smallness assumptions on $u_0$ and on the \emph{additive} part of $g_n$. In deterministic framework this result is well-known and for $H^{1/2}$-data (here $d=3$) is due to Fujita and Kato in the seminal paper \cite{FS64_NS}. Later the latter was extended to $B^{d/q-1}_{q,p}$-data and even to $\mathrm{BMO}^{-1}$-data by Koch and Tataru \cite{KT01}. In the stochastic framework, under smallness assumptions on the noise, some results in this direction can be found in \cite{DZ20_critical_NS,K10,KXZ21}. Note that in \cite{DZ20_critical_NS,KXZ21} the transport noise term $(b_n\cdot \nabla) u$ is not allowed, and in \cite{K10} a small transport noise term can be used, where an additional smallness condition on the data is used, but in a non critical space.
In Theorem \ref{t:global_small_data} we are able to give an extension of the previous results, where we also allow $b$ to be \emph{not} small and $g_n$ of quadratic type. As in \eqref{it:local_intro}, the smallness condition on $u_0$ is given in a $B^{d/q-1}_{q,p}$-norm with smoothness up to $-\frac{1}{2}$.

%%%
\eqref{it:global_intro}: In Theorem \ref{t:NS_initial_data_large_two_dimensional}
 we prove \emph{global} existence for \eqref{eq:Navier_Stokes} for initial data which are much \emph{rougher} than data from the usual energy space $L^2(\Tor^2;\R^2)$. More precisely, we prove the existence of $\zeta\in (2,4)$ (depending on the smoothness of $b_n$) such that \eqref{eq:Navier_Stokes} is globally well-posed for data in $B^{2/\zeta-1}_{\zeta,p}(\Tor^2;\R^2)$ where $p<\infty$ is large. Since $L^2(\Tor^2;\R^2)\embed B^{2/\zeta-1}_{\zeta,p}(\Tor^2;\R^2)$ for $p\geq 2$,  this result extends the well-known global existence result for $L^2$-data in two dimensions. For the reader's convenience, we also provide a self-contained proof of the latter result in Appendix \ref{s:global_L_2_rough_noise} which also improves and unifies several existing results, see Remark \ref{r:comparison_2D_other_results}. In the deterministic case similar results have been shown in \cite{GP02}, and also for rough forcing terms in \cite{BF09}.

\medskip
To prove our main results we apply the full power of the local well-posedness theory which we recently developed in \cite{AV19_QSEE_1} together with the blow-up criteria and new bootstrapping techniques of our paper \cite{AV19_QSEE_2}. In these works we completely revised the theory of abstract stochastic evolution equations for semi- and quasi-linear parabolic SPDEs, and is the stochastic analogue of \cite{CriticalQuasilinear} by Pr\"uss, Simonett and Wilke. One of the main ingredients in our theory is the so-called stochastic maximal $L^p$-regularity for the linear part, which we prove in an $H^{s,q}$-setting. In case of the Navier-Stokes equations this linear part is the turbulent Stokes system, for which we will prove stochastic maximal $L^p$-regularity in an $H^{s,q}(\T^d;\R^d)$-setting in Section \ref{s:max_reg_linear}, and it can be seen as one of the main result of our work as well.

\subsection{Notation}
Here we collect some notation which will be used throughout the paper. Further notation will be introduced where needed.
As usual, we write $A \lesssim_P B$ (resp.\ $A\gtrsim_P B$) whenever there is a constant $C$ only depending on the parameter $P$ such that $A\leq C B$ (resp.\ $A \geq C B$). We write
$C_{P}$ or $C(P)$ if $C$ depends only on $P$. Let $a\vee b = \max\{a,b\}$ and $a\wedge b=\min\{a,b\}$ where $a,b\in \R$. For a $(Y,\mathsf{d}_Y)$ a metric space, we set
$\B_{Y}(y,\eta):=\{y'\in Y: \mathsf{d}_Y(y,y')<\eta\}$.

For $a,\a\in \R$, the weights will be denoted by
$w_{\a}^a=|t-a|^{\a}$ and $w_{\a}:=w_{\a}^0$.
For a Banach space $X$, $a,b\in \R$ and an interval $(a,b)\subseteq \R_+$, $L^p(a,b,w_{\a}^a;X)$ denotes the set of all strongly measurable maps $f:(a,b)\to X$ such that
$$
\|f\|_{L^p(a,b,w_{\a}^a;X)}:=\Big(\int_a^b \|f(t)\|_{X}^p \, w^{a}_{\a}(t)\,\dd t\Big)^{1/p}<\infty.
$$
Moreover, $W^{1,p}(a,b,w_{\a}^a;X)\subseteq L^p(a,b,w_{\a}^a;X)$ denotes the set of all $f$ such that $f'\in L^p(a,b,w_{\a}^a;X)$ and we set
$$
\|f\|_{W^{1,p}(a,b,w_{\a}^a;X)}:=
\|f\|_{L^p(a,b,w_{\a}^a;X)}+
\|f'\|_{L^p(a,b,w_{\a}^a;X)}.
$$

Let $[\cdot,\cdot]_{\theta}$ and $(\cdot,\cdot)_{\theta,p}$ be the complex and real interpolation functor, see e.g.\ \cite{BeLo,Tri95}. For each $\theta\in (0,1)$, we set
$$
H^{\theta,p}(a,b,w_{\a}^a;X):=[L^p(a,b,w_{\a}^a;X),W^{1,p}(a,b,w_{\a}^a;X)]_{\theta}.
$$
The above spaces can also be equivalently introduced by using Fourier methods on $\R$ and restrictions, see e.g.\ \cite[Definition 3.1]{ALV21} and the references therein.

For $\A\in \{L^p,H^{\theta,p}\}$, we denote by $\A_{\loc}(a,b;X)$ (resp.\ $\A_{\loc}([a,b);X)$) the set of all strongly measurable maps $f:(a,b)\to X$ such that $f\in\A(c,d;X)$ for all
$a<c<d<b$ (resp.\ $f\in \A(a,d;X)$ for all $a<d<b$).

The $d$-dimensional torus is denoted by $\Tor^d$. For $m\geq 1$ and $s,r\in (0,1)$, $C^{s,r}_{\loc}((a,b)\times \Tor^d;\R^m)$ denotes the space of all maps $v:(a,b)\times \T^d \to \R^m$ such that for all $a<a'<b'<b$ we have
$$
|v(t,x)-v(t',x')|\lesssim_{a',b'} |t-t'|^{s}+|x-x'|^{r}, \qquad t,t'\in [a',b'], \, x,x'\in \Tor^d.
$$

Finally, we collect the main probabilistic notation. Through the paper, we fix a filtered probability space $(\O,\MeasurableP,(\F_t)_{t\geq 0}, \P)$ and we let $\E[\cdot]:=\int_{\O}\cdot\, \dd\P$.
A mapping $\tau:\O\to [0,\infty]$ is said to be a stopping time if $\{\tau\leq t\}\in \F_t$ for all $t\geq 0$. For a stopping time $\tau$ we let
$$
[0,\tau]\times \O:=\{(t,\om)\,:\, 0\leq t\leq \tau(\om)\}.
$$
Similar definitions hold for $[ 0,\tau)\times \O$, $(0,\tau)\times \O$ etc. $\Progress$
denotes the progressive $\sigma$-algebra on the above mentioned filtered probability space.

\section{Statement of the main results}
\label{s:statements}
In this section we state our main results on the following stochastic Navier-Stokes equations on the $d$-dimensional torus $\T^d$
\begin{equation}
\label{eq:Navier_Stokes_generalized}
\left\{\begin{aligned}
\dd  u & =\Big[\div(a\cdot \nabla u) -\div(u\otimes u)+\Fd_0(\cdot, u)+\div(f(\cdot, u))\\
&\qquad\qquad\qquad\qquad\qquad\qquad \qquad\qquad
 -\nabla p+\partial_h \wt{p}+\partial_{\hp}^2 \wt{p}\, \Big]\,\dd t \\
& \
+\sum_{n\geq 1}\big[(\btwod_n \cdot \nabla ) u +\Gforce_{n}(\cdot,u) -\nabla \wt{p}_n\big] \,\dd {w}_t^n,
\\ \div \,u &=0,
\\ u(\cdot,0) &=u_0,
\end{aligned}\right.
\end{equation}
where we have rewritten the convective term $(u\cdot\nabla)u$ in \eqref{eq:Navier_Stokes} in the standard conservative form $\div(u\otimes u)$. Here $u=(u^k)_{k=1}^d:[0,\infty)\times \O\times \Tor^d\to \R^d$ is the unknown velocity field, $p,\wt{p}_n:[0,\infty)\times \O\times \Tor^d\to \R$ the unknown pressures, $(w^n)_{n\geq 1}$ is a sequence of standard independent Brownian motions on the previously mentioned filtered probability space $(\O,\MeasurableP,(\F_t)_{t\geq 0}, \P)$, $u\otimes u:=(u^i u^j)_{i,j=1}^d$,
\begin{equation}
\label{eq:def_main_operators_Navier_Stokes_generalized}
 \begin{aligned}
\div(a\cdot \nabla u) &=\Big(\sum_{i,j=1}^d\partial_i(a^{i,j}\partial_j u^k)\Big)_{k=1}^d,
&
\partial_h \wt{p}&=\Big(\sum_{n\geq 1} \sum_{j=1}^d   h_n^{j,k} \partial_j \wt{p}_n \Big)_{k=1}^d, \\
(b_n\cdot \nabla) u&=\Big(\sum_{j=1}^d b^j_n \partial_j u^k \Big)_{k=1}^d,
&
\partial_{\hp}^2 \wt{p}&= \Big(\sum_{n\geq 1}\sum_{j=1}^d \partial_j (\hp_{n}^{j}\partial_{k}\wt{p}_n)\Big)_{k=1}^d, 
\end{aligned}
\end{equation}
and $\div(f(\cdot,u))=\big(\sum_{j=1}^d \partial_j (f^k_j(\cdot, u))\big)_{k=1}^d$. Here $a^{ij}, b^{j}_n, h^{j,k}_n,\hp_{n}^j$ are given functions. As explained near \eqref{eq:regularity_btwod} the $b^{j}_n$ are used to model turbulence. The $a^{ij}$ model a variable viscosity term and may also take into account It\^o's corrections in case of transport noise in Stratonovich form, see \eqref{eq:stratonovich_correction} and the discussion below it. Finally, $\partial_h \wt{p}$ and $\partial_{\hp}^2 \wt{p}$ model the effect of the turbulent part of the pressure $\wt{p}_n$
(see \cite[Example 1]{MiRo04} and again the discussion on the Stratonovich formulation near \eqref{eq:stratonovich_correction}). In case of It\^o noise one can take $\hp^j_n=0$, and in case of Stratonovich noise it is natural to take $\hp_n^j = -\frac12 b^j_n$. The additional forcing terms $\Fd_0(u)$ and $\div(f(\cdot,u))$ can be used to model the effect of gravity or the Coriolis force. For simplicity, we do not consider lower-order terms in the leading differential operators in \eqref{eq:Navier_Stokes_generalized}, see below Theorem \ref{t:maximal_Stokes_Tord} for some comments.

\medskip

This section is organized as follows.
In Subsections  \ref{ss:functional_analytic_setup} and \ref{ss:main_assumption_definition} we fix our notation and describe the spaces in which we consider \eqref{eq:Navier_Stokes_generalized}.
In Subsection \ref{ss:NS_statements} we present local existence and regularization phenomena for solutions to \eqref{eq:Navier_Stokes_generalized} with $u_0$ in critical spaces (see Theorem \ref{t:NS_critical_local}), in Subsection \ref{ss:Serrin_statement} we state a stochastic version of the Serrin's criteria and in Subsection \ref{ss:global_existence_NS} we state our global existence result for small initial data if $d\geq 3$, and for general rough initial data if $d=2$. The proofs will be given Section \ref{s:Navier_Stokes}.

Let us emphasise that even in the case $a^{ij} = \delta_{i,j}$, $f_0 = 0$, $f=0$, $h^{i,j}_n = 0$, $\hp_{n}=0$ and with either $b\neq 0$ or $g\neq 0$, our well-posedness and regularity results are new:
\begin{itemize}
\item The wide class of initial data we can treat has not been considered before.
\item New higher order regularity results are proved.
\end{itemize}
Finally, the global well-posedness result for $d=2$ is also new even with the above simplifications.

\subsection{The Helmholtz projection and function spaces}
\label{ss:functional_analytic_setup}
For any $s\in \R$, $p\in [1, \infty]$, $q\in (1,\infty)$ and integers $d,m\geq 1$, let $L^{q}(\Tor^d;\R^m)$, $H^{s,q}(\Tor^d;\R^m)$ and $B^s_{q,p}(\Tor^d;\R^m)$ denote the Lebesgue, Bessel-potential and Besov spaces on $\Tor^d$ with values in $\R^d$, respectively. The reader is referred to \cite[Chapter 3]{SchmTr} for details on periodic function spaces.

We denote by $\p$ the \emph{Helmholtz projection} which, for $f=(f^n)_{n=1}^d \in C^{\infty}(\Tor^d;\R^d)$ and $n\in \{1,\dots,d\}$, is given by $\p f=((\p f)^n)_{n=1}^d $ where
\begin{equation}
\label{eq:def_helmholtz_projection_Fourier}
(\widehat{\p f})^{n}(k):= \widehat{f^n}(k)-\sum_{j=1}^d \frac{k_j k_n}{|k|^2}  \widehat{f^j}(k),
\quad k\in \Z^{d}\setminus \{0\},
\quad (\widehat{\p f})^n(0):=\widehat{f^n}(0).
\end{equation}
Here, $\widehat{f}(k)$ denotes the $k$-th Fourier coefficient of $f$. Note that $\div \, (\p f)=0$ for all $f\in C^{\infty}(\Tor^d;\R^d)$, and by  duality $\p : \D'(\Tor^d;\R^d)\to \D'(\Tor^d;\R^d)$ where $\D(\Tor^d;\R^d)=C^{\infty}(\Tor^d;\R^d)$. Moreover, by standard Fourier multiplier theorems (see \cite[Theorem 5.7.11]{Analysis1}), $\p$ restricts uniquely to a bounded linear operator
\begin{equation}
\label{eq:boundedness_p_H_sq}
\p :H^{s,q}(\Tor^d;\R^d) \to H^{s,q}(\Tor^d;\R^d) \quad \text{ for }  s\in \R, \, q\in (1,\infty).
\end{equation}
%%%%%
For $(s,q)$ as in \eqref{eq:boundedness_p_H_sq} and $\mathcal{Y}\in \{L^q,H^{s,q},B^{s}_{q,p}\}$, we set
\begin{equation}
\label{eq:spaces_divergence_free}
\mathbb{Y}(\Tor^d):=\Big\{ f\in \mathcal{Y}(\Tor^d;\R^d)\,:\,\div\, f=0\text{ in }\D'(\Tor^d)\Big\} , \ \  \|f\|_{\mathbb{Y}(\Tor^d)}:=\|f\|_{\mathcal{Y}(\Tor^d;\R^d)}.
\end{equation}
Note that $\mathbb{Y}(\Tor^d)=\p(\mathcal{Y}(\Tor^d;\R^d))$ for $\mathcal{Y}\in \{L^q,H^{s,q},B^{s}_{q,p}\}$ and $s,q$ as above.
By standard interpolation theory (see \cite[Chapter 3.6]{SchmTr} and \cite[Theorem 1.2.4]{Tri95})
\begin{equation}
\label{eq:complex_real_interpolation}
\begin{aligned}
\mathbb{H}^{s,q}(\Tor^d)&= [\mathbb{H}^{s_0,q}(\Tor^d),\mathbb{H}^{s_1,q}(\Tor^d)]_{\theta},\\
\mathbb{B}^{s}_{q,p}(\Tor^d)&= (\mathbb{H}^{s_0,q}(\Tor^d),\mathbb{H}^{s_1,q}(\Tor^d))_{\theta,p},
\end{aligned}
\end{equation}
for all $s_0,s_1\in \R$, $\theta\in (0,1)$, $s:=(1-\theta)s_0+\theta s_1$ and $p,q\in (1,\infty)$.

\subsection{Main assumptions and definitions}
\label{ss:main_assumption_definition}
Here we collect the main assumptions and definitions. Further hypotheses will be given where needed.  Below $\Progress$ and $\Borel$ denotes the progressive and Borel $\sigma$-algebra, respectively.

\begin{assumption}
\label{ass:NS}
Let $d\geq 2$.
We say that \emph{Assumption \ref{ass:NS}$(p,\kappa,q,\s)$} holds if $\s\in (-1,0]$, and one of the following two is satisfied:
\begin{itemize}
\item $q\in [2,\infty)$, $p\in (2,\infty)$ and $\a\in [0,\frac{p}{2}-1)$;
\item $q=p=2$ and $\a=0$,
\end{itemize}
and the following conditions hold:
\begin{enumerate}[{\rm(1)}]
\item  For all $i,j\in \{1,\dots,d\}$ and $n\geq 1$, the mappings $a^{i,j}, b^j_n,h^{i,j}_n,\hp^j_{n}:\R_+\times \O\times \Tor^d\to \R$ are $\Progress\otimes \Borel(\Tor^d)$-measurable.
\item\label{it:max_reg_regularity_a_btwod} There exist $M>0$ and $\eta>\max\{d/\xi,-\delta\}$, with $\xi\in [2, \infty)$ such that, a.s.\ for all $t\in \R_+$ and $i,j\in \{1,\dots,d\}$,
\begin{align*}
&\|a^{i,j}(t,\cdot)\|_{H^{\eta,\xi}(\Tor^d)}
+
\|(\btwod_{n}^{j}(t,\cdot))_{n\geq 1}\|_{H^{\eta,\xi}(\Tor^d;\ell^2)}&\\
& \  +
\|(\hp^{j}_{n}(t,\cdot))_{n\geq 1}\|_{H^{\eta,\xi}(\Tor^d;\ell^2)}+
\|(h^{i,j}_n(t,\cdot))_{n\geq 1}\|_{H^{\eta,\xi}(\Tor^d;\ell^2)}
\leq M.
\end{align*}
\item\label{it:max_reg_measurability_a_btwod} There exists $\ellip>0$ such that, a.s.\ for all $t\in \R_+$, $\Upsilon=(\Upsilon_i)_{i=1}^d\in \R^d$ and $x\in \Tor^d$,
$$
\sum_{i,j=1}^d\Big[a^{i,j}(t,x)
-\frac{1}{2}\Big(
\sum_{n\geq 1}\btwod_n^i(t,x) \btwod_n^j(t,x)\Big)\Big]\Upsilon_i\Upsilon_j \geq \ellip |\Upsilon|^2.
$$
\item\label{it:NS_g_force_estimatesmeas} For all $j\in \{0,\dots,d\}$ and $n\geq 1$, the maps $\Fd_j,\Gforce_{n}:\R_+\times \O\times \Tor^d\times \R^d\to \R^d$ are $\Progress\otimes \Borel(\Tor^d\times \R^d)$-measurable.
\item\label{it:NS_g_force_estimates}  For all $j\in \{0,\dots,d\}$,
$
\Fd_j(\cdot,0)\in L^{\infty}(\R_+\times \O\times \Tor^d;\R^d)
$,
$
(\Gforce_n(\cdot,0))_{n\geq 1}\in L^{\infty}(\R_+\times\O\times \Tor^d; \ell^2)
$,
and a.s.\ for all $t\in \R_+$, $x\in \Tor^d$ and $y,y'\in \R^d$,
\begin{multline*}
 |\Fd_j(t,x,y)-\Fd_j(t,x,y')|
 +\|(\Gforce_n(t,x,y)-\Gforce_n(t,x,y'))_{n\geq 1}\|_{\ell^2}
\lesssim (1+|y|+|y'|)|y-y'|.
\end{multline*}
\end{enumerate}
\end{assumption}

The parameter $p$ will be used for $L^p$-integrability in time with weight $t^{\a}\, \dd t$. The parameter $q$ will be used for integrability in space with Sobolev smoothness $1+\delta$.
In the main results below we will always use $\s\in [-\frac{1}{2},0]$. However, in Section \ref{s:Navier_Stokes} we will also consider $\delta\in (-1,0]$.
By Sobolev embeddings and Assumption \ref{ass:NS}\eqref{it:max_reg_regularity_a_btwod} we have
\begin{equation}
\label{eq:L_infty_a_b}
\|a^{i,j}\|_{C^{\eta-\frac{d}{\xi}}(\Tor^d)}+ \|(b^j_n)_{n\geq 1}\|_{C^{\eta-\frac{d}{\xi}}(\Tor^d;\ell^2)}\lesssim_{\eta,\xi}
M.
\end{equation}
The same also holds for $(h^{j,k}_n)_{n\geq 1}$ and $(\hp^j_n)_{n\geq 1}$.
Below, we often write $b^j:=(b_n^j)_{n\geq1}$ and similar for $g^j,h^{i,j},\hp^j$ for convenience.

\begin{remark}\
\label{r:ass_NS}
\begin{itemize}
\item
If Assumption \ref{ass:NS} holds for $\delta=0$, then it also holds for some $\delta<0$.
\item The regularity assumption on $a^{i,j},b^j$ in \eqref{it:max_reg_regularity_a_btwod} also fit naturally in the case of stochastic Navier-Stokes equations in Stratonovich form, see \eqref{eq:stratonovich_correction} and the text below it. Indeed, if $b^j\in H^{\eta,\xi}(\Tor^d;\ell^2)$ uniformly in $(t,\om)$, then by \eqref{eq:L_infty_a_b} and \cite[Proposition 4.1(1)]{AV21_SMR_torus} the coefficients of the divergence operator in the It\^o correction satisfy $a^{i,j}_b=\frac{1}{2}\sum_{n\geq 1}b^i_nb^j_n\in H^{\eta,\xi}(\Tor^d)$ and $(\hp_{n}^j)_{n\geq 1}=(-\frac{1}{2}b_n^j)_{n\geq1 }\in H^{\eta,\xi}(\Tor^d;\ell^2)$  both with norm estimates uniform in $(t,\om)$.
\item
Assumption \ref{ass:NS} also covers the case of $b^j$ with $\gamma$-H\"{o}lder smoothness which is the natural one as explained around \eqref{eq:regularity_btwod}. Indeed,
$
C^{\gamma}(\Tor^d;\ell^2)\hookrightarrow H^{\eta,\xi}(\Tor^d;\ell^2)
$
for all $\gamma>\eta$ and $\xi\in (1,\infty)$. Note that typically $\gamma\in (0,1)$ is small in applications. A similar remark holds for $a^{i,j}$.
\item The quadratic growth of $f_j$ and $g_n$ shows that we can allow nonlinearities which are as strong as the convection term $\div(u\otimes u)$ (see Subsection \ref{ss:scaling}). 
Moreover, the proofs show that also nonlocal nonlinearities $f_j$ and $g_n$ can be used as long as the estimates of Lemma \ref{l:hyp_H_NS} hold. 
\end{itemize}
\end{remark}

To introduce the notion of solutions to \eqref{eq:Navier_Stokes_generalized} we rewrite it as a stochastic evolution equation for the unknown velocity field $u$.
To this end, to have a divergence free stochastic part we need
$$
\nabla \wt{p}_n=(I-\p)\big[ (b_n \cdot \nabla )u +\Gforce_n(\cdot,u)\big]
$$
where $I$ is the identity operator.
Therefore, letting $h^{\cdot,k}_n:=(h^{j,k}_n)_{j=1}^d$ and setting $\wt{f}:=(\wt{f}^{\, k})_{k=1}^d$,
\begin{align}\label{eq:wtfbg}
\wt{f}^{\, k} (\cdot,u)
&:= \sum_{n\geq 1} \Big[(I-\p)\big[(b_n \cdot \nabla )u +\Gforce_n(\cdot,u)\big]\Big]\cdot h^{\cdot,k}_n\\
\nonumber
& \ \ + \sum_{n\geq 1}\div \Big(\hp_{n} \big((I-\p)\big[(b_n \cdot \nabla )u +\Gforce_n(\cdot,u)\big]\big)^k\Big),
\end{align}
we have
$
\wt{f}(\cdot,u) = \partial_h \wt{p}+\partial_{\hp}^2 \wt{p}$.
As above, $(\cdot)^j$ denotes the $j$-th component.

Applying the Helmholtz projection to the first equation in \eqref{eq:Navier_Stokes_generalized}, at least formally, the system \eqref{eq:Navier_Stokes_generalized} is equivalent to
\begin{equation}
\label{eq:Navier_Stokes_generalized_with_projection}
\left\{
\begin{aligned}
\dd u &=\p\big[\div(a\cdot \nabla u)-\div(u\otimes u)+ f_0(\cdot,u) +\div(f(\cdot,u))+\wt{f}(\cdot,u) \big]\,\dd t \\
& \qquad +\sum_{n\geq 1}\p \big[(\btwod_n \cdot \nabla ) u +\Gforce_{n}(\cdot,u)\big]\, \dd {w}_t^n,
\\ u(\cdot,0)&=u_0.
\end{aligned}\right.
\end{equation}

We are in position to define solutions to \eqref{eq:Navier_Stokes_generalized}. We recall that the sequence of independent Brownian motions $(w^n)_{n\geq 1}$ induces uniquely an $\ell^2$-cylindrical Brownian motion $W_{\ell^2}$ (see e.g.\ \cite[Example 2.12]{AV19_QSEE_1}).

\begin{definition}
\label{def:sol_NS}
Let Assumption \ref{ass:NS}$(p,\kappa,q,\s)$ be satisfied.
\begin{enumerate}[{\rm(1)}]
\item Let $\sigma$ be a stopping time and $u:[0,\sigma)\times \O\to \Hs^{1+\s,q}(\Tor^d)$ be a stochastic process.
 $(u,\sigma)$ is called a \emph{local $(p,\kappa,q,\s)$-solution} to \eqref{eq:Navier_Stokes_generalized} if there exists a sequence of stopping times $(\sigma_j)_{j\geq 1}$ such that the following hold.
\begin{itemize}
\item $\sigma_j\leq \sigma$ a.s.\ for all $j\geq 1$ and $\lim_{j\to \infty} \sigma_j=\sigma$ a.s.;
\item For all $j\geq 1$, the process $
\one_{[0,\sigma_j]\times \O}u
$ is progressively measurable;
\item a.s.\ for all $j\geq 1$, we have $u\in L^p(0,\sigma_j,w_{\a};\Hs^{1+\s,q}(\Tor^d))$ and
\begin{equation}
\label{eq:integrability_condition}
\begin{aligned}
-\div(u\otimes u)+f_0(\cdot,u) +\div(f(\cdot,u))+\wt{f}(\cdot,u)&\in L^p(0,\sigma_j,w_{\a};H^{-1+\s,q}(\Tor^d;\R^d)),\\
\big(\Gforce_{n}(\cdot,u)\big)_{n\geq 1} &\in  L^p(0,\sigma_j,w_{\a};H^{\s,q}(\Tor^d;\ell^2(\R^d))).
\end{aligned}
\end{equation}
\item a.s.\ for all $j\geq 1$ the following equality holds for all $t\in [0,\sigma_j]$:
\begin{equation}
\label{eq:integral_equation_Navier_Stokes}
\begin{aligned}
u(t)-u_0& =  \int_0^t \p\big[\div(a\cdot \nabla u) -\div(u\otimes u) +  f_0(\cdot,u) +\div(f(\cdot,u))+\wt{f}(\cdot,u)\big]\,\dd s\\
& \qquad + \int_0^t \big(\one_{[0,\sigma_j]}  \p \big[(\btwod_n \cdot \nabla ) u +\Gforce_{n}(\cdot,u)\big]\big)_{n\geq 1}\,\dd  W_{\ell^2}(s).
\end{aligned}
\end{equation}
\end{itemize}
\item $(u,\sigma)$ is called a \emph{$(p,\kappa,q,\s)$-solution} to \eqref{eq:Navier_Stokes_generalized} if for any other local $(p,\kappa,q,\s)$-solution $(v,\tau)$ to \eqref{eq:Navier_Stokes_generalized} we have $\tau\leq \sigma$ a.s.\ and $u=v$ a.e.\ on $[0,\tau)\times \O$.
\item A $(p,\kappa,q,\s)$-solution $(u,\sigma)$ is called \emph{global} (in time) if $\sigma=\infty$ a.s. 
\end{enumerate}
\end{definition}

Note that $(p,q,\s,\a)$-solutions are \emph{unique} by definition.
By Assumption \ref{ass:NS} and \cite[Proposition 4.1]{AV21_SMR_torus}, a.s.\ for all $j\geq 1$,
\begin{align*}
\div(a\cdot \nabla u) &\in L^p(0,\sigma_j,w_{\a};H^{-1+\s,q}(\T^d;\R^d)),
\\ \big( (\btwod_n \cdot \nabla) u\big)_{n\geq 1}& \in  L^p(0,\sigma_j,w_{\a};H^{\s,q}(\T^d;\ell^2(\R^d))).
\end{align*}
Thus the deterministic integrals in \eqref{eq:integral_equation_Navier_Stokes} are well-defined as Bochner integrals with values in $H^{-1+\s,q}(\T^d;\R^d)$. The stochastic integral is well-defined as an $H^{\s,q}(\Tor^d;\R^d)$-valued stochastic integral (see e.g.\ \cite[Theorem 4.7]{NVW13} and use the identity \eqref{eq:gamma_identification_H} below).
In case of global $(p,\kappa,q,\s)$-solutions we simply write $u$ instead of $(u,\sigma)$.

Suppose that Assumption \ref{ass:NS} holds and $(p,\a,\s,q)$ satisfies \eqref{eq:critical_equation_NS_lemma_hypothesis_H}. If one has
\begin{equation}
\label{eq:regularity_local_solutions_j}
u\in \bigcap_{\theta\in [0,1/2)} H^{\theta,p}(0,\sigma_j,w_{\a};\Hs^{1+\s-2\theta,q}(\Tor^d)) \ \ \text{ a.s.\ for all }j\geq 1,
\end{equation}
then from Step 1 of Theorem \ref{t:NS_critical_localgeneralkappa} and \cite[Lemma 4.10 and 4.12]{AV19_QSEE_1}, one can see that \eqref{eq:integrability_condition} holds.

\subsection{Local existence and regularization}
\label{ss:NS_statements}

The following is our main local well-posedness result. It is formulated for the space of critical initial data. The general case is covered by Theorem \ref{t:NS_critical_localgeneralkappa}, where $\delta\in (-1, 1/2)$ is included as well.
Global well-posedness for small initial data will be discussed afterwards in Theorem \ref{t:global_small_data}, and the special case $d=2$ is considered in Theorem \ref{t:NS_initial_data_large_two_dimensional}.
\begin{theorem}[Local well-posedness]
\label{t:NS_critical_local}
Let Assumption \ref{ass:NS}$(p,\kappa,q,\s)$ be satisfied and suppose that one of the following conditions holds:
\begin{itemize}
\item $ \delta\in [-\frac{1}{2} ,0]$, $\frac{d}{2+\s}<q<\frac{d}{1+\s}$, $\frac{2}{p}+\frac{d}{q}\leq 2+\s$, and $\a=\a_{\crit}=-1+\frac{p}{2}\big(2+\s-\frac{d}{q}\big)$;
\item $\delta = \kappa=\kappa_{\crit} =0$ and $p=q=d=2$.
\end{itemize}
Then for all $u_0\in L^0_{\F_0}(\O;\Bs^{\frac{d}{q}-1}_{q,p}(\Tor^d))$,
\eqref{eq:Navier_Stokes_generalized} has a unique $(p,\a_{\crit},q,\s)$-solution $(u,\sigma)$ such that a.s.\ $\sigma>0$ and
\begin{equation}
\label{eq:NS_regularity_near_0}
u\in L^p_{\rm loc}([0,\sigma),w_{\a_{\crit}};\Hs^{1+\s,q}(\Tor^d))\cap
C([0,\sigma);\Bs^{\frac{d}{q}-1}_{q,p}(\Tor^d)).
\end{equation}
Moreover, $(u,\sigma)$ instantaneously regularizes in time and space: a.s.
\begin{align}
\label{eq:H_regularization_NS}
u & \in H^{\theta,r}_{\rm loc}(0,\sigma;\Hs^{1-2\theta,\zeta}(\Tor^d)) \ \
\text{ for all } \theta\in [0,1/2), \ r,\zeta\in (2,\infty),
\\
u& \in C^{\theta_1,\theta_2}_{\rm loc} ((0,\sigma)\times\Tor^d;\R^d) \ \ \text{ for all } \theta_1\in [0,1/2),\ \theta_2\in (0,1).\label{eq:C_regularization_NS}
\end{align}
\end{theorem}
In case $p>2$, one can check that by the conditions on $p$ and $q$ one has $\a_{\crit}\in [0,\frac{p}{2}-1)$.
The assertions \eqref{eq:H_regularization_NS}-\eqref{eq:C_regularization_NS} show an instantaneous regularization for solutions to \eqref{eq:Navier_Stokes_generalized}. By Sobolev embedding one can check that \eqref{eq:H_regularization_NS} implies \eqref{eq:C_regularization_NS}.
In particular, $u$ becomes \emph{almost} $C^1$ in space for rough initial data $u_0$, rough coefficient $a^{ij}$, $h^{ij}$, and rough nonlinearities $f_0, f$ and $g$. If $\eta>\max\{d/\xi,1/2\}$ in Assumption \ref{ass:NS}, then we can set $\delta=-1/2$, and therefore taking $q\uparrow 2d$ it suffices to know that for some $p>2$ and $\varepsilon>0$,
\[u_0\in L^0_{\F_0}(\Omega;\Bs^{\varepsilon-\frac{1}{2}}_{2d,p}(\Tor^d)).\]
It would be interesting to know whether this is optimal.

\begin{remark}\
\label{r:invariance}
\begin{enumerate}[{\rm(1)}]
\item
If $p>2$, then it follows from the proof of Theorem \ref{t:NS_critical_local} that the following weighted regularity holds up to $t=0$:
\begin{equation*}
u\in H^{\theta,p}_{\loc}([0,\sigma);w_{\a_{\crit}};\Hs^{1+\s,q}(\Tor^d)) \ \  \text{for all} \ \theta\in [0,1/2).
\end{equation*}
\item\label{it:local_continuity}
Similar to \cite[Proposition 2.9]{AV22_reaction_local}, the solution provided by  Theorem \ref{t:NS_critical_local} depends  continuously on the initial data on a small random time interval.
\item As in \cite[Proposition 3.5]{AV22_reaction_local}, compatibility with respect to the parameters $(p,\s,q)$ holds, i.e.\  different choices of $(p,\s,q)$ lead to the same solution in Theorem \ref{t:NS_critical_local}.
\item\label{it:Ld_data} ($\Ls^d$-data).
Setting $q=d$ it follows from $L^d(\Tor^d;\R^d)\hookrightarrow
B^0_{d,p}(\Tor^d;\R^d)$ for $p\in [d,\infty)$ that Theorem \ref{t:NS_critical_local} (with $\s<0$) also applies to initial data from the scaling invariant space $L^0_{\F_0}(\Omega;\Ls^d(\Tor^d))$, see Subsection \ref{ss:scaling}.
By Remark \ref{r:ass_NS}, it is enough to know that Assumption \ref{ass:NS}$(p,\a,d,0)$ holds where $\a=-1+\frac{p}{2}(1+\s)$ and $p\geq d$.
\end{enumerate}
\end{remark}

Next we consider \emph{higher order regularity}. We are particular interested in studying how the regularity of $b_n$ in the transport term $(b_n\cdot \nabla) u$ affects the regularity of $u$. One could expect that $u$ gains  one order of smoothness compared to $(b_n)_{n\geq 1}$, i.e.\ $\nabla u$ is as smooth as $(b_n)_{n\geq 1}$. The following result shows that this is indeed the case. Of course in order to prove such a result we also need further smoothness of the nonlinearities $f_0, f, g$. We emphasize that we do not impose further conditions on the initial data $u_0$.

The next assumption roughly says that $f_0$ and  $(f,g_n)$ are $C^{\lceil \eta\rceil}$ and $C^{\lceil \eta+1\rceil}$ in the $y$-variable, respectively; where $\eta>0$ is some fixed number (here $\lceil n \rceil:=n+1$ if $n$ is an integer).
\begin{assumption}
\label{ass:high_order_nonlinearities}
Let $\eta>0$, $f_0, f$ and $g_n$ be as in Assumption \ref{ass:NS}.
Suppose that $f_0,f$ and $g_n$ are $x$-independent. Moreover, assume that, for all $n\geq 1$ and a.e.\ on $\R_+\times \O$, the mappings  $f_0$ and $ (f,g_n)$ are $C^{\lceil \eta \rceil}$ and $C^{\lceil \eta+1 \rceil}$ in $y$, respectively; and for all $N\geq 1$ there exists $C_N>0$ such that, a.e.\ on $\R_+\times \O$, 
$$
\sum_{j=1}^{\lceil \eta\rceil } |\partial_y^j f_0(\cdot,y)|
+
\sum_{j=1}^{\lceil \eta+1\rceil }
\big(|\partial_y^j f(\cdot,y)|
+
\|(\partial_y^j g_n(\cdot,y))_{n\geq 1}\|_{\ell^2}\big)
\leq C_N \ \ \text{ for }\   |y|\leq N.
$$
\end{assumption}

\begin{theorem}[Higher order regularity]
\label{t:NS_high_order_regularity}
Let the assumptions of Theorem \ref{t:NS_critical_local} be satisfied, where
$(\eta, \xi)$ are such that Assumption \ref{ass:NS}\eqref{it:max_reg_regularity_a_btwod} holds, i.e. $\eta>\max\{d/\xi,-\delta\}$, $\xi\in [2, \infty)$, and there exists $M>0$ such that, for all $t\in \R_+$ and a.s., 
\begin{align*}
&\|a^{i,j}(t,\cdot)\|_{H^{\eta,\xi}(\Tor^d)}
+
\|(\btwod_{n}^{j}(t,\cdot))_{n\geq 1}\|_{H^{\eta,\xi}(\Tor^d;\ell^2)}&\\
& \  +
\|(\hp^{j}_{n}(t,\cdot))_{n\geq 1}\|_{H^{\eta,\xi}(\Tor^d;\ell^2)}+
\|(h^{i,j}_n(t,\cdot))_{n\geq 1}\|_{H^{\eta,\xi}(\Tor^d;\ell^2)}
\leq M.
\end{align*}
Furthermore, suppose that Assumption \ref{ass:high_order_nonlinearities} holds for such $\eta>0$.
Let $(u,\sigma)$ be the $(p,\a_{\crit},q,\s)$-solution to \eqref{eq:Navier_Stokes_generalized} provided by Theorem \ref{t:NS_critical_local}. Then, a.s.,
\begin{align}
\label{eq:H_regularization_NS_improved}
u&\in H^{\theta,r}_{\rm loc}(0,\sigma;\Hs^{1+\eta-2\theta,\xi}(\Tor^d))
&\text{for all } \theta\in [0,1/2), \, r\in (2,\infty),&
\\ u& \in  C^{\theta_1,\theta_2+\eta-\frac{d}{\xi}}_{\rm loc} ((0,\sigma)\times\Tor^d;\R^d)  &\text{for all }  \theta_1\in [0,1/2), \, \theta_2\in (0,1).&
\label{eq:C_regularization_NS_improved}
\end{align}
\end{theorem}
From the above theorem one can see how the regularity and integrability of order $\eta$ and $\xi$ of the coefficients appear in \eqref{eq:H_regularization_NS_improved} and \eqref{eq:C_regularization_NS_improved}. In particular, \eqref{eq:H_regularization_NS_improved} with $\theta=0$ shows that the regularity of $u$ is one order higher than the regularity of $(a,b,h,\hp)$.
Under additional smoothness conditions in space, one can also allow $x$-dependent $f_0$, $f$ and $g_n$ in the above.

\begin{remark}
If $u_0\in L^0_{\F_0}(\Omega;\Bs^{1+\eta - \frac{2}{r}}_{\xi,r}(\Tor^d))$ for some fixed $r\in (2,\infty)$, then one can check from the proofs that the regularity result \eqref{eq:H_regularization_NS_improved} (for the fixed $r$) holds locally on $[0,\sigma)$ instead of $(0,\sigma)$. However, this will not be used in the sequel.
\end{remark}

\subsection{Stochastic Serrin's blow-up criteria}
\label{ss:Serrin_statement}
In this section we state some blow-up criteria for \eqref{eq:Navier_Stokes_generalized} which can be seen as a stochastic analogue of the \textit{Serrin}'s criteria for Navier-Stokes equations  (see e.g.\ \cite[Theorem 11.2]{LePi}).

\begin{theorem}[Stochastic Serrin's criteria]
\label{t:NS_serrin}
Let the conditions of Theorem \ref{t:NS_critical_local} be satisfied and let $(u,\sigma)$ be the $(p,\a_{\crit},q,\s)$-solution to \eqref{eq:Navier_Stokes_generalized}.
Suppose that Assumption \ref{ass:NS}$(p_0,0,q_0,\delta_0)$ holds in one of the following cases:
\begin{itemize}
\item $\delta_0\in [-\frac{1}{2},0]$, $\frac{d}{2+\delta_0}<q_0<\frac{d}{1+\delta_0}$
and $\frac{2}{p_0}+\frac{d}{q_0}\leq 2+\delta_0$;
\item $\delta_0=0$ and $p_0=q_0=d=2$.
\end{itemize}
Then for each $0<\varepsilon<T<\infty$:
\begin{equation}
\label{eq:Serrin_NS_H}
\P\Big(\varepsilon<\sigma<T,\,\|u\|_{L^{p_0}(\varepsilon,\sigma;H^{\gamma_0,q_0}(\Tor^d;\R^d))}<\infty\Big) =0, \ \ \text{where} \ \ \gamma_0=\frac{2}{p_0}+\frac{d}{q_0}-1.
\end{equation}
In particular, for all $0<\varepsilon<T<\infty$ the following hold:
\begin{enumerate}[{\rm(1)}]
\item \label{it:blow_up_NS_critical_space_not_sharp_B}
$
\displaystyle{
\P\Big(\varepsilon<\sigma<T,\,\sup_{t\in [\varepsilon,\sigma)}\|u(t)\|_{B^{d/q_0-1}_{q_1,\infty}(\Tor^d;\R^d)}<\infty\Big)=0
}
$
if $q_1>q_0$;
\item\label{it:Serrin_NS_L} $\displaystyle \P\Big(\varepsilon<\sigma<T,\,\|u\|_{L^{p_0}(\varepsilon,\sigma;L^{q_0}(\Tor^d;\R^d))}<\infty\Big)=0$ if $p_0$ satisfies $\frac{2}{p_0}+\frac{d}{q_0}=1$.
\end{enumerate}
\end{theorem}
Item \eqref{it:Serrin_NS_L} is the stochastic version of the well--known Serrin's criterion. It is a special case of \eqref{eq:Serrin_NS_H} where we can even allow negative smoothness.
In particular, one can take $p_0=p$, $q_0=q$ and $\delta_0 = \delta$, but the above formulation is much more flexible as it allows us to transfer information on $\sigma$ between different settings. An interesting choice is $p_0=q_0=2$ and $\delta_0=0$ in case $d=2$, and will be used in Theorem \ref{t:NS_initial_data_large_two_dimensional} for $d=2$.
Note that by \eqref{eq:H_regularization_NS}-\eqref{eq:C_regularization_NS}, the paths of $u$ are regular enough on $(0,\sigma)$ to consider any of the above norms.

\begin{remark}\
\begin{enumerate}[{\rm (1)}]
\item Choosing $q_0,p_0$ large enough, one has $\frac{2}{p_0}+\frac{d}{q_0}<1$ (thus $\gamma_0<0$) and \eqref{eq:Serrin_NS_H} provides blow-up criteria in spaces of \emph{negative} smoothness. The same holds for  \eqref{it:blow_up_NS_critical_space_not_sharp_B} if $q_0>d$.
To see how far one can go note that for $\delta_0=-\frac{1}{2}$ and $q_0\to 2d$, $p_0\to \infty$ one would obtain $\frac{d}{q_0}-1\to -\frac{1}{2}$ in \eqref{it:blow_up_NS_critical_space_not_sharp_B} and
$\gamma_0\to -\frac{1}{2}$ in \eqref{eq:Serrin_NS_H}.
\item The \emph{space-time Sobolev index} (keeping in mind the parabolic scaling) of the space $L^{p_0}(0,t;H^{\gamma_0,q_0})$ appearing in \eqref{eq:Serrin_NS_H} and \eqref{it:Serrin_NS_L} is given by $-\frac{2}{p_0} +\gamma_0-\frac{d}{q_0}=-1$. The latter is equivalent to the Sobolev index of the space of initial data $\Bs^{d/q-1}_{q,p}(\Tor^d)$, which means that these blow-up criteria are optimal. In contrary, the space-time Sobolev index of $L^{\infty}(0,t;B^{d/q_0-1}_{q_1,p_0})$ in \eqref{it:blow_up_NS_critical_space_not_sharp_B} is given by $\frac{d}{q_0}-1-\frac{d}{q_1}>-1$ which is sub-optimal. We do not know if \eqref{it:blow_up_NS_critical_space_not_sharp_B} holds in the borderline case $q_1=q_0$.
\end{enumerate}
\end{remark}

In the deterministic setting, blow-up criteria in the critical space $L^3$ (here $d=3$ for simplicity) are known to be valid even with quantitative growth assumption in the $L^{\infty}(\varepsilon,\sigma;L^3)$-norm, see \cite[Theorem 1.4]{T_NS_blow_up}. It would be interesting to extend Theorem \ref{t:NS_serrin} into this direction by exploiting the structure of the equation \eqref{eq:Navier_Stokes_generalized}.

\subsection{Global well-posedness results}
\label{ss:global_existence_NS}

Global well-posedness for the deterministic Navier-Stokes equations is known for small initial data. The next result is its stochastic analogue.

\begin{theorem}[Global well-posedness for small and rough initial data]
\label{t:global_small_data}
Let Assumption \ref{ass:NS}$(p,\kappa,q,\s)$ be satisfied and suppose that one of the following conditions holds:
\begin{itemize}
\item $ \delta\in [-\frac{1}{2} ,0]$, $\frac{d}{2+\s}<q<\frac{d}{1+\s}$, $\frac{2}{p}+\frac{d}{q}\leq 2+\s$, and $\a=\a_{\crit}=-1+\frac{p}{2}\big(2+\s-\frac{d}{q}\big)$;
\item $\delta = \kappa=\kappa_{\crit} =0$ and $p=q=d=2$.
\end{itemize}
Assume that there are $M_{1}, M_2>0$ such that, a.s.\ for all $t\geq 0$, $x\in \Tor^d$ and $y\in \R^d$,
\begin{align}\label{eq:condfggrowtheq}
\sum_{j=0}^d|f_j(t,x,y)| + \|(g_n(t,x,y))_{n\geq 1}\|_{\ell^2} \leq M_1 +M_2(|y|+|y|^2).
\end{align}
Fix $u_0\in L^p_{\F_0}(\Omega;\Bs^{\frac{d}{q}-1}_{q,p}(\T^d))$ and
let $(u,\sigma)$ be the $(p,\a_{\crit},q,\s)$-solution to \eqref{eq:Navier_Stokes_generalized} provided by Theorem \ref{t:NS_critical_local}. For each $\varepsilon\in (0,1)$ and $T\in (0,\infty)$ there exists a constant $C_{\varepsilon,T}>0$, independent of $u_0$, such that the following assertions hold provided
\begin{equation*}
\E\|u_0\|_{B^{d/q-1}_{q,p}(\Tor^d;\R^d)}^p+M_1^p\leq C_{\varepsilon,T} .
\end{equation*}
\begin{enumerate}[{\rm(1)}]
\item $(u, \sigma)$ exists up to time $T$ with probability $>1-\varepsilon$, i.e.\ $\P(\sigma \geq T)>1-\varepsilon$.
\item There exists a stopping time $\tau\in (0, \sigma]$ a.s.\ such that $\P(\tau\geq  T)>1-\varepsilon$ and
\begin{align*}
\E\Big[\one_{\{\tau\geq  T\}} \|u \|_{H^{\theta,p}(0,T,w_{\a_{\crit}};H^{1+\s-2\theta,q}(\Tor^d;\R^d))}^p\Big]
\lesssim_{\theta} 
\E\|u_0\|_{B^{d/q-1}_{q,p}(\Tor^d;\R^d)}^p+M_1
\end{align*}
for all $\theta\in [0,\tfrac{1}{2})$,
where in case $p=q=d=2$ and $\theta>0$ the {\normalfont{LHS}} of the above estimate should be replaced by $\E\big[\one_{\{\tau\geq  T\}} \|u\|_{C([0,\tau];\Ls^2)}^2\big]$.
\end{enumerate}
\end{theorem}

Recall that $\Ls^d(\Tor^d)\embed \Bs^{0}_{d,p}(\Tor^d)$ for all $p\geq d$, see Remark \ref{r:invariance}\eqref{it:Ld_data}. Thus Theorem \ref{t:global_small_data} also proves existence for large times with high probability under a smallness assumption on $u_0$ in the scaling invariant Lebesgue space $\Ls^d(\Tor^d)$ provided $M_1$ is small as well.

Next we turn our attention to the two-dimensional setting in which one can avoid smallness of the initial data. A variation of a classical result for $\Ls^2$-initial data can be found in Theorem \ref{t:2D_global_rough_noise} in the appendix, where the results are partly based on \cite{AV19_QSEE_1}. The main new contributions in the result below are:
\begin{itemize}
\item We obtain global existence for initial data from  $\Bs^{2/q-1}_{q,p}(\Tor^2)$ for suitable $p\in [2, \infty)$ and $q\in [2, 4)$. By Sobolev embedding the latter space contains $\Ls^2(\Tor^2)$. By the restriction $q<4$ below, we see that we can allow smoothness of negative order $-\frac12+\varepsilon$ for any $\varepsilon>0$.
\item We prove arbitrary spatial regularity, and time regularity of order $C^{\frac{1}{2}-\varepsilon}$ for any $\varepsilon>0$, where the initial data is still assumed to be irregular.
\end{itemize}

\begin{theorem}[Global well-posedness in two dimensions with rough initial data]
\label{t:NS_initial_data_large_two_dimensional}
Let $d=2$. Suppose that Assumption \ref{ass:NS}$(p,\kappa,q,\s)$ holds
and that for all $n\geq 1 $, $ j\in \{1,\dots,d\}$
\begin{equation}
\label{eq:theta_equal_b_Lp_setting}
 \hp^j_n=\alpha_n b_n^j \ \  \text{ for some } \   \alpha_n\in [- \tfrac{1}{2},\infty) .
\end{equation}
Assume that one of the following conditions holds:
\begin{itemize}
\item $ \delta\in [-\frac{1}{2} ,0]$, $\frac{2}{2+\s}<q<\frac{2}{1+\s}$, $\frac{2}{p}+\frac{2}{q}\leq 2+\s$, and $\a=\a_{\crit}=-1+\frac{p}{2}\big(2+\s-\frac{2}{q}\big)$;
\item $\delta = \kappa =\kappa_{\crit}=0$ and $p=q=2$.
\end{itemize}
Assume that $\Fd$ and $\Gforce$ additionally satisfy the following linear growth assumption: There exists a $C>0$ such that,
a.s.\ for all $t\geq 0$, $ x\in \T^2$ and $y\in \R^2$,
\begin{equation}
\label{eq:theta_equal_b_Lp_setting2}
\sum_{j=0}^2 |\Fd_j(t,x,y)|+\|(\Gforce_n(t,x,y))_{n\geq 1}\|_{\ell^2}\leq C(1+|y|).
\end{equation}
Then for all $u_0\in L^0_{\F_0}(\O;\Bs^{\frac{2}{q}-1}_{q,p}(\Tor^2))$ there exists a unique \emph{global} $(p,q,\s,\a_{\crit})$-solution $u$ of \eqref{eq:Navier_Stokes_generalized} with
\begin{equation*}
u\in L^p_{\rm loc}([0,\infty),w_{\a_{\crit}};\Hs^{1+\s,q}(\Tor^2))\cap
C([0,\infty);\Bs^{\frac{2}{q}-1}_{q,p}(\Tor^2)) \text{ a.s.\ }
\end{equation*}
Moreover, a.s., 
\begin{align*}
u & \in H^{\theta,r}_{\rm loc}(0,\infty;\Hs^{1-2\theta,\zeta}(\Tor^2)) \ \
\text{ for all } \theta\in [0,1/2), \ r,\zeta\in (2,\infty),
\\
u& \in C^{\theta_1,\theta_2}_{\rm loc} ((0,\infty)\times\Tor^2;\R^2) \ \ \text{ for all } \theta_1\in [0,1/2),\ \theta_2\in (0,1).
\end{align*}
Furthermore, if the conditions of Theorem \ref{t:NS_high_order_regularity} hold for $d=2$, then $u$ satisfies \eqref{eq:H_regularization_NS_improved} and \eqref{eq:C_regularization_NS_improved} with $\sigma = \infty$.
\end{theorem}

In particular, note that if $a^{ij},b^{j},h^{ij}, c_n^j, f_j$, and $g_n$ satisfy the conditions of Theorem \ref{t:NS_high_order_regularity}  for \emph{all} $\eta>0$ and $\xi\in [2, \infty)$, then a.s.\ the paths of $u$ have regularity up to any order in space:
\[u\in C^{\theta, \infty}_{\rm loc}((0,\infty)\times\T^2;\R^2) \ \text{ a.s.\ for all} \ \theta\in [0,1/2).\]
The conditions \eqref{eq:theta_equal_b_Lp_setting} and \eqref{eq:theta_equal_b_Lp_setting2} are used for global well-posedness in the $L^2$-setting, see Appendix \ref{s:global_L_2_rough_noise}. The choice $\alpha_n= -\frac{1}{2}$ for all $n\geq 1$ corresponds to transport noise in Stratonovich form, see \eqref{eq:stratonovich_correction}.

The above regularity results of solutions to \eqref{eq:Navier_Stokes_generalized} appears to be new. It is not possible to prove this by classical bootstrapping arguments based on $L^2$-theory since the nonlinearity $\div(u\otimes u)$ is critical in this setting. Indeed, after one proves existence of a \emph{global} solution $u$ for which the paths satisfy a.s.\ $u\in L^2(0,T;H^{1,2})\cap C([0,T];L^2)$, one is tempted to prove further regularity by analyzing the term $\div(u\otimes u)$. Note that the following embeddings are sharp in dimension two
\[L^2(0,T;H^{1,2})\cap C([0,T];L^2)\hookrightarrow L^4(0,T;H^{1/2,2})\hookrightarrow L^4(0,T;L^4).\]
Therefore, we can only conclude $\div(u\otimes u)\in L^2(0,T;H^{-1,2})$ a.s.\ which is not good enough to bootstrap regularity. In order to prove the desired smoothness we will apply our recent bootstrapping method from \cite[Sections 6.2 and 6.3]{AV19_QSEE_2} which is based on arguments involving weights in time.

We conclude by remarking that in general, one can not improve the starting point $u\in L^2(0,T;H^{1,2})\cap C([0,T];L^2)$ a.s.\ of the above argument by using energy methods. Indeed, as remarked in Section \ref{s:intro}, our regularity assumptions on $a, b,h,c$, and $u_0$ do \emph{not} allow us to reformulate \eqref{eq:Navier_Stokes_generalized} in terms of the \emph{vorticity}, i.e.\ for the unknown $\mathrm{curl}\, u$. Thus besides the usual energy estimate in the space $L^2(0,T;H^1)\cap C([0,T];L^2)$, no other estimates are available for solutions to \eqref{eq:Navier_Stokes_generalized} in general.

\section{Maximal $L^p$-regularity for the turbulent Stokes couple}
\label{s:max_reg_linear}
Before we turn to the main results on the stochastic Navier-Stokes equations  with transport noise \eqref{eq:Navier_Stokes_generalized}, we analyse the linear part of the problem, which is central in our approach. In order to be as flexible as possible in later applications we obtain $L^p(s,T,w_{\a}^s;H^{1+\delta,q})$-regularity estimates for any $p,q\geq 2$, $\a\in [0,\frac{p}{2}-1)$ and $\delta\in \R$. We consider the following \emph{turbulent Stokes system} on $\Tor^d$:
\begin{equation}
\label{eq:NS_Tor_linearized}
\left\{
\begin{aligned}
\dd u +\p\big(\A  u+ \pp u\big)\, \dd t&=f \, \dd t+ \sum_{n\geq 1}\big(\p\b_n u+  g_n\big)\,\dd w_t^n,
\\
u(s)&=0.
\end{aligned}\right.
\end{equation}
Here $s\in [0,\infty)$, $(w^n)_{n\geq 1}$ is a sequence of independent standard Brownian motion on $(\O,\mathscr{A},(\F_t)_{t\geq 0},\P)$, $\p$ denotes the Helmholtz projection, and for $v=(v^k)_{k=1}^d$,
\begin{equation}
\begin{aligned}
\label{eq:NS_differential_operators_concise_form}
\A(t)v&=\big(-\div(a(t,\cdot)\cdot \nabla v^k) \big)_{k=1}^d  ,\\
 \b_n(t)v&= \big((\btwod_n(t,\cdot)\cdot \nabla )v\big)_{k=1}^d,\\
\pp(t)v&=
\Big(-\sum_{n\geq 1}\big\{ \div(\hp_{n} [(I-\p)\bb_n(t) v]^k)
+ [(I-\p)\b_n (t)v]\cdot h^{\cdot,k}_n\big\}
\Big)_{k=1}^d, \\
 \end{aligned}
\end{equation}
where $a, b,\hp,h$ are given and $[\cdot]^k$ denotes the $k$-th component.

The results below also cover the case where in \eqref{eq:NS_Tor_linearized} one additional considers non-trivial initial data and lower order terms in the leading differential operators; see the comments below Theorem \ref{t:maximal_Stokes_Tord}.

\subsection{Assumptions and main results}\label{ss:linearassmain}
In this section we employ the following conditions.

\begin{assumption}
\label{ass:max_reg_parameters}
Let $d\geq 2$, $\delta\in \R$ and $0\leq s<T<\infty$.  Suppose that one of the following holds
\begin{itemize}
\item $q\in [2,\infty)$, $p\in (2,\infty)$ and $\a\in \big[0,\frac{p}{2}-1\big)$;
\item $q=p=2$ and $\a=0$,
\end{itemize}
and that the following conditions hold:
\begin{enumerate}[{\rm(1)}]
\item\label{it:max_reg_measurability_SMR} For all $i,j\in \{1,\dots,d\}$ and $n\geq 1$, the mappings $a^{i,j},b^j_n,\hp_n^j,h^{i,j}_n:\R_+\times \O\times \Tor^d\to \R$ are $\Progress\otimes \Borel(\Tor^d)$-measurable.
\item\label{it:max_reg_regularity_a_btwod_SMR} Let $\xi\in [q',\infty)$ and $\eta>\frac{d}{q}$ be such that one of the following hold:
\begin{itemize}
\item$\s\leq 0$ and $\eta>-\s$;
%%%
\item $\s>0$, $\eta\geq \s$, $\xi\leq q$ and  $\eta-\frac{d}{\xi}\geq \s-\frac{d}{q}$.
\end{itemize}
There exists $M>0$ such that, a.s.\ for all $t\in \R_+$ and $i,j\in \{1, \ldots, d\}$,
\begin{align*}
&\|a^{i,j}(t,\cdot)\|_{H^{\eta,\xi}(\Tor^d)}
+
\|(\btwod_{n}^{j}(t,\cdot))_{n\geq 1}\|_{H^{\eta,\xi}(\Tor^d;\ell^2)}&\\
& \  +
\|(\hp^{j}_{n}(t,\cdot))_{n\geq 1}\|_{H^{\eta,\xi}(\Tor^d;\ell^2)}+
\|(h^{i,j}_n(t,\cdot))_{n\geq 1}\|_{H^{\eta,\xi}(\Tor^d;\ell^2)}
\leq M.
\end{align*}
\item\label{it:max_reg_measurability_a_btwod_Smr} There exists $\ellip>0$ such that, a.s.\ for all $t\in \R_+$, $\Upsilon=(\Upsilon_i)_{i=1}^d\in \R^d$ and $x\in \Tor^d$,
$$
\sum_{i,j=1}^d\Big[a^{i,j}(t,x)
-\frac{1}{2}\Big(
\sum_{n\geq 1}\btwod_n^i(t,x) \btwod_n^j(t,x)\Big)\Big]\Upsilon_i\Upsilon_j \geq \ellip |\Upsilon|^2.
$$
\end{enumerate}
\end{assumption}

By Sobolev embeddings and
Assumption \ref{ass:max_reg_parameters}\eqref{it:max_reg_regularity_a_btwod_SMR}, for all $i,j\in \{1,\dots,d\}$,
\begin{equation}
\label{eq:Holder_continuity_phi_h}
a^{i,j}\in C^{\alpha}(\Tor^d),
\quad \text{ and }\quad
(b_n^j)_{n\geq 1},\,(h_n^{i,j})_{n\geq 1}, \, (c_n^j)_{n\geq 1} \in C^{\alpha}(\Tor^d;\ell^2),
\end{equation}
where $\alpha:=\eta-\frac{d}{\xi}>0$. Note that
Assumption \ref{ass:max_reg_parameters} coincides with part of Assumption \ref{ass:NS}$(p,\a,q,\s)$ in the case $\s\in (-1,0]$, and of course the nonlinear parts of Assumption \ref{ass:NS}\eqref{it:NS_g_force_estimatesmeas}-\eqref{it:NS_g_force_estimates} do not play a role here. The case $\s>0$ will be used to prove the higher order regularity result of Theorem \ref{t:NS_high_order_regularity}.

Under the previous assumptions, we consider the \textit{turbulent Stokes couple}:
\begin{equation}
\begin{aligned}
\label{eq:generalizaed_Stokes_couple_divergence}
(A_{\Stok},B_{\Stok})&:\R_+\times \O\to \calL\big(\Hs^{1+\s,q}(\Tor^d),
\Hs^{-1+\s,q}(\Tor^d)\times \Hs^{\s,q}(\Tor^d;\ell^2)\big)\\
(A_{\Stok}(\cdot) ,B_{\Stok}(\cdot))u &:=\big(\p[\A(\cdot)u+ \pp(\cdot)u], (\p\b_n(\cdot) u)_{n\geq 1}\big),\qquad
u\in \Hs^{1+\s,q},
\end{aligned}
\end{equation}
where $\Hs^{s,q}$ is as in \eqref{eq:spaces_divergence_free} with $\mathcal{Y}=H^{s,q}$. By Assumption \ref{ass:max_reg_parameters}\eqref{it:max_reg_regularity_a_btwod_SMR}, a.e.\ on $\R_+\times \O$,
\begin{align}
\label{eq:estimate_AB_stok}
&\|A_{\Stok}(\cdot)v\|_{\Hs^{-1+\s,q}(\Tor^d)}
+
\|B_{\Stok}(\cdot)v\|_{\Hs^{\s,q}(\Tor^d;\ell^2)}\\
\nonumber
& \stackrel{\eqref{eq:boundedness_p_H_sq}}{\lesssim} \max_{i,j,k}\Big(
\|a^{i,j}\partial_j v^k\|_{H^{\s,q}(\Tor^d)}
+\Big\|\sum_{n\geq 1}\big[(I-\p)[(b_n\cdot \nabla) v]\big]^k h^{j,k}_{n}\Big\|_{H^{-1+\s,q}(\Tor^d)}\\
\nonumber
&
\quad +\Big\|\sum_{n\geq 1} \hp^j_n [(I-\p) ((b_n\cdot\nabla )v) ]^k \Big\|_{H^{\s,q}(\Tor^d)}+
\|(b^{j}_n(t,\cdot)\partial_j v^k)_{n\geq 1}\|_{H^{\s,q}(\Tor^d;\ell^2)}
\Big)\\
\nonumber
&\lesssim_{\s,q} M \|v\|_{\Hs^{1+\s,q}(\Tor^d)},
\end{align}
where we used \cite[Corollary 4.2 and Proposition 4.1(3)]{AV21_SMR_torus} if $\s> 0$ and $\s\leq 0$, respectively.
By Assumption \ref{ass:max_reg_parameters}\eqref{it:max_reg_measurability_SMR}, the mapping $(A_{\Stok},B_{\Stok})$ defined in \eqref{eq:generalizaed_Stokes_couple_divergence} is strongly progressively measurable.

For any stopping time $\tau:\O\to [s,T]$, $f\in L^0_{\Progress}(\O;L^1(s,\tau;\Hs^{-1+\s,q}(\Tor^d)))$, $g\in L^0_{\Progress}(\O;L^2(s,\tau;\Hs^{-1+\s,q}(\Tor^d;\ell^2)))$ we say that $u\in L^0_{\Progress}( \O;L^2(s,\tau;\Hs^{1+\s,q}))$ is a \emph{strong solution to \eqref{eq:NS_Tor_linearized} (on $[s,\tau]\times \O$)} if a.s.\ for all $t\in [s,\tau]$
$$
u(t)+\int_{s}^t A_{\Stok}(r) u(r)\, \dd r =\int_s^t f(r)\, \dd r+\int_{s}^t (B_{\Stok} (r)u(r)+g(r))\, \dd W_{\ell^2}(r),
$$
where $W_{\ell^2}$ is the $\ell^2$-cylindrical Brownian motion associated with $(w^n)_{n\geq 1}$, see \cite[Example 2.12]{AV19_QSEE_1}. Note that the stochastic integral is well-defined in $\Hs^{\s,q}(\Tor^d)$ due to \eqref{eq:estimate_AB_stok}, \cite{Ondrejat04} or \cite[Theorem 4.7]{NVW13} and $q\geq 2$.

The aim of this section is to prove the following.

\begin{theorem}[Stochastic maximal $L^p$-regularity for the turbulent Stokes system]
\label{t:maximal_Stokes_Tord}
Suppose that Assumption \ref{ass:max_reg_parameters} holds. Then for each $s\in [0,T)$, and each progressively measurable $f\in  L^p((s ,T)\times \O,w_{\a}^{s};\Hs^{-1+\delta,q})$, and
$g\in L^p((s ,T)\times \O,w_{\a}^s;\Hs^{\delta,q}(\ell^2))$, there exists a unique strong solution $u$ to \eqref{eq:NS_Tor_linearized} on $[s,T]\times \O$. Moreover, letting
\begin{align*}
  J_{p,q,\kappa}(f,g):= & \|f\|_{L^p((s ,T)\times \O,w_{\a}^{s};\Hs^{-1+\delta,q})}+\|g\|_{L^p((s ,T)\times \O,w_{\a}^{s};\Hs^{\delta,q}(\ell^2))},
\end{align*}
the following results hold, where the constants $C_i$ do not depend on $(s,f,g)$:
\begin{enumerate}[{\rm(1)}]
\item in case $p\in (2, \infty)$, $q\in [2, \infty)$, $\theta\in [0,\frac12)$,
\begin{align*}
\|u\|_{L^p((s ,T)\times \O;w_{\a}^{s}; \Hs^{1+\delta,q})} &\leq C_1 J_{p,q,\kappa}(f,g),
\\ \|u\|_{L^p(\O;\hz^{\theta,p}(s,T,w_{\a}^{s};\Hs^{1-2\theta+\delta,q}))} &\leq C_{2}(\theta) J_{p,q,\kappa}(f,g), \ \ \theta\in [0,1/2),
\\ \|u\|_{L^p(\O;C([s,T];\Bs^{1+\delta-2(1+\kappa)/p}_{q,p}))}
&\leq C_3 J_{p,q,\kappa}(f,g),
\\ \|u\|_{L^p(\O;C([s+\varepsilon,T];\Bs^{1+\delta-2/p}_{q,p}))}
&\leq C_{4}(\varepsilon) J_{p,q,\kappa}(f, g), \ \ \varepsilon\in (s,T).
\end{align*}
\item in case $p=q=2$ (and thus $\a=0$),
\begin{align*}
\|u\|_{L^2(\O;L^2(s,T;\Hs^{1+\delta,2}))} &\leq C_5 J_{2,2,0}(f,g),
\\ \|u\|_{L^2(\O;C([s,T];\Hs^{\delta,2}))} &\leq C_6 J_{2,2,0}(f,g).
\end{align*}
\end{enumerate}
\end{theorem}

By either \cite[Proposition 3.10]{AV19_QSEE_1} or \cite[Proposition 3.9]{AV19_QSEE_2}, the above maximal $L^p$-regularity result for \eqref{eq:NS_Tor_linearized} extends to the case of non-trivial initial data provided $u(s)\in L^p_{\F_s}(\O;\Bs^{1+\s-2\frac{1+\a}{p}}_{q,p})$ and the corresponding norm is taken into account in $J_{p,q,\a}$.
Similarly, by a general perturbation argument, we may allow lower order terms in the differential operators appearing in \eqref{eq:NS_Tor_linearized} by using \cite[Theorem 3.2]{AV21_SMR_torus} provided the coefficients of such lower order terms are sufficiently regular.

Theorem \ref{t:maximal_Stokes_Tord} can be seen as an extension of \cite{M02}, where the $\R^d$-case with $q=p$ was considered; see also Remark \ref{r:Rd_case} below. The proof of Theorem \ref{t:maximal_Stokes_Tord} uses similar ideas as in \cite[Section 5]{AV21_SMR_torus} in which we combined perturbation and localization arguments to obtain stochastic maximal regularity for parabolic problems on $\Tor^d$. In the present case, additional difficulties arise due the nonlocal nature of the Helmholtz projection $\p$.
The $\R^d$-case can also be covered with the same method if one imposes the following condition:

\begin{remark}[The $\R^d$-case]
\label{r:Rd_case}
Theorem \ref{t:maximal_Stokes_Tord} also holds for \eqref{eq:NS_Tor_linearized} with $\Tor^d$ replaced by $\R^d$ if the coefficients become constants as $|x|\to \infty$: There exists progressively measurable maps $\wh{a}^{i,j}:[s,T]\times \O\to \R$, $(\wh{b}^{j}_n)_{n\geq 1}:[s,T]\times \O\to \ell^2$ such that
\begin{align*}
\lim_{|x|\to \infty}\esssup_{\om\in \O} \sup_{r\in [s,T]}\Big( |a^{i,j}(r,\om,x)-\wh{a}^{i,j}(r,\om)|
+ \|(b^{j}_n(r,\om,x)-\wh{b}^{j}_n(r,\om))_{n\geq 1}\|_{\ell^2}\Big)=0,
\end{align*}
which is needed in our localization argument (see Steps 3 and 4 of the proof of Lemma \ref{l:a_priori_estimates_NS}).
For instance, if $a^{i,j}=\delta^{i,j}$ and $(b_n^j)_{n\geq 1}\in H^{\eta,\xi}(\R^d;\ell^2)$ for some $\eta>0$ and $\xi\in [q',\infty)$ such that $\eta>\frac{d}{\xi}$, then the above condition is satisfied with $\wh{b}^j_n\equiv 0$.
\end{remark}

Our methods can also be extended to the non-divergence case setting. For this one needs to modify the assumptions on $a^{i,j}$ in Assumption \ref{ass:max_reg_parameters}\eqref{it:max_reg_regularity_a_btwod_SMR}. A detailed statement for related non-divergence systems can be found in \cite[Section 5]{AV21_SMR_torus}.

Theorem \ref{t:maximal_Stokes_Tord} will be proved in Subsection \ref{ss:proof_smr_weak_setting} below. The core of the proof is the a priori estimate on small time intervals of Lemma \ref{l:a_priori_estimates_NS} which will be proven in Subsection \ref{ss:proof_lemma_small_time_interval}.

\subsection{Proof of Theorem \ref{t:maximal_Stokes_Tord}}
\label{ss:proof_smr_weak_setting}

In the rest of this section, for $s\in \R$ and $q\in (1,\infty)$, we write $L^q$, $H^{s,q}$ and $B^{s}_{q,p}$ instead of $L^q(\T^d;\R^d)$, $H^{s,q}(\T^d;\R^d)$ and $B^{s}_{q,p}(\T^d;\R^d)$.

Recall that $\p$ denotes the Helmholtz projection, see \eqref{eq:def_helmholtz_projection_Fourier}. For future convenience, we recall another construction of $\p$. Let $f\in \D'(\Tor^d;\R^d)$. Consider the following elliptic problem on $\T^d$:
\begin{equation}
\label{eq:Helmholtz_projection_equation_phi}
\left\{\begin{aligned}
\Delta \psi&=\div f,\\
\l 1 , \psi\r &=0,
\end{aligned}\right.
\end{equation}
where $\l \cdot , \cdot \r$ denotes the duality of $\D(\Tor^d)$  and $\D'(\Tor^d)$.
If we set $\q f:= \psi$, then the Helmholtz projection can be equivalently defined by
\begin{equation}
\label{eq:def_Helmholtz_proposition}
\p f:=f-\nabla\q f.
\end{equation}
Indeed, by standard Fourier methods, one can check that \eqref{eq:def_helmholtz_projection_Fourier} and \eqref{eq:def_Helmholtz_proposition} coincide. Moreover $\nabla\q$ is a projection (i.e.\ $\nabla\q=(\nabla\q)^2$) and $\q$ restricts to a map
\begin{equation}
\label{eq:NS_Tor_mapping_property_p}
 \q: H^{s,q}\to H^{s+1,q}.
\end{equation}

Below we collect some lemmata which will be useful in the proof of the main results. The following concerns an operator appearing several times in the proofs.

\begin{lemma}
\label{l:J_operator}
Let
$s\in \R$, $q\in (1,\infty)$ and $\phi\in C^{\infty}(\Tor^d)$. For any $f\in H^{s,q}(\Tor^d;\R^d)$, we set $\J_{\phi} f:=\psi$, where $\psi\in \D'(\Tor^d;\R^d)$ is the unique solution to the following elliptic problem on $\T^d$:
\begin{equation}
\label{eq:PDE_defining_J}
\left\{\begin{aligned}
\Delta \psi&=\nabla \phi\cdot f- \lb \nabla \phi, f\rb,\\
\lb 1,\psi\rb&=0.
\end{aligned}\right.
\end{equation}
Then
$
\|\J_{\phi} f\|_{H^{s+2,q}}\lesssim_{s,q} \|f\|_{H^{s,q}}.
$
\end{lemma}

\begin{proof}
Let us denote by $\Delta_N^{-1}$ the operator defined by $\widehat{(\Delta_N^{-1} f)}(0)=0$ and $\widehat{(\Delta_N^{-1} f)}(k)=\frac{1}{|k|^2} \widehat{f}(k)$ for $\Z^d \ni k\neq 0$. By standard Fourier techniques, one can show that for each $s\in \R$ and $q\in (1,\infty)$,
\begin{equation}
\label{eq:Delta_R_mapping}
\Delta_N^{-1}: \{f\in H^{s,q}:\widehat{f}(0)=0\} \to \{f\in H^{s+2,q}:\widehat{f}(0)=0\}.
\end{equation}
The claim follows by noticing that
$\J_{\phi}f=\Delta_N^{-1}(\nabla \phi\cdot f-\l \nabla \phi,f\r)$.
\end{proof}

Next we consider an estimate for a commutator operator appearing in the main localization argument.

\begin{lemma}[A commutator estimate]
\label{l:commutator_localization}
Let
$s\in \R$, $q\in (1,\infty)$ and $\phi\in C^{\infty}(\Tor^d)$. Let $[A,B]:=AB-BA$ be the commutator.
Then
$$
\big\| [\nabla\q,\phi] f \big\|_{H^{s+1,q}}\lesssim \|f\|_{H^{s,q}}.
$$
\end{lemma}
In the above and in the following, we do not distinguish between the function $\phi$ and multiplication operator $v\mapsto \phi v$.

\begin{proof}
Employing the notation used in the proof of Lemma \ref{l:J_operator}, the construction \eqref{eq:Helmholtz_projection_equation_phi}-\eqref{eq:NS_Tor_mapping_property_p} shows that
$$
\nabla \q f= \nabla \Delta^{-1}_N \div f , \ \ \text{ for } f \in \D'(\Tor^d;\R^{d}).
$$
Let $M$ the projection onto mean zero vector fields, i.e.\ $Mf :=f-\l 1,f\r$ for $f \in \D'(\Tor^d;\R^{d})$.
Now note that $\nabla \q = \nabla \Delta^{-1}_N M \div$ and
\begin{equation}
\label{eq:commutator_equality_q_proof}
[\nabla \q,\phi] = [\nabla, \phi] \Delta^{-1}_N M \div + \nabla [\Delta_N^{-1} M ,\phi ]\div + \nabla \Delta_{N}^{-1}M [\div,\phi].
\end{equation}
It is easy to see that $[\nabla, \phi] $ and $[\div,\phi]$ map continuously $H^{s,q}$ into itself. To conclude it remains to estimate the second term on the RHS\eqref{eq:commutator_equality_q_proof}. In turn it suffices to show that
\begin{equation}
\label{eq:commutator_DeltaNM_phi}
\big\| [\Delta_N^{-1} M ,\phi ] f\big\|_{H^{r,q}}\lesssim \|f\|_{H^{r-3,q}}, \quad \text{ for all }r\in \R.
\end{equation}
Note that  $ [\Delta_N^{-1} M ,\phi ] f= v_{\phi}-\phi v$ where $v:= \Delta_N^{-1} M f$, $v_{\phi}:= \Delta^{-1}_N M (\phi f)$ and
\begin{equation}
\label{eq:equation_v_u_commutator}
\Delta (\phi v- v_{\phi})= (\Delta \phi) v + 2\nabla \phi \cdot\nabla v - \phi\l 1,f \r + \l \phi, f\r.
\end{equation}
Hence \eqref{eq:commutator_DeltaNM_phi} follows from:
\begin{align*}
\big\| [\Delta_N^{-1} M ,\phi ] f\big\|_{H^{r,q}}
\lesssim
\| v\|_{H^{r-1,q}} + \|f\|_{H^{r-3,q}}
\lesssim
\|f\|_{H^{r-3,q}},
\end{align*}
where we used elliptic regularity twice and \eqref{eq:equation_v_u_commutator} in the first estimate.
\end{proof}

The following simple extension result is taken from \cite[Lemma 5.8 and Remark 5.6]{AV21_SMR_torus}.
\begin{lemma}[Extension operator]
\label{l:extension_operators}
Let $\alpha\in (0,N]$ and $\zeta\in (1,\infty)$ be such that $\alpha>\frac{d}{\zeta}$. Let $X$ be a separable Hilbert space. Then, for any $y\in \Tor^d$ and any $r\in (0,\frac{1}{8})$ there exists an extension operator
\[
\e_{y,r}^{\Tor^d}:H^{\alpha,\zeta}(\B_{\Tor^d}(y,r);X)\to H^{\alpha,\zeta}(\Tor^d;X)
\]
which satisfies the following properties:
\begin{enumerate}[{\rm(1)}]
\item\label{it:extension_constant} $\e_{y,r}^{\Tor^d} f|_{\B_{\Tor^d}(y,r)}=f$, $\e_{y,r}^{\Tor^d} c\equiv c$ for any $f\in H^{\alpha,\zeta}(\B_{\Tor^d}(y,r);X)$ and $c\in X$;
\item\label{it:extension_bounds_C_alpha} $\|\e_{y,r}^{\Tor^d}\|_{\calL(H^{\alpha,\zeta}(\B_{\Tor^d}(y,r);X),H^{\alpha,\zeta}(\Tor^d;X))}\leq C_r$ for some $C_r$ independent of $y$;
\item\label{it:extension_bounds_L_infty} $\|\e_{y,r}^{\Tor^d}\|_{\calL(C(\overline{\B_{\Tor^d}(y,r)};X),C(\Tor^d;X))}\leq C$ for some $C$ independent of $y,r$.
\end{enumerate}
%%%
\end{lemma}
The above lemma also holds if $X$ is a Banach space with the UMD property. The reader is referred to \cite[Chapter 4]{Analysis1} for details.

The key step in the proof of Theorem \ref{t:maximal_Stokes_Tord} is the following a priori estimate for solutions to \eqref{eq:NS_Tor_linearized} on small time intervals.

\begin{lemma}
\label{l:a_priori_estimates_NS}
Let Assumption \ref{ass:max_reg_parameters} be satisfied.
Then there exist $\tT,C>0$ depending only on the parameters $q,p,\a,d,\reg,\ellip,\eta,\xi,M$ such that for any $t\in [s,T)$, $\wtwo\in \{0,\a\}$, any stopping time $\tau:\O\to [t,t(t+\tT)\wedge T]$, any
\begin{equation}
\label{eq:fg_t_star}
f\in L^p_{\Progress}(( t,\tau)\times \O,w_{\wtwo}^t;\Hs^{-1+\s,q}), \ \
\ \ g\in L^p_{\Progress}(( t,\tau)\times \O, w_{\wtwo}^t;\Hs^{\s,q}(\ell^2))
\end{equation}
and any strong solution $u\in L^p_{\Progress}(( t,\tau)\times \O,w_{\wtwo}^t;\Hs^{1+\s,q})$ to \eqref{eq:NS_Tor_linearized} on $[ t,\tau]\times \O$ one has
\begin{equation*}
\|u\|_{L^p(( t,\tau)\times \O,w_{\wtwo}^t;\Hs^{1+\s,q})}
\leq
C \|f\|_{L^p(( t,\tau)\times \O,w_{\wtwo}^t;\Hs^{-1+\s,q})}+
C\|g\|_{L^p(( t,\tau)\times \O,w_{\wtwo}^t;\Hs^{\s,q}(\ell^2))}.
\end{equation*}
\end{lemma}

We first show how Lemma \ref{l:a_priori_estimates_NS} implies Theorem \ref{t:maximal_Stokes_Tord}.

\begin{proof}[Proof of Theorem \ref{t:maximal_Stokes_Tord}]
Below we employ the notation introduced in \cite[Definition 2.3]{AV21_SMR_torus} for the sets of couples of operators having the stochastic maximal $L^p$-regularity, i.e.\ $\mathcal{SMR}_{p,\a}(s,T)$ and $\mathcal{SMR}_{p,\a}^{\bullet}(s,T)$. Here we are using $X_j=\Hs^{-1+2j +\s,q}(\Tor^d)$. Also note that
$\Xap:=(X_0,X_1)_{1-\frac{1+\a}{p},p}=\Bs^{1+\s-2\frac{1+\a}{p}}_{q,p}$ and $X_{\theta}:=[X_0,X_1]_{\theta}=\Hs^{1+\s-2\theta,q}$ for $\theta\in (0,1)$.
Since $\Delta$ and $\p$ commute on $\Hs^{1+\s,q}$ (see \eqref{eq:def_helmholtz_projection_Fourier}), the Stokes operator $\Stok:=-\Delta: \Hs^{s+2,q}\subseteq \Hs^{s,q}\to \Hs^{s,q},$  is well-defined, and by the periodic version of \cite[Theorem 10.2.25]{Analysis2} $1+\Stok$ has a bounded $H^\infty$-calculus of angle zero. Therefore, by \cite[Theorem 7.16]{AV19} we have
\begin{equation}
\label{eq:Stokes_maximal_regularity}
\Stok\in \MRtas,\quad \text{ for all } 0\leq s<T<\infty
\end{equation}
with constant only depending on $q,\a,p,d$ and $T$.

The assertions of Theorem \ref{t:maximal_Stokes_Tord} follow if we can prove that $(A_{\Stok},B_{\Stok})\in \mathcal{SMR}_{p,\wtwo}^{\bullet}(t,t+\tT)$ (see \cite[Proposition 3.9]{AV19_QSEE_2}).
By \cite[Proposition 3.1]{AV21_SMR_torus}, it is enough to show that $(A_{\Stok},B_{\Stok})\in \mathcal{SMR}_{p,\wtwo}(t,t+\tT)$ for all $\wtwo\in \{0,\a\}$, $t\in [0,T-\tT]$ where $\tT$ is as in Lemma \ref{l:a_priori_estimates_NS}.
%%%%
To prove the latter, we use the method of continuity (see \cite[Proposition 3.13]{AV19_QSEE_2}).

For any $\lambda\in [0,1]$ and $v\in \Hs^{1+\s,q}$, we set
\begin{align*}
{A}_{\Stok,\lambda}v&:=(1-\lambda)\Stok v+ \lambda A_{\Stok}v
\\ & =-\p\big(\div(a_{\lambda}\cdot \nabla v) \big) \\
&-\Big(\sum_{n\geq 1}\big\{ \div(\lambda\hp_{n} \big[(I-\p)[(b_n\cdot \nabla) v]\big]^k) 
+ \big[(I-\p)[(b_n\cdot\nabla)v ]\big]\cdot \lambda h^{\cdot,k}_n\big\}
\Big)_{k=1}^d,
\\
{B}_{\Stok,\lambda}v&:=\lambda B_{\Stok}v
= \big(\p(\lambda(\btwod_n\cdot \nabla) v)\big)_{n\geq 1}
\end{align*}
where $a^{i,j}_{\lambda}=\lambda\delta^{i,j} +(1-\lambda)a^{i,j}$ for all $i,j\in \{1,\dots,d\}$.

Note that $({A}_{\Stok,0},{B}_{\Stok,0})=(\Stok,0)$, and $({A}_{\Stok,\lambda},{B}_{\Stok,\lambda})$ satisfies the Assumption \ref{ass:max_reg_parameters} uniformly w.r.t.\ $\lambda\in [0,1]$. Thus Lemma \ref{l:a_priori_estimates_NS} ensures that the a priori estimate \eqref{l:a_priori_estimates_NS} holds for strong solutions to \eqref{eq:NS_Tor_linearized} on $[ t,\tau]\times \O$ with $C$ independent of $\lambda\in [0,1]$. The claim of this step follows from the method of continuity in \cite[Proposition 3.13]{AV19_QSEE_2} and \eqref{eq:Stokes_maximal_regularity}.
%%%%
%%%%
\end{proof}

\subsection{Proof of Lemma \ref{l:a_priori_estimates_NS}}
\label{ss:proof_lemma_small_time_interval}
We will next prove the key estimate. In principle we use the perturbation method of our recent work \cite[Lemma 5.4]{AV21_SMR_torus}, but due to the non-local behavior of the Helmholtz projection, there are several complications in the proof.
\begin{proof}[Proof of Lemma \ref{l:a_priori_estimates_NS} in case $h=0$]
The proof will be divided into several steps.
Moreover, for exposition convenience, we set $s=0$ and in Steps 1-6 we only consider the case $\s\leq 0$ and in Step 7 we discuss the modifications needed for the case $\s>0$.

We will use the method where we freeze the coefficients and use comparison on a small ball. For any $y\in\Tor^d$, $r\in (0,\frac{1}{8})$, $\B(y,r):=\B_{\Tor^d}(y,r)$, and for $v\in \Hs^{1+\reg,q}$, we set
\begin{equation}
\label{eq:differential_operator_frozen_extended}
\begin{aligned}
\A_{y}(t)v &:=-\div(a(t,y)\cdot \nabla v),
&
\A_{y,r}^{\e}(t)v&:=-\div( a^{\e}_{y,r} (t,\cdot)\cdot \nabla v),
\\
\b_{y,n}(t)v&:=(\btwod_n(t,y) \cdot \nabla ) v,
&  \b_{y,r,n}^{\e}(t)v&:=(\btwod^{\e}_{n,y,r}(t,\cdot)\cdot \nabla) v,\\
\b_{y}(t)v&:=(\b_{y,n}(t)v)_{n\geq 1},
&  \b_{y,r}^{\e}(t)v&:=(\b_{y,r,n}^{\e}(t) v)_{n\geq 1},
\end{aligned}
\end{equation}
where
$$
a^{\e}_{y,r}:=\Big(\e_{y,r}^{\Tor^d}(a^{i,j}(t,\cdot))\Big)_{i,j=1}^d, \quad
\btwod^{\e}_{n,y,r}:=\Big(\e_{y,r}^{\Tor^d}(\btwod^{j}_n(t,\cdot))\Big)_{j=1}^d, \ \ \  n\geq 1,
$$
and $\e_{y,r}^{\Tor^d}$ is the extension operator of Lemma \ref{l:extension_operators}. The operators $\A_{y}$ and $\b_{y}$ have a ``frozen coefficient at $y\in \Tor^d$" and $\A_{y,r}^{\e}$, $\b_{y,r}^{\e}$ are the operators whose coefficients are the extensions of $a^{i,j}|_{\B(y,r)} $, $\btwod^j_n|_{\B(y,r)}$.
Similarly, for $(y,r,v)$ as above, we set
\begin{equation}
\label{eq:differential_operator_frozen_extended_pp}
\pp_{y,r}^{\e} (t)v:= -\Big(\sum_{n\geq 1} \div( \hp_n (t,\cdot) \big(\nabla \q [ \b_{y,r,n}^{\e}(t) v])^k \big)\Big)_{k=1}^d,
\end{equation}
where $\q$ is as defined below \eqref{eq:Helmholtz_projection_equation_phi}.
Note that the coefficients $\hp_n$ in \eqref{eq:differential_operator_frozen_extended_pp} have \emph{not} been changed. Finally, there is no need to define the ``frozen'' operator $\pp_y (t)$ since
\begin{equation}
\label{eq:triviality_of_p_yn}
 \q  [\b_{y,n}(t)v]= (b_n(t,y)\cdot \nabla)[\q v]=0,
\ \ \text{
for all }v\in \Hs^{1+\s,q}, \ n\geq 1, \ t\in \R_+.
\end{equation}

Let $t\in [s,T)$, $\wtwo\in \{0,\a\}$ and let $\tau:\O\to [t,t^*]$ be a stopping time with $t^*:= T \wedge(t+\tT)$ where $\tT>0$ will be chosen in Step 6.
We use same notation as in \cite[Subsection 3.3]{AV19_QSEE_1} in case $(\p\A,\p\b):=(\p \A,(\p \b_n)_{n\geq 1})\in \mathcal{SMR}_{p,\wtwo}(t,T)$ with $X_j=\Hs^{-1+2j+\s,q}$ where $j\in\{0,1\}$. The solution operator $\Sol_{t,(\p\A,\p\b)}$ associated to the couple $(\p \A,\p \b)$ (see \eqref{eq:NS_differential_operators_concise_form}) maps
\begin{align}
\label{eq:mapping_solution_operators}
L^p_{\Progress}(( t,T)\times \O,w_{\wtwo}^{t};\Hs^{-1+\s,q})
&\times L^p_{\Progress}(( t,T)\times \O,w_{\wtwo}^{t};\Hs^{\s,q}(\ell^2))
\\
\nonumber
&\to L^p_{\Progress}(( t,T)\times \O ,w_{\wtwo}^t;\Hs^{1+\s,q}),
\end{align}
where 
$\Sol_{t,(\p\A,\p\b)}(f,g):=u$ 
and $u$ is the unique strong to \eqref{eq:NS_Tor_linearized} on $[t,\tau]\times \O$, and $(f,g)$ are as in \eqref{eq:fg_t_star}. In addition, we denote by $
K^{p,\wtwo}_{(\p\A,\p\b)}(t,T)$ its operator norm:
$$
K^{p,\wtwo}_{(\p\A,\p\b)}(t,T)=\|\Sol_{t,(\p\A,\p \b)}(\cdot,\cdot)\|
$$
on the spaces indicated in \eqref{eq:mapping_solution_operators}.
For $(f,g)$ as in \eqref{eq:fg_t_star}, and for $\tau$ as above, we set
\begin{equation}
\label{eq:N_f_g_stokes}
N_{f,g}(t,\tau):=\|f\|_{L^p(( t,\tau)\times \O,w_{\wtwo}^t;\Hs^{-1+\s,q})}
+\|g\|_{L^p(( t,\tau)\times \O,w_{\wtwo}^t;\Hs^{\s,q}(\ell^2))}.
\end{equation}
To abbreviate the dependencies let
$\Set =\{q,p,\a,d,\reg,\ellip,\eta,\xi,M\}$, and whenever we wish to emphasize such dependencies we write $C(\Set)$ instead of $C$.

\textit{Step 1: There exists $C_1(\Set)>0$ such that for each $y\in \Tor^d$, one has $(\p\A_{y},\p\b_{y})\in \mathcal{SMR}_{p,\wtwo}(t,T)$ and}
$
K^{p,\wtwo}_{(\p\A_y,\p\b_y)}(t,T)
\leq C_1.
$

Due to \eqref{eq:Helmholtz_projection_equation_phi}-\eqref{eq:def_Helmholtz_proposition}, one can check that
$$
\p\A_{y} v=\A_y \p v,\quad \p\b_{y} v=\b_{y}\p v, \ \ \text{ a.e.\ on }[ t,T]\times \O  \text{ for all } v\in  H^{1+\s,q}.
$$
Therefore the result is immediate from \cite[Lemma 5.4]{AV21_SMR_torus}.

\textit{Step 2: There exists $\smallpar(\Set)>0$ for which the following holds: }

\textit{If $y\in \Tor^d$ and $r\in (0,\frac{1}{8})$ are such that a.s.\ for all $t\in (0,T)$, $i,j\in \{1,\dots,d\}$, }
\begin{equation}
\label{eq:Step3_smallness_condition_eta}
\|a^{i,j}(t,\cdot)-a^{i,j}(t,y)\|_{L^{\infty}(\B(y,r))}
+\|(\btwod^j_n(t,\cdot)-\btwod^j_n (t,y))_{n\geq 1}\|_{L^{\infty}(\B(y,r);\ell^2)}
\leq \smallpar,
\end{equation}
\textit{then $(\p\A_{y,r}^{\e},\p\b_{y,r}^{\e})\in\mathcal{SMR}_{p,\wtwo}(t,T)$ and}
$
K^{p,\wtwo}_{(\p\A_{y,r}^{\e},\p\b_{y,r}^{\e})}(t,T)
\leq C_2(\Set).
$

To prove this we apply the perturbation result of \cite[Theorem 3.2]{AV21_SMR_torus}. To this end, we write
\begin{equation}
\label{eq:NS_Tor_perturbation_0}
\p\A_{y,r}^{\e}=\p \A_{y} +\p(\A_{y,r}^{\e}-\A_{y}),\qquad
\p\b_{y,r}^{\e}=\p \b_{y} +\p(\b_{y,r}^{\e}-\b_{y}).
\end{equation}
By \eqref{eq:boundedness_p_H_sq}, for each $v\in \Hs^{1+\s,q}$,
\begin{equation*}
\|\p(\A_{y,r}^{\e}-\A_{y}) v\|_{\Hs^{-1+\s,q}}
\lesssim \sum_{i,j=1}^d \Big\|(a^{i,j}(t,y)-\e_{y,r}^{\Tor^d}(a^{i,j}(t,\cdot))\partial_j v\Big\|_{H^{\s,q}}.
\end{equation*}
Next we estimate each term separately:
\begin{align*}
&\Big\|(a^{i,j}(t,y)-\e_{y,r}^{\Tor^d}(a^{i,j}(t,\cdot))\partial_j v\Big\|_{H^{\s,q}}\\
&\qquad \qquad
\stackrel{(i)}{=}\Big\|\e_{y,r}^{\Tor^d}\big(a^{i,j}(t,y)-a^{i,j}(t,\cdot)\big)\partial_j v\Big\|_{H^{\s,q}}\\
&\qquad \qquad
\stackrel{(ii)}{\lesssim}\Big(
\Big\|\e_{y,r}^{\Tor^d}\big(a^{i,j}(t,y)-a^{i,j}(t,\cdot)\big)\Big\|_{L^{\infty}}
\|\partial_j v\|_{H^{\s,q}}\\
&\qquad \qquad \qquad\qquad\qquad \quad
+\Big\|\e_{y,r}^{\Tor^d}\big(a^{i,j}(t,y)-a^{i,j}(t,\cdot)\big)\Big\|_{H^{\eta,\xi}}
\|\partial_j v\|_{H^{\s-\varepsilon,q}}\Big)\\
&\qquad \qquad
\stackrel{(iii)}{\leq }
\smallpar \|v\|_{H^{1+\s,q}}+C_r M\|v\|_{H^{1+\s-\varepsilon, q}}\\
&\qquad \qquad
 \stackrel{(iv)}{\leq } 2 \smallpar \|v\|_{H^{1+\s,q}}+C_{r,M,\smallpar}\|v\|_{H^{-1+\s, q}},
\end{align*}
where in $(i)$ we used Lemma \ref{l:extension_operators}\eqref{it:extension_constant}, in $(ii)$ \cite[Proposition 4.1(3)]{AV21_SMR_torus} for some $\varepsilon>0$ and in $(iii)$ Lemma \ref{l:extension_operators}\eqref{it:extension_bounds_C_alpha}-\eqref{it:extension_bounds_L_infty} and \eqref{eq:Step3_smallness_condition_eta}. Finally, in $(iv)$ we used a standard interpolation inequality.

For the $\b$-term, recall that $\xi\geq q'$ and $\eta>-\s$. Thus employing similar arguments, by Assumption \ref{ass:NS}\eqref{it:max_reg_regularity_a_btwod} and \cite[Proposition 4.1(3)]{AV21_SMR_torus} with $H=\ell^2$, one obtains
\begin{equation*}
\|\p(\b_{y,r}^{\e}-\b_{y}) v\|_{\Hs^{\s,q}(\ell^2)}\leq 2 \smallpar \|v\|_{H^{1+\s,q}}
+C_{M,r,\smallpar}\|v\|_{H^{-1+\s, q}}.
\end{equation*}
Similarly, by \eqref{eq:triviality_of_p_yn} and a variation of \cite[Proposition 4.1(3)]{AV21_SMR_torus},
\begin{align}
\label{eq:estimate_pp_new_term_localization}
\|\pp_{y,r}^{\e} v\|_{\Hs^{-1+\s,q}}
&\lesssim_{\s,q}\max_{j,k}
\Big\|\sum_{n\geq 1} \partial_j\Big( \hp_n^j  \big(\nabla \q [ (\b_{y,r,n}^{\e}-\b_y) v])^k \Big)
\Big\|_{H^{1+\s,q}}\\
\nonumber
&\lesssim \big\|\big((\b_{y,r}^{\e}-\b_y) v\big)_{n\geq 1}\big\|_{H^{\s,q}(\ell^2)}\\
\nonumber
&\leq
2 \smallpar \|v\|_{H^{1+\s,q}}
+C_{M,r,\smallpar}\|v\|_{H^{-1+\s, q}}.
\end{align}
The claim of Step 2 follows from Step 1, \cite[Theorem 3.2]{AV21_SMR_torus}, the above estimates and the arbitrariness of $\ell\in \{0,\a\}$.

\textit{Step 3: Let $\smallpar$ be as in step 2. There exist an integer $\Lambda\geq 1$, $(y_{\lambda})_{\lambda=1}^{\Lambda}\subseteq \Tor^d$, $(r_{\lambda})_{\lambda=1}^{\Lambda}\subseteq (0,\frac{1}{8})$, depending only on the quantities in $\Set$, such that $\Tor^d \subseteq \cup_{\lambda=1}^{\Lambda} \B_{\lambda}$, where $\B_{\lambda}:=\B_{\Tor^d}(y_{\lambda},r_{\lambda})$, and a.s.\ for all $t\in (0,T)$, $i,j\in \{1,\dots,d\}$,}
$$
\|a^{i,j}(t,y_{\lambda})-a^{i,j}(t,\cdot)\|_{L^{\infty}(\B_{\lambda})}
+
\|(\btwod^j_n(t,y_{\lambda})-\btwod^{j}_n(t,\cdot))_{n\geq 1}\|_{L^{\infty}(\B_{\lambda};\ell^2)}
\leq \smallpar.
$$
\textit{In particular, for all $\lambda\in \{1,\dots,\Lambda\}$},
\begin{align*}
\big(\p[\A_{\lambda}^{\e}+\pp^{\e}_{\lambda}],\p\b_{\lambda}^{\e}\big)\in \mathcal{SMR}_{p,\wtwo}(t,T), \ \  \text{ where}&\\ (\A_{\lambda}^{\e},\pp^{\e}_{\lambda},\b_{\lambda}^{\e}):=(\A^{\e}_{x_{\lambda},r_{\lambda}},\pp^{\e}_{x_{\lambda},r_{\lambda}},\b^{\e}_{x_{\lambda},r_{\lambda}}),&
\end{align*}
{\em with $K^{p,\wtwo}_{(\p\A_{\lambda}^{\e},\p\b_{\lambda}^{\e})}(t,T)\leq C_3(\Set)$.}

The last claim follows from the first one and Step 2. Let $\alpha:=\eta-\frac{d}{\xi}>0$ due Assumption \ref{ass:max_reg_parameters}\eqref{it:max_reg_measurability_SMR}. By Sobolev embeddings, $H^{\eta,\xi}(H)\embed C^{\alpha}(H)$ where $H\in \{\R;\ell^2\}$. We denote by $R_0$ the embedding constant. Thus for any  $y\in \Tor^d$,
\begin{align*}
&\|a^{i,j}(t,y)-a^{i,j}(t,\cdot)\|_{L^{\infty}(\B(y,r))}
+
\|\btwod^j(t,y)-\btwod^{j}(t,\cdot)\|_{L^{\infty}(\B(y,r);\ell^2)}\\
&\qquad
\leq \big( [a^{i,j}(t,\cdot)]_{C^{\alpha}(\B(y,r))}+[\btwod^j(t,\cdot)]_{C^{\alpha}(\B(y,r);\ell^2)}\big) r^{\alpha}
\leq \wt{C}M \,r^{\alpha}\leq \smallpar,
\end{align*}
where the last inequality follows by choosing $r:=\min \big\{\big(\frac{\smallpar}{R_0 M}\big)^{1/\alpha},\frac{1}{8}\big\}$. Since $\smallpar=\smallpar(\Set)$, it follows that  $r=r(\Set)$. To conclude, it remains to note that $\Tor^d$ can be covered by finitely many balls of the form $\B(y,r)$.

\textit{Step 4: Let $(\partition_{\lambda})_{\lambda=1}^{\Lambda}$ be a smooth partition of the unity subordinate to the covering $(\B_\lambda)_{\lambda=1}^{\Lambda}$ (see Step 3). Let $\Sol_{\lambda}:=\Sol_{t,(\p\A_{\lambda}^{\e},\p\b_{\lambda}^{\e})} $ be the solution operator associated to $(\p\A_{\lambda}^{\e},\p\b_{\lambda}^{\e})\in \mathcal{SMR}_{p,\wtwo}(t,T)$.
Recall that $f,g$ are as in \eqref{eq:fg_t_star}, $\tau$ is a stopping time with values in $[t,(t+\tT)\wedge T]$ and $u\in L^p_{\Progress}(( t,\tau)\times \O,w_{\wtwo}^t;\Hs^{1+\s,q})$ is a strong solution to \eqref{eq:NS_Tor_linearized} on $[ t,\tau]\times \O$. Then for any $\lambda\in \{1,\dots,\Lambda\}$ the following holds}
\begin{equation*}
\p(\partition_{\lambda} u)=\Sol_{\lambda}(0,\Ff_{\lambda} u,\G_{\lambda} u)+
\Sol_{\lambda}(0,\p f_{\lambda},\p g_{\lambda}),\quad \text{ a.e. on }[ t,\tau]\times \O,
\end{equation*}
\textit{where $\J_{\lambda}:=\J_{\partition_{\lambda}}$ (see Lemma \ref{l:J_operator}), $g_{\lambda}:=(g_{\lambda,n})_{n\geq 1}:=(g_n \partition_{\lambda})_{n\geq 1}$, $f_{\lambda}:=\partition_{\lambda} f$ and}
\begin{equation}
\begin{aligned}
\label{eq:max_reg_step5_definition_F_G}
\Ff_{\lambda} u&:=-\p[(\A_{\lambda}^{\e}+ \pp_{\lambda}^{\e})\nabla \J_{\lambda} u] -\p [\partition_{\lambda},\A]u
-\p [\partition_{\lambda},\pp ]u
\\&\qquad   -\p[(\nabla \partition_{\lambda})\q ([\A +\pp] u) ], \\
\G_{\lambda,n}u &:=\p\b_{\lambda,n}^{\e}\nabla\J_{\lambda} u +\p[\partition_{\lambda},\b_{n}]u+
\p[(\nabla \partition_{\lambda})\q( \b_{n} u)],\\
\G_{\lambda} u&:=(\G_{\lambda,n} u)_{n\geq 1},
\end{aligned}
\end{equation}
\textit{with $\q$ as in \eqref{eq:def_Helmholtz_proposition}, and $[\cdot,\cdot]$ denotes the commutator}.

To begin, set $u_{\lambda}:=\partition_{\lambda} u$. Note that $\p[(\A+\pp) u]=(\A +\pp) u -\nabla \q [(\A+\pp) u]$ and $\p[\b_{n} u]=\b_{n} u -\nabla \q[\b_{n} u]$. Using these identities and multiplying  \eqref{eq:NS_Tor_linearized} by $\partition_{\lambda}$, one obtains on $\Tor^d$
\begin{equation}
\label{eq:identity_u_h_lcaolized_equation}
\begin{aligned}
& \dd u_{\lambda} +(\A+\pp) \,u_{\lambda}\,\dd r\\
&=\Big([\A,\partition_{\lambda}]u +[\pp,\partition_{\lambda}]u + \nabla (\partition_{\lambda} \q [(\A+\pp) u])-(\nabla \partition_{\lambda})\q [(\A +\pp)u]+ f_{\lambda}\Big)\,\dd r\\
&\quad
+\sum_{n\geq 1}\Big(\b_{n} u_{\lambda} +[\partition_{\lambda},\b_{n}]u -\nabla (\partition_{\lambda} \Q[\b_n u])
+(\nabla \partition_{\lambda})\q [\b_{n} u]  +  g_{\lambda,n}\Big)\, \dd w_r^n.
\end{aligned}
\end{equation}
Since $\supp(u_{\lambda})\subseteq \B_{\lambda}$, for each $\lambda \in \{1,\dots,\Lambda\}$ and $n\geq 1$ one has (see \eqref{eq:differential_operator_frozen_extended}-\eqref{eq:differential_operator_frozen_extended_pp})
\begin{align*}
\A \,u_{\lambda}=\A^{\e}_{x_{\lambda},r_{\lambda}}u_\lambda&=\A_{\lambda}^{\e} u_{\lambda},\qquad
 \b_{n}u_{\lambda}=\b^{\e}_{n,x_{\lambda},r_{\lambda}} u_{\lambda}=\b_{\lambda,n}^{\e} u_{\lambda},\\
 &\pp u_{\lambda} = \pp^{\e}_{x_{\lambda},r_{\lambda}} u_{\lambda}  = \pp^{\e}_{\lambda} u_{\lambda} ,
\end{align*}
where we used the notation of Step 3. The Helmholtz decomposition gives
\begin{equation}
\label{eq:decomposition_u_lambda}
u_{\lambda}=\pi_{\lambda} u =\nabla \q ( \pi_{\lambda} u)+\p ( \pi_{\lambda} u)
\stackrel{(i)}{=}\nabla \J_{\lambda} u+\p u_{\lambda}.
\end{equation}
Here in $(i)$ we used $\div \,u=0$ in $\D'(\Tor^d)$, \eqref{eq:PDE_defining_J} and
$$
\div \,u_{\lambda}=\div\, u_{\lambda} -\lb 1, \div \,u_{\lambda} \rb
=\nabla \partition_{\lambda} \cdot u- \lb 1, \nabla \partition_{\lambda} \cdot u\rb.
$$
Thus $u_{\lambda}=\nabla \J_{\lambda} u+v_{\lambda}$ with $v_{\lambda}:=\p u_{\lambda}$. Applying the operator $\p$ to \eqref{eq:identity_u_h_lcaolized_equation}, we get
\begin{align*}
& \dd v_{\lambda} +\p[(\A_{\lambda}^{\e}+\pp_{\lambda}^{\e}) v_{\lambda}]\,\dd r\\
& =\Big(-\p[(\A_{\lambda}^{\e}+\pp_{\lambda}^{\e})\nabla \J_{\lambda} u] +\p [\A,\partition_{\lambda}]u +\p[\pp,\partition_{\lambda}]u  -\p\big[(\nabla \partition_{\lambda})\q [(\A+\pp) u]\big] +\p f_{\lambda}\Big)\, \dd r\\
&  + \sum_{n\geq 1}\Big(\p\b_{\lambda,n}^{\e} v_{\lambda}+\p\b_{\lambda,n}^{\e}\nabla\J_{\lambda} u+\p[\partition_{\lambda},\b_n]u +\p[(\nabla \partition_{\lambda})\q (\b_{n} u)]+  \p g_{\lambda,n}\Big)\, \dd w_r^n
\\ & =\big(\Ff_{\lambda} u + \p f_{\lambda}\big) \, \dd r + \sum_{n\geq 1}\Big(\p\b_{\lambda,n}^{\e} v_{\lambda} + \G_{\lambda,n} u+\p g_{\lambda,n}\Big)\, \dd w_r^n
\end{align*}
with initial value $v_{\lambda}(t)=0$. Recall that by Step 3, one has $(\p\A_{\lambda}^{\e},(\p\b_{\lambda,n}^{\e})_{n\geq 1})\in \mathcal{SMR}_{p,\wtwo}(t,T)$ for $\wtwo\in \{0,\a\}$. Thus the claim follows from \cite[Proposition 3.12]{AV19_QSEE_1} and the previous displayed formula.

\textit{Step 5: There exists $\varepsilon(\Set)\in (0,1)$, $C_5(\Set)>0$ such that for each $\lambda\in \{1, \ldots, \Lambda\}$, $v\in \Hs^{1+\s,q}$ and $t\in [0,T)$,}
$$
\|\Ff_{\lambda} v\|_{\Hs^{-1+\s,q}}+\|(\G_{\lambda,n} v)_{n\geq 1}\|_{\Hs^{\s,q}(\ell^2)}\leq C_5 \|v\|_{H^{1+\s-\varepsilon,q}},  \ \ \
\text{ a.e.\ on }[ t,T]\times \O.
$$
First consider $\Ff_{\lambda}$. Recall that $\Lambda=\Lambda(\Set)<\infty$  by Step 3. Thus it is enough to prove suitable estimates for $\Ff_{\lambda}$ where $\lambda\in \{1,\dots,\Lambda\}$ is fixed. Let us write $\Ff_{\lambda}:=\Ff_{\lambda,\J}+\Ff_{\lambda,\A}+\Ff_{\lambda,\q}+\Ff_{\lambda,\hp}$, where
\begin{align*}
\Ff_{\lambda,\J}v &:= -\p[(\A^{\e}_{\lambda}+\pp^{\e}_{\lambda})\nabla \J_{\lambda} v], &
\Ff_{\lambda,\A} v&:=\p\big([\A,\partition_{\lambda}]v\big),\\
\Ff_{\lambda,\q} v&:=-\p\big[(\nabla \partition_\lambda)\q[(\A+\pp) v]\big], &
\Ff_{\lambda,\hp}v&:= \p\big([\pp,\partition_{\lambda}]v\big).
\end{align*}
%%%
Then by Lemma \ref{l:extension_operators} and the pointwise multiplication result of \cite[Proposition 4.1(3)]{AV21_SMR_torus} (recall that  $\eta>-\s$ and $\xi\in [q',\infty)$ by Assumption \ref{ass:max_reg_parameters}\eqref{it:max_reg_regularity_a_btwod_SMR}),
\begin{equation*}
\begin{aligned}
\|\p \A^{\e}_{\lambda}(\nabla \J_{\lambda} v)\|_{\Hs^{-1+\s,q}}
&\lesssim_{\s,q}\max_{i,j} \| \e_{y,r}^{\Tor^d} (a^{i,j}(t,\cdot))\partial_j \nabla \J_\lambda v \|_{H^{\s,q}}\\
&\lesssim_{\s,q} M \max_j  \|\partial_j \nabla\J_\lambda v\|_{H^{\s,q}}
\lesssim_{\s,q} M\|v\|_{H^{\s,q}},
\end{aligned}
\end{equation*}
where in the last estimate we applied Lemma \ref{l:J_operator}.
The above argument can be also applied to estimate $\p \pp^{\e}_{\lambda}(\nabla \J_{\lambda} v)$. Hence
$$
\|\Ff_{\lambda,\J}v \|_{\Hs^{-1+\s,q}}\lesssim_{M,\s,q} \|v\|_{\Hs^{s,q}}, \ \text{ for }v\in \Hs^{1+\s,q}.$$ 
To estimate $\Ff_{\lambda,\A}$, note that
$$
[\partition_{\lambda},\A] v= \sum_{i,j=1}^d\Big[ \partial_i\big(a^{i,j} (\partial_j \partition_{\lambda} )v \big)
+ (\partial_j v)\, a ^{i,j} (\partial_i \partition_\lambda)\Big],\qquad \text{ in }\D'(\Tor^d).
$$
By \cite[Proposition 4.1(3)]{AV21_SMR_torus} we obtain
\begin{align*}
\|\p[\partial_i(a^{i,j} (\partial_j \partition_\lambda )v )]\|_{\Hs^{-1+\s,q}}
\lesssim_{\s,q} \|a^{i,j} (\partial_j \partition_\lambda )v\|_{H^{\s,q}}
\lesssim_{\s,q}  M  \| v\|_{H^{\s,q}}.
\end{align*}
Letting $\varepsilon\in (0,1)$ be such that $\eta\geq -\s+\varepsilon$, in a similar way we obtain
\begin{equation}
\label{eq:NS_Tor_estimate_F_2}
\begin{aligned}
\|\p(\partial_j v\, a^{i,j} \partial_i \partition_\lambda)\|_{\Hs^{-1+\s,q}}
&\lesssim_{\s,q}  \|a^{i,j} \partial_j v\|_{H^{-1+\s,q}}\\
& \lesssim_{\s,q} \| a^{i,j} \partial_j v\|_{H^{\s-\varepsilon,q}}\\
&\lesssim_{\s,q} M \|\partial_j v\|_{H^{\s-\varepsilon,q}}
\lesssim_{\s,q} M \|v\|_{H^{1+\s-\varepsilon,q}},
\end{aligned}
\end{equation}
from which we obtain the required bound for $\Ff_{\lambda,\A}$.
By \eqref{eq:NS_Tor_mapping_property_p}, $\Ff_{\lambda,\q}$ can be estimates as follows:
\begin{align*}
\|\Ff_{\lambda,\q} v\|_{\Hs^{-1+\s,q}}
\lesssim \|\q(\A v)\|_{H^{-1+\s,q}}
\lesssim \| a^{i,j} \partial_j v\|_{H^{-1+\s,q}}
\lesssim \|v\|_{H^{1+\s-\varepsilon,q}},
\end{align*}
where the last estimate can be proved similarly as in \eqref{eq:NS_Tor_estimate_F_2}.
Next we discuss $\Ff_{\lambda,\hp}$. For notational convenience, set $\Lc_n v:= \div( \hp_n\otimes v)=(\div(\hp_n v^k))_{k=1}^d$ for $v\in \Hs^{\s,q}$. Note that the multiplication $ \hp_n\otimes v$ is well-defined due to  \cite[Proposition 4.1(3)]{AV21_SMR_torus} and the regularity assumptions on $(\hp_n)_{n\geq 1}$ in Assumption \ref{ass:max_reg_parameters}\eqref{it:max_reg_regularity_a_btwod_SMR}. Hence, for all $v\in H^{1+\s,q}$ we have
$$
F_{\lambda,\hp} v=\p \Big[\sum_{n\geq 1}\big( [\Lc_n,\partition_{\lambda}] \nabla\q (\bb_n v)
+\Lc_n[\nabla\q,\partition_{\lambda}]  (\bb_n v)
+ \Lc_n [\bb_n,\partition_{\lambda}] v \big)\Big].
$$
%%%%
Note that $[\Lc_n,\partition_{\lambda}]$ and $[\bb_n,\partition_{\lambda}] $ are zero-order differential operators and $[\nabla\q,\partition_{\lambda}]$ is a smoothing operator by Lemma \ref{l:commutator_localization}. Hence, reasoning as for $F_{\lambda,\A}$, we have $\|\Ff_{\lambda,\hp}v\|_{H^{-1+\s,q}}\lesssim \|v\|_{H^{1+\s-\varepsilon,q}}$ for some $\varepsilon>0$ depending only on $(d,\eta,\xi,\s)$.
Combining the previous estimates one obtains the claim for $\Ff_{\lambda}$.

By \eqref{eq:max_reg_step5_definition_F_G} we have $\G_{\lambda,n}:=\G_{\lambda,\q,n}+\G_{\lambda,\J,n}+\G_{\lambda,\b,n}$ where
\begin{equation*}
\G_{\lambda,\q,n} v:=
\p[(\nabla \partition_\lambda)\q( \b_{n} v)],
\quad
\G_{\lambda,\J,n}v :=\p\b_{n}\nabla\J_\lambda v,
\quad
\G_{\lambda,\b,n}v:=\p[(\partial_j \partition_\lambda) \btwod^j_n v]
\end{equation*}
where $v\in \Hs^{1+\s,q}$ and we used that $[\partition_\lambda,\b_{n}]v=(\partial_j \partition_{\lambda}) \btwod^j_n v$.
%%%
Now the estimate for $\G_{\lambda}$ can be proved in a similar way as we did for $\Ff_{\lambda}$.

\textit{Step 6: Conclusion.} Let $(u,f,g,\tau,t)$ be as at the beginning of the proof. By Step 3 we know that $(\p[\A_{\lambda}^{\e}+\pp^{\e}_{\lambda}],\p\b_{\lambda}^{\e})\in \MRtatz$ for all $\lambda\in \{1,\dots,\Lambda\}$. Combining the latter with Steps 4-5, one can see that there exist $\varepsilon(\Set)\in (0,1)$, $C_6(\Set)>0$ such that for all $\lambda\in \{1,\dots,\Lambda\}$,
\begin{align}
\label{eq:v_k_estimate_uniformly}
\|\p(\partition_{\lambda} u)\|_{L^p(( t,\tau)\times \O,w_{\wtwo}^t;\Hs^{1+\s,q})}
 \leq C_6 N_{\Ff_{\lambda} u,\G_{\lambda}u}(t,\tau)+C_6 N_{f,g}(t,\tau)&
 \\
\nonumber
 \leq C_5 C_6 \|u\|_{L^p(( t,\tau)\times \O,w_{\wtwo}^t;H^{1+\s-\varepsilon,q})}+C_6 N_{f,g}(t,\tau)&
\end{align}
where $N_{f,g}(t,\tau)$ is as in \eqref{eq:N_f_g_stokes}. As in \eqref{eq:decomposition_u_lambda}, since $\div\,u=0$ in $\D'(\Tor^d)$ a.e.\ on $[ t,\tau]\times \O$, one has
$
\partition_{\lambda} u=\p(\partition_{\lambda} u)+\nabla \J_{\lambda} u.
$
Since $(\partition_{\lambda})_{\lambda=1}^{\Lambda}$ is a partition of unity, we can write $u=\sum_{\lambda=1}^{\Lambda} u_{\lambda}$, and hence the previous considerations yield,
\begin{equation}
\label{eq:u_k_uniform_estimate_step_1}
\begin{aligned}
&\|u\|_{L^p(( t,\tau)\times \O,w_{\wtwo}^t;H^{1+\s,q})}\\
&{\leq }\sum_{\lambda=1}^{\Lambda} \Big(\|\p(\partition_{\lambda} u)\|_{L^p(( t,\tau)\times \O,w_{\wtwo}^t;H^{1+\s,q})}+
\|\nabla \J_{\lambda} u\|_{L^p(( t,\tau)\times \O,w_{\wtwo}^t;H^{1+\s,q})}\Big)\\
&\stackrel{(i)}{\lesssim} \Lambda C_6 N_{f,g}(t,\tau)+
\Lambda  C_5 C_6\| u\|_{L^p(( t,\tau)\times \O,w_{\wtwo}^t;H^{1+\s-\varepsilon,q})}
+ \Lambda C_7 \|u\|_{L^p(( t,\tau)\times \O,w_{\wtwo}^t;H^{\s,q})}\\
&\stackrel{(ii)}{\leq} \wt{C}_6 N_{f,g}(t,\tau)+\frac{1}{4} \|u\|_{L^p(( t,\tau)\times \O,w_{\wtwo}^t;H^{1+\s,q})}
+\wt{C}_6\|u\|_{L^p(( t,\tau)\times \O,w_{\wtwo}^t;H^{-1+\s,q})},
\end{aligned}
\end{equation}
where in $(i)$ we used \eqref{eq:v_k_estimate_uniformly} and Lemma \ref{l:J_operator} and in $(ii)$ a standard interpolation inequality. Since $u$ is a strong solution to \eqref{eq:NS_Tor_linearized} on $[ t,\tau]\times \O$, by \eqref{eq:estimate_AB_stok} and \cite[Lemma 3.14]{AV19_QSEE_1} there exists $(k_{s}(\Set))_{s>0}$ that $\lim_{s\downarrow 0} k_s=0$ and
\begin{equation}
\begin{aligned}
\label{eq:u_k_uniform_estimate_step_1_2}
\|u\|_{L^p(( t,\tau)\times \O,w_{\wtwo}^t;H^{-1+\s,q})}
&\leq k_{\tT}
\| u\|_{L^p(( t,\tau)\times \O,w_{\wtwo}^t;H^{1+\s,q})}+ k_{\tT}\,N_{f,g}(t,\tau)
\end{aligned}
\end{equation}
where we have used that $\tau-t\leq \tT$.

Combining \eqref{eq:u_k_uniform_estimate_step_1}-\eqref{eq:u_k_uniform_estimate_step_1_2}, one obtains
\begin{equation*}
\|u\|_{L^p(( t,\tau)\times \O,w_{\wtwo}^t;H^{1+\s,q})}\leq \wt{C}_6 N_{f,g}(t,\tau)+
\Big(\frac{1}{4}+k_{\tT} \wt{C}_6\Big) \|u\|_{L^p(( t,\tau)\times \O,w_{\wtwo}^t;H^{1+\s,q})},
\end{equation*}
By choosing $\tT(\Set)>0$ so that $k_{\tT}\leq 1/(4\wt{C}_6)$, the previous formula and \eqref{eq:spaces_divergence_free} imply the claimed a priori estimate of Lemma \ref{l:a_priori_estimates_NS}.

\emph{Step 7: The case $\s>0$}. To prove Lemma \ref{l:a_priori_estimates_NS} for $\s>0$, one can argue as in Steps 1-6. The only changes appear in Steps 2 and 5, where instead of   \cite[Proposition 4.1(3)]{AV21_SMR_torus}
we use \cite[Proposition 4.1(1)]{AV21_SMR_torus}.
\end{proof}

\begin{proof}[Proof of Lemma \ref{l:a_priori_estimates_NS} for general $h$]
Let us divide the proof of this step into three cases.
In each case we use the previously obtained estimate in the case $h=0$, and view the non-zero $h$ as a lower order perturbation.
Below, for all $v\in \Hs^{1+\s,q}$, we set
\begin{equation*}
\noise_{b}(t)v :=\Big(\sum_{n\geq 1} \Big[(I-\p)((\btwod_n(t)\cdot \nabla )v )\Big]\cdot h^{\cdot,k}_n\Big)_{k=1}^d.
\end{equation*}

\emph{Case $\s\leq 0$}.
Pick $\varepsilon\in (0,1)$ such that $\eta>-\s+\varepsilon$. Let $M$ be as in Assumption \ref{ass:max_reg_parameters}\eqref{it:max_reg_regularity_a_btwod_SMR}. Note that for each $v\in \Hs^{1+\s,q}(\Tor^d)$,
\begin{align*}
\|\p [\noise_{b} v]\|_{\Hs^{-1+\s,q}}
&\lesssim \max_{k} \Big\|\sum_{n\geq 1} \Big[(I-\p)[(b_n \cdot \nabla)v ]\Big]\cdot h^{\cdot,k}_n \Big\|_{\Hs^{\s-\varepsilon,q}}\\
&\stackrel{(i)}{\lesssim} \max_{j,k}\|(h^{j,k}_n)_{n\geq 1}\|_{H^{\eta,\xi}(\ell^2)}\|(b_n^j \partial_j v)_{n\geq 1} \|_{H^{\s-\varepsilon,q}}\\
&\stackrel{(ii)}{\lesssim} M \max_j  \|(b_n^j)_{n\geq 1}\|_{H^{\eta,\xi}(\ell^2)}\|\partial_j v\|_{H^{\s-\varepsilon,q}}
 \lesssim M^2 \|v\|_{H^{1+\s-\varepsilon,q}},
\end{align*}
where the implicit constants depend only on $s,q,d$. Moreover, in $(i)-(ii)$ we used that $\eta>-\s$, $\xi\geq q'$ and the pointwise multiplication result of \cite[Proposition 4.1(3)]{AV21_SMR_torus} and the text below it.

By Young's inequality and standard interpolation result, for all $\varepsilon>0$ we have
$$
 \|v\|_{H^{1+\s-\varepsilon,q}}\leq \varepsilon \|v\|_{H^{1+\s,q}}+ C_{\varepsilon} \|v\|_{H^{-1+\s,q}}
 \stackrel{\eqref{eq:spaces_divergence_free}}{=}\varepsilon \|v\|_{\Hs^{1+\s,q}}+ C_{\varepsilon} \|v\|_{\Hs^{-1+\s,q}}.
$$
Thus the claim of this step in case $\s\leq 0$ follows by combining the above estimate with the estimate in the case $h=0$. Indeed, since $u$ is a solution to \eqref{eq:NS_Tor_linearized}, setting
\[N_{f,g}:=\|f\|_{L^p(( t,\tau)\times \O,w_{\wtwo}^t;\Hs^{-1+\s,q})}+ \|g\|_{L^p(( t,\tau)\times \O,w_{\wtwo}^t;\Hs^{\s,q}(\ell^2))},\]
the estimate for $h=0$ and the above estimate for  $\p[\noise_b u]$ imply
\begin{align*}
\|u\|_{L^p(( t,\tau)\times \O,w_{\wtwo}^t;\Hs^{1+\s,q})}
& \leq C N_{f,g}+
C\|\p[\noise_b u]\|_{L^p(( t,\tau)\times \O,w_{\wtwo}^t;\Hs^{-1+\s,q})}
\\ & \leq
C N_{f,g} + C \varepsilon\|u\|_{L^p(( t,\tau)\times \O,w_{\wtwo}^t;\Hs^{1+\s,q})}\\
&\ + C_{\varepsilon}' \|u\|_{L^p(( t,\tau)\times \O,w_{\wtwo}^t;\Hs^{-1+\s,q})}.
\end{align*}
Choosing $\varepsilon=1/(2C)$, we get
\begin{align*}
\|u\|_{L^p(( t,\tau)\times \O,w_{\wtwo}^t;\Hs^{1+\s,q})}
 \leq
2C N_{f,g} + 2C_{\varepsilon}'\|u\|_{L^p(( t,\tau)\times \O,w_{\wtwo}^t;\Hs^{-1+\s,q})}.
\end{align*}
Reasoning as in Step 6 of Lemma \ref{l:a_priori_estimates_NS} one can also adsorb the remaining lower order term by using $\tau-t\leq T^*$ and choosing $T^*$ sufficiently small.

\emph{Case $\s\in (0,1]$}. In this case, for all $v\in \Hs^{1+\s,q}$,
\begin{align*}
\|\p [\noise_{b} v]\|_{\Hs^{-1+\s,q}}
&\lesssim
\| \noise_{b} v\|_{L^{q}}\\
&\lesssim \max_{i,j}\Big(\|(h_n^{i,j})_{n\geq 1}\|_{L^{\infty}(\ell^2)}\|(b_n^j)_{n\geq 1}\|_{L^{\infty}(\ell^2)}\Big)\|\nabla v\|_{L^{q}}
\stackrel{\eqref{eq:Holder_continuity_phi_h}}{\lesssim} \|v\|_{H^{1,q}}.
\end{align*}
Since $\s>0$, $\| v\|_{H^{1,q}}$ is lower order compared to $\|v\|_{H^{1+\s,q}}$. Thus the claim follows from standard interpolation estimates and Young's inequality as in the case $\s\leq 0$.

\emph{Case $\s> 1$}.
The proof goes as in the case $\s\leq 0$, where one should replace  \cite[Proposition 4.1(3)]{AV21_SMR_torus} by \cite[Corollary 4.2]{AV21_SMR_torus}.
\end{proof}

\section{Proofs of the main results}
\label{s:Navier_Stokes}

\subsection{Local well-posedness and the proof of Theorem \ref{t:NS_critical_local}}
\label{ss:NS_proofs}
Theorem \ref{t:NS_critical_local} will be derived from a more general local well-posedness which will be proved using the methods of \cite{AV19_QSEE_1,AV19_QSEE_2}. For $\delta\in (-1, 0]$ and $q\in [2, \infty)$ we define
\begin{equation}
\label{eq:ABFG_choice_NS}
\begin{aligned}
X_0&=\Hs^{-1+\s,q}, \qquad X_1 = \Hs^{1+\s,q} \qquad  A(\cdot)v=A_{\Stok} v, \qquad  B(\cdot)v =B_{\Stok} v,
\\ F(\cdot,v)&=\p\big[\Fd_0(\cdot,v)+\div(\Fd(\cdot,v))+f_{g,h}(\cdot,v) + f_{g,\hp}(\cdot,v) \big]-\p(\div(v\otimes v)),
\\
G(\cdot,v)&=\big(\p[\Gforce_n(\cdot,v )]\big)_{n\geq 1},
\end{aligned}
\end{equation}
for all $v\in X_1$.
Here $(A_{\Stok},B_{\Stok})$ is as in \eqref{eq:NS_differential_operators_concise_form} and \eqref{eq:generalizaed_Stokes_couple_divergence}, and
\begin{align*}
f_{g,h}(\cdot, v) &= \Big(\sum_{n\geq 1} \Big[(I-\p)\big[\Gforce_n(\cdot,v)\big]\Big]\cdot h^{\cdot,k}_n\Big)_{k=1}^d,\\
f_{g,\hp}(\cdot, v) &= \Big(
 \sum_{n\geq 1}\div \big(\hp_{n} \big[(I-\p)\Gforce_n(\cdot,u)\big]^k\big)
\Big)_{k=1}^d.
\end{align*}
Note that the linear part of $\wt{f}$ (see \eqref{eq:wtfbg}) is in the operator $A$, and the nonlinear part in $f_{g,h}$. In this way \eqref{eq:Navier_Stokes_generalized_with_projection} can be formulated as the semilinear stochastic evolution equation (see Definition \ref{def:sol_NS} for the definition of $(p,\a,q,\s)$-solutions):
\begin{equation}\label{eq:generalSEENS}
\left\{\begin{aligned}
\dd u + A u \,\dd t &= F(\cdot, u)\, \dd t + (B u +G(u)) \, \dd W_{\ell^2},\\
u(0)&=u_{0}.
\end{aligned}\right.
\end{equation}

Through this section we set
\begin{align*}
\Xap:=(X_0,X_1)_{1-\frac{1+\a}{p},p}=\Bs^{1+\s-2\frac{1+\a}{p}}_{q,p},\ \  \text{ and }\\
X_{\theta}:=[X_0,X_1]_{\theta}=\Hs^{1+\s-2\theta,q}, \ \ \text{ for }\theta\in (0,1),&
\end{align*}
where we used \eqref{eq:complex_real_interpolation}.
The following is our main result on local existence and regularization.
Below we set $1/0:=\infty$.

\begin{theorem}[Local well-posedness]
\label{t:NS_critical_localgeneralkappa}
Let Assumption \ref{ass:NS}$(p,\kappa,q,\delta)$ hold with
\begin{equation}
\label{eq:critical_equation_NS_lemma_hypothesis_H}
\delta\in (-1,0],  \ \ \
\frac{d}{2+\s}<q<\frac{d}{-\s} \  \ \ \text{ and }\ \  \
2\frac{1+\kappa}{p}+\frac{d}{q}\leq 2+\s.
\end{equation}
Then for every $u_0\in L^0_{\F_0}(\Omega;\Bs^{1+\delta-2\frac{1+\kappa}{p}}_{q,p}(\Tor^d))$, \eqref{eq:generalSEENS} has a $(p,\a,q,\s)$-solution $(u,\sigma)$ such that $\sigma>0$ a.s.\ and
\begin{equation}
\label{eq:NS_regularity_near_0generalkappa}
u\in L^p_{\rm loc}([0,\sigma),w_{\a};\Hs^{1+\s,q}(\Tor^d))\cap
C([0,\sigma);\Bs^{1+\delta-2\frac{1+\kappa}{p}}_{q,p}(\Tor^d)), \ \ \text{a.s.}
\end{equation}
and the trace space $\Bs^{1+\delta-2\frac{1+\kappa}{p}}_{q,p}(\Tor^d)$ is critical for \eqref{eq:generalSEENS} if and only if  the second part of \eqref{eq:critical_equation_NS_lemma_hypothesis_H} holds with equality. Moreover, $(u,\sigma)$ instantaneously regularizes in time and space in the following sense: a.s.
\begin{align}
\label{eq:H_regularization_NSgeneralkappa}
u&\in H^{\theta,r}_{\rm loc}(0,\sigma;\Hs^{1-2\theta,\zeta}(\Tor^d)) \ \
\text{ for all } \theta\in [0,1/2),\ r,\zeta\in (2,\infty),
\\ u& \in C^{\theta_1,\theta_2}_{\rm loc} ((0,\sigma)\times\Tor^d;\R^d) \ \  \text{ for all }  \theta_1\in [0,1/2),\ \theta_2\in (0,1).
\label{eq:C_regularization_NSgeneralkappa}
\end{align}
\end{theorem}
By Sobolev embedding it is straightforward to see that \eqref{eq:H_regularization_NSgeneralkappa} implies \eqref{eq:C_regularization_NSgeneralkappa}. The definition of criticality which we use is taken from \cite[Section 4.1]{AV19_QSEE_1} and some details can be found in the proof of Theorem \ref{t:NS_critical_localgeneralkappa}.

In order to prove the theorem, we first prove a lemma for the nonlinearities.
For the noise part we use the language of $\gamma$-radonifying operators for which we refer to \cite[Chapter 9]{Analysis2}.
The following identification (as sets and isomorphically) will be used several times below:
\begin{equation}
\label{eq:gamma_identification_H}
\gamma(\ell^2,H^{s,q}(\Tor^d))= H^{s,q}(\Tor^d;\ell^2)
\end{equation}
where $s\in \R$ and $q\in (1,\infty)$. The identity \eqref{eq:gamma_identification_H} follows from \cite[Proposition 9.3.1]{Analysis2} and the fact that $(I-\Delta)^{s/2}$ induces uniquely an isomorphism between $H^{s,q}(\Tor^d;\ell^2)$ and $L^q(\Tor^d;\ell^2)$.

By \eqref{eq:boundedness_p_H_sq} and the ideal property of $\gamma$-spaces \cite[Theorem 9.1.10]{Analysis2}, we also note that $\p$ extends to a bounded mapping from $\gamma(\ell^2,H^{s,q}(\Tor^d;\R^d))$ into $\gamma(\ell^2,\Hs^{s,q}(\Tor^d))$.

To prove the theorem we will first show that the nonlinearities $F$ and $G$ have the right mapping properties:
\begin{lemma}
\label{l:hyp_H_NS}
Let
Assumption \ref{ass:NS}$(p,\kappa,q,\delta)$ be satisfied and suppose that \eqref{eq:critical_equation_NS_lemma_hypothesis_H} holds.
Set $\beta=\frac{1}{2}(1-\frac{\s}{2}+\frac{d}{2q})$. Then there is a constant $C$ such that for all $v,v'\in \Hs^{\theta,q}$,
\begin{align*}
\|F(\cdot,v)-F(\cdot,v')\|_{X_{0}}& \leq C (1+\|\vone\|_{X_{\beta}}+\|\vtwo\|_{X_{\beta}})\|\vone-\vtwo\|_{X_{\beta}}
\\ \|F(\cdot,v)\|_{X_{0}}& \leq C (1+\|\vone\|_{X_{\beta}}^2)
\\ \|G(\cdot,\vone)-G(\cdot,\vtwo)\|_{\gamma(\ell^2,X_{1/2})}
&\leq  C(1+\|\vone\|_{X_{\beta}}+\|\vtwo\|_{X_{\beta}})\|\vone-\vtwo\|_{X_{\beta}}.
\\ \|G(\cdot,\vone)\|_{\gamma(\ell^2,X_{1/2})}
&\leq  C(1+\|\vone\|_{X_{\beta}}^2).
\end{align*}
\end{lemma}
Note that $\beta\in (0,1)$ by the conditions on the parameters in \eqref{eq:critical_equation_NS_lemma_hypothesis_H}.
\begin{proof}
Note that $X_{\beta} = \Hs^{\theta,q}$ with $\theta  = \frac{d}{2q}+\frac{\s}{2}$.
First we check the estimate for $F$. Let $F=F_1+F_2+F_3+F_4$ where
\begin{equation}
\begin{aligned}
 F_1(\cdot,v)&=\p[\div (v\otimes v)],
\\ \label{eq:defFi} F_2(\cdot,v)&=\p[\Fd_0(\cdot,v)+\div(\Fd(\cdot,v))],
\\ F_3(\cdot,v) &= \p[f_{g,h}(\cdot,v)],
\\  F_4(\cdot,v) &= \p[f_{g,\hp}(\cdot,v)].
\end{aligned}
\end{equation}
for all $v\in \Hs^{1+\s,q}$.
First we estimate $F_1$. For all $\vone,\vtwo\in X_1$,
\begin{equation}
\label{eq:NS_F_estimate_F_s_1_2}
\begin{aligned}
\|F_1(\cdot,\vone)-F_1(\cdot,\vtwo)\|_{\Hs^{-1+\s,q}}
&\lesssim \|(\vone\otimes \vone)-(\vtwo\otimes \vtwo)\|_{H^{\s,q}}\\
&\stackrel{(i)}{\lesssim} \|(\vone\otimes \vone)-(\vtwo\otimes \vtwo)\|_{L^{\lambda}}\\
&\lesssim \|\vone\otimes (\vone-\vtwo)\|_{L^{\lambda}}+ \|(\vone-\vtwo)\otimes \vone\|_{L^{\lambda}}\\
&\leq (\|\vone\|_{L^{2\lambda}}+\|\vtwo\|_{L^{2\lambda}})\|\vone-\vtwo\|_{L^{2\lambda}}\\
&\stackrel{(ii)}{\lesssim}(\|\vone\|_{\Hs^{\theta,q}}+\|\vtwo\|_{\Hs^{\theta,q}})\|\vone-\vtwo\|_{\Hs^{\theta,q}},
\end{aligned}
\end{equation}
where in $(i)$ and $(ii)$ we have used the Sobolev embedding with $-\frac{d}{\lambda}=\s-\frac{d}{q}$ and $\theta-\frac{d}{q}=-\frac{d}{2 \lambda}$. Therefore, we find $\theta=\frac{d}{2q}+\frac{\s}{2}$.
To ensure that both Sobolev embeddings hold note that $\lambda\in (q/2,q]$ and $\theta>0$ by \eqref{eq:critical_equation_NS_lemma_hypothesis_H}.
The estimate for $\|F_1(v)\|_{X_0}$ follows since $F_1(0) = 0$.

We prove a similar estimate for $F_2$. Indeed, by Assumption \ref{ass:NS}, for all $v,v'\in X_1$
\begin{equation}
\label{eq:NS_F_estimate_F_s_1_2_second_part}
\begin{aligned}
\|F_2(\cdot,v)-F_2(\cdot,v')\|_{\Hs^{-1+\s,q}}
&\lesssim  \sum_{j=0}^d
\|\Fd_j(\cdot,v)-\Fd_j(\cdot,v')\|_{L^{\lambda}}\\
&\lesssim \|(1+|v|+|v'|)|v-v'|\|_{L^{\lambda}}\\
&\lesssim (1+\|v\|_{L^{2\lambda}}+|v'\|_{L^{2\lambda}})\|v-v'\|_{L^{2\lambda}}\\
&\lesssim(1+\|\vone\|_{\Hs^{\theta,q}}+\|\vtwo\|_{\Hs^{\theta,q}})\|\vone-\vtwo\|_{\Hs^{\theta,q}},
\end{aligned}
\end{equation}
where $\lambda$ and $\theta$ are as before. The estimate for $\|F_2(\cdot,v)\|_{X_0}$  follows as well by using the boundedness of $\Fd_j(\cdot, 0)$.

Before treating $F_3$ and $F_4$, it is convenient to prove the required estimates for $G$. Recall that $X_{1/2}=\Hs^{\s,q}$ (see \eqref{eq:complex_real_interpolation}) and let $\beta,\theta$ be as above.
Then, Assumption \ref{ass:NS} and \eqref{eq:gamma_identification_H} yield that, for all $\vone,\vtwo\in X_1$,
\begin{align}
\label{eq:NS_G_estimate_G_s_1_2}
\|G(\cdot,\vone)-G(\cdot,\vtwo)\|_{\gamma(\ell^2,\Hs^{\s,q})}
&\lesssim \|\Gforce(\cdot,\vone)-\Gforce(\cdot,\vtwo)\|_{\gamma(\ell^2,L^{\lambda})}\\
\nonumber
&\lesssim (1+\|\vone\|_{L^{2\lambda}}+\|\vtwo\|_{L^{2\lambda}})\|\vone-\vtwo\|_{L^{2\lambda}}\\
\nonumber
&\lesssim (1+\|\vone\|_{\Hs^{\theta,q}}+\|\vtwo\|_{\Hs^{\theta,q}})\|\vone-\vtwo\|_{\Hs^{\theta,q}}.
\end{align}
Hence $G$ satisfies the required estimates as $X_{\beta}=\Hs^{\theta,q}$ by construction, and $G(\cdot,0)\in L^{\infty}(\T^d;\ell^2)$ by Assumption \ref{ass:NS}\eqref{it:NS_g_force_estimatesmeas}.

Next we consider $F_3$. Since $h^{j,k}\in L^\infty(\Tor^d;\ell^2)$ (see Assumption \ref{ass:NS} and the text below \eqref{eq:L_infty_a_b}), by Assumption \ref{ass:NS}\eqref{it:NS_g_force_estimatesmeas}, for all $v,v'\in X_1$,
\begin{align}
\label{eq:NS_F_estimate_F_s_third_part}
\|F_3(\cdot,v)-F_3(\cdot,v')\|_{\Hs^{-1+\s,q}}&\lesssim \|f_{g,h}(\cdot,v) - f_{g,h}(\cdot,v)\|_{L^{\lambda}}
\\ 
\nonumber
& \lesssim \max_{j,k}\|h^{j,k}\|_{L^\infty(\ell^2)}\|g(\cdot,v)-g(\cdot,v')\|_{L^{\lambda}(\ell^2)}
\\  
\nonumber
&\lesssim(1+\|\vone\|_{\Hs^{\theta,q}}+\|\vtwo\|_{\Hs^{\theta,q}})\|\vone-\vtwo\|_{\Hs^{\theta,q}}.
\end{align}
The estimate for $\|F_3(\cdot,v)\|_{X_0}$  follows as well by using the boundedness of $f_{g,h}(\cdot, 0)$.

Finally we consider $F_4$. Arguing as in \eqref{eq:estimate_pp_new_term_localization}, by Assumption \ref{ass:NS}\eqref{it:max_reg_regularity_a_btwod} and a variation of \cite[Proposition 4.1(3)]{AV21_SMR_torus}, for all $v,v'\in \Hs^{1+\s,q}$,
\begin{align*}
\|F_4(\cdot,v)-F_4(\cdot,v')\|_{X_0}
&\lesssim \max_{k} \Big\|
 \sum_{n\geq 1}\div \big(\hp_{n} \big[(I-\p)(\Gforce_n(\cdot,\vone)-\Gforce_n(\cdot,\vtwo))\big]^k\big)\Big\|_{H^{-1+\s,q}}\\
&\lesssim\|\Gforce(\cdot,\vone)-\Gforce(\cdot,\vtwo)\|_{H^{\s,q}(\T^d;\ell^2)}\\
&\lesssim(1+\|\vone\|_{\Hs^{\theta,q}}+\|\vtwo\|_{\Hs^{\theta,q}})\|\vone-\vtwo\|_{\Hs^{\theta,q}},
\end{align*}
where the last step follows as in \eqref{eq:NS_G_estimate_G_s_1_2} since
$\gamma(\ell^2,H^{\s,q})=H^{\s,q}(\T^d;\ell^2)$
by \eqref{eq:gamma_identification_H}. The estimate for $F_4(\cdot, v)$ follows since $F_4(\cdot, 0)\in L^\infty(\T^d;\R^d)$. 

Combining the above and the estimates \eqref{eq:NS_F_estimate_F_s_1_2}, \eqref{eq:NS_F_estimate_F_s_1_2_second_part} and \eqref{eq:NS_F_estimate_F_s_third_part}, we obtain the required estimate for $F$ by the definition of $\beta$.
\end{proof}

The proof of Theorem \ref{t:NS_critical_localgeneralkappa} is split into three parts. In Part (A) we prove local well-posedness. In Parts (B) and (C) we prove \eqref{eq:H_regularization_NSgeneralkappa} in the case $[p>2]$ and $[q=p=d=2,\s=0]$, respectively. The proof of \eqref{eq:C_regularization_NSgeneralkappa} follows as a simple consequence as already  mentioned before.

\begin{proof}[Proof of Theorem \ref{t:NS_critical_localgeneralkappa} Part (A) -- Local well-posedness]
The proof is divided in several steps.

{\em Step 1: the conditions (HF) and (HG) of \cite[Section 4.1]{AV19_QSEE_1}
hold with \eqref{eq:ABFG_choice_NS}, and the trace space $\Xap =\Bs^{1+\s-2\frac{1+\kappa}{p}}_{q,p}$  is critical for \eqref{eq:generalSEENS} if and only if
\begin{align}\label{eq:criticalcond}
2\frac{1+\kappa}{p}+\frac{d}{q} = 2+\s.
\end{align}
}
To prove this we use Lemma \ref{l:hyp_H_NS} and consider two cases. Recall $\beta=\frac{1}{2}(1-\frac{\s}{2}+\frac{d}{2q})$.
%%%%
\begin{enumerate}[{\rm(1)}]
\item\label{it:F_1_NS_trace_regular} If $1-\frac{1+\kappa}{p}\geq \beta$,
one can check that \eqref{eq:criticalcond} does not hold by using $q>d/(2+\delta)$. We prove the desired mapping properties of $F$ and $G$ and non-criticality. Using that $X_{1-\frac{1+\kappa}{p}+\varepsilon}\hookrightarrow X_{\beta}$ for each $\varepsilon>0$, we get that Lemma \ref{l:hyp_H_NS} holds with $\beta$ replaced by $1-\frac{1+\kappa}{p}+\varepsilon$.
Letting $\rho_j=1,\beta_j=\varphi_j=1-\frac{1+\kappa}{p}+\varepsilon$ for $j\in \{1, 2\}$, where $\varepsilon>0$ is such that $\varepsilon<\frac{1+\kappa}{2p}$, one can check $\rho_1 \Big(\varphi_1-1+\frac{1+\kappa}{p}\Big) + \beta_1<1$, which means that $\Xap$ is not critical.
\item\label{it:F_1_NS_trace_not_strong}  If $1-\frac{1+\kappa}{p}< \beta$, then we set $\rho_j=1,\beta_j=\varphi_j=\beta$ and by Lemma \ref{l:hyp_H_NS} the condition in \cite[(4.2)]{AV19_QSEE_1} becomes
\begin{equation}
\label{eq:critical_condition_F_1_NS}
\frac{1+\kappa}{p}\leq \frac{\rho_j+1}{\rho_j}(1-\beta)=1-\frac{d}{2q}+\frac{\s}{2}.
\end{equation}
Note that \eqref{eq:critical_condition_F_1_NS} is equivalent to the second part of \eqref{eq:critical_equation_NS_lemma_hypothesis_H}.
Finally, the corresponding trace space $\Bs^{-1+\s-2\frac{1+\kappa}{p}}_{q,p}$ is critical for \eqref{eq:Navier_Stokes_generalized} if and only if the equality in \eqref{eq:critical_condition_F_1_NS} holds.
\end{enumerate}

{\em Step 2: Application of \cite[Theorem 4.8]{AV19_QSEE_1}.}
By Step 1 and Theorem \ref{t:maximal_Stokes_Tord} the conditions of the latter are satisfied. Therefore, there is a $(p,\a,q,\s)$-solution $(u,\sigma)$ and \eqref{eq:NS_regularity_near_0generalkappa} holds. The assertions on criticality of the trace space also follows from Step 1.
\end{proof}

Next we prove Part (B) of Theorem \ref{t:NS_critical_localgeneralkappa}, i.e.\ we show \eqref{eq:H_regularization_NSgeneralkappa}-\eqref{eq:C_regularization_NSgeneralkappa} in the case $p>2$. Here we exploit the results in \cite[Section 6]{AV19_QSEE_2} in the following way.
\begin{itemize}
\item Bootstrap regularity in time via \cite[Proposition  6.8]{AV19_QSEE_2} (Step 1a below) and \cite[Corollary 6.5]{AV19_QSEE_2} (Step 1b below).
\item Bootstrap high-order integrability in space. Here we apply \cite[Theorem  6.3]{AV19_QSEE_2} finitely many times in the $(\Hs^{-1,q_j},\Hs^{1,q_j},r,\alpha)$-setting with $(q_j)_{j\geq 1}$ such that $q_{j+1}-q_j\geq c>0$ (see Step 2).
\item Bootstrap regularity in space by applying \cite[Theorem  6.3]{AV19_QSEE_2} on the shifted scales $Y_i=\Hs^{-1+\s+2i}$ and $\wh{Y}_i=\Hs^{-1+2i}$ (see Step 3).
\end{itemize}

\begin{proof}[Proof of Theorem \ref{t:NS_critical_localgeneralkappa} Part (B) -- Proof of \eqref{eq:H_regularization_NSgeneralkappa} in the case $p>2$]
Let $(u,\sigma)$ be the $(p,\a,q,\s)$-solution to \eqref{eq:generalSEENS} provided by Part (A). In the proof below we use Theorem \ref{t:maximal_Stokes_Tord} without further reference to obtain the required stochastic maximal regularity in different settings.

\emph{Step 1: For all $r\in (2,\infty)$,}
\begin{equation}
\label{eq:step_1_regularization_NS}
u\in \bigcap_{\theta\in [0,1/2)} H^{\theta,r}_{\rm loc}(0,\sigma;\Hs^{1+\s-2\theta,q}), \text{ a.s.\ }
\end{equation}
To prove this regularization effect in time, in the case $\a=0$ we divide the argument into two sub-steps. In the latter case, we will first use
\cite[Proposition 6.8]{AV19_QSEE_2} to create a weighted setting with a slight increase in integrability. After that we will apply \cite[Corollary 6.5]{AV19_QSEE_2} to extend the integrability to arbitrary order. In the case $\a>0$, Step 1a below can be skipped.

\textit{Step 1a: If $\a=0$, then \eqref{eq:step_1_regularization_NS} holds for some $r>p$.}
Let $\alpha>0$ be such that
\begin{equation*}
Y_i=X_i=\Hs^{-1+2i-\s,q}, \ \  i\in \{0,1\}, \ \frac1p = \frac{1+\alpha}{r}, \ \ \text{ and } \ \ \frac{1}{r}=\beta_1-1+\frac{1}{p},
\end{equation*}
where $\beta_1\in (0,1)$ is as in Part (A) of the proof. Then $\alpha <\frac{r}{2}-1$. Now by Part (A) of the proof we can apply \cite[Proposition 6.8]{AV19_QSEE_2}  to obtain \eqref{eq:step_1_regularization_NS}.

\textit{Step 1b: \eqref{eq:step_1_regularization_NS} holds for all $r\in (2,\infty)$.}
Let either [$r=p$ and $\alpha=\a$, if $\a>0$] or [$r>p,\alpha>0$ be as in Step 1a, if $\a=0$]. Let $\wh{r}\in [r,\infty)$ be arbitrary and let $\wh{\alpha}\in [0,\frac{\wh{r}}{2}-1)$ be such that $\frac{1+\wh{\alpha}}{\wh{r}}<\frac{1+\alpha}{r}$. Step 1 of the above proof Part (A) and the above Step 1a show that \cite[Corollary 6.5]{AV19_QSEE_2} applies with $Y_i = \Hs^{-1+\s+2i,q}(\Tor^d)$, $\a$, $r,\alpha$  and $\wh{r},\wh{\alpha}$ as above.

\emph{Step 2: For all $r\in (2, \infty)$ and $\zeta\in (2,\frac{d}{-\delta})$,}
\begin{equation}
\label{eq:NS_regularity_step3}
u\in \bigcap_{\theta\in [0,1/2)} H^{\theta,r}_{\rm loc}(0,\sigma;\Hs^{1+\delta-2\theta,\zeta}), \text{ a.s.\ }
\end{equation}
From Step 2, we know that \eqref{eq:NS_regularity_step3} holds for all $r\in (2,\infty)$ and $\zeta=q$. Fix $\alpha>0$ and choose $r>p$ such that
$2\frac{1+\alpha}{r} + \frac{d}{q}<2+\delta$. By Step 1 of Part (A) we know that for all $\zeta\in [q,d/(-\delta))$, the assumptions (HF), (HG) in \cite[Section 4.1]{AV19_QSEE_1} are satisfied and non-criticality holds in the $(\Hs^{-1+\s,\zeta},\Hs^{1+\s,\zeta},r,\alpha)$-setting.
To prove the claim, it suffices to show the existence of $\varepsilon>0$ depending only on $q,\s$ such that for any $\zeta\in [q, d/(-\delta))$,
\begin{equation}
\label{eq:implication_Step_2_proof_regularization}
u\in \bigcap_{\theta\in [0,1/2)} H^{\theta,r}_{\rm loc}(0,\sigma;\Hs^{1+\delta-2\theta,\zeta})\text{ a.s. }
\Longrightarrow \
u\in \bigcap_{\theta\in [0,1/2)} H^{\theta,r}_{\rm loc}(0,\sigma;\Hs^{1+\delta-2\theta,\zeta +\varepsilon})\text{ a.s.}
\end{equation}
To prove \eqref{eq:implication_Step_2_proof_regularization},we apply \cite[Theorem 6.3]{AV19_QSEE_2}. Due to the previous choice of $r,\alpha$, Step 1 of Part (A), and the left-hand side of  \eqref{eq:implication_Step_2_proof_regularization}, \cite[Theorem 6.3]{AV19_QSEE_2} is applicable with
$$
Y_i=\Hs^{1+\delta-2i ,\zeta},  \ \  \wh{Y}_i =\Hs^{1+\delta-2i ,\zeta+ \varepsilon}, \ \  \wh{r}=r , \ \ \wh{\alpha}=\alpha
$$
where $\varepsilon>0$ is chosen so that
$$
Y^{\Tr}_{r}=\Bs^{1+\s-\frac{2}{r}}_{\zeta,r}
\embed
\Bs^{1+\s-2\frac{1+\alpha}{r}}_{\zeta+\varepsilon,r}=
\wh{Y}^{\Tr}_{\wh{\alpha},\wh{r}}.
$$
By Sobolev embeddings, the previous display holds provided $\frac{1}{\zeta+\varepsilon}-\frac{1}{\zeta}\leq 2\frac{\alpha}{d r}$. To satisfy the latter, it is enough to choose $\varepsilon:=2\frac{\alpha}{d r}$.

\emph{Step 3: For all $\zeta,r\in (2,\infty)$,}
\begin{equation}
\label{eq:step_2_regularization_NS}
u\in \bigcap_{\theta\in [0,1/2)} H^{\theta,r}_{\rm loc}(0,\sigma;\Hs^{1-2\theta,\zeta}), \text{ a.s.\ }
\end{equation}
If $\delta=0$, then \eqref{eq:step_2_regularization_NS} follows from \eqref{eq:NS_regularity_step3}. In case $\delta<0$ first fix $\zeta\in (d,d/(-\delta))$.
Let $r\ge p$ be so large that $r>\frac{2}{1+\s}$ and
\begin{align}\label{eq:choicerstep3}
\frac{2}{r} + \frac{d}{\zeta}<2+\delta.
\end{align}
This is possible since $\zeta>d$ and $\delta> -1$.
Set $\wh{\alpha}:= r(\frac{1}{r}-\frac{\s}{2})-1\in [0,\frac{r}{2}-1)$. Then by \eqref{eq:choicerstep3} and
Step 1 of Part (A), (HF), (HG), and non-criticality hold in the $(\Hs^{-1,q},\Hs^{1,q},r,\wh{\alpha})$-setting.
Now \cite[Theorem 6.3]{AV19_QSEE_2} with
$$
X_i=Y_i=\Hs^{-1+\s+2i,q}, \ \
\wh{Y}_i=\Hs^{-1+2i,q}, \ \
\wh{r}=r,\ \ \alpha=0, \ \ \wh{\alpha} \text{ as above. }
$$
gives \eqref{eq:step_2_regularization_NS} with the above $\zeta$ and $r$.
To check the conditions (1)-(3) in \cite[Theorem 6.3]{AV19_QSEE_2}, note that (1) and (2) follow from \eqref{eq:NS_regularity_step3}, Step 1 of Part (A) and the above choice of $Y_i,\wh{Y}_i,r,\wh{r},\alpha,\wh{\alpha}$. To check (3) note that $\frac{1+\wh{\alpha}}{r}=\frac{1}{r}-\frac{\s}{2}$. Thus
$$
Y_{r}^{\Tr}=\Bs^{1+\s-\frac{2}{r}}_{q,r}
= \Bs^{1-2\frac{1+\wh{\alpha}}{r}}_{q,r}=
\wh{Y}_{\wh{\alpha},\wh{r}}^{\Tr}.
$$
The second part of condition (3) holds due to \cite[Lemma 6.2(4)]{AV19_QSEE_2} since $\frac{1+\wh{\alpha}}{r}=\frac{1}{r}-\frac{\s}{2}$, $\wh{Y}_{1+\frac{\s}{2}}=\Hs^{1+\s,q}=Y_1$ and
$\wh{Y}_{0}=\Hs^{-1,q}=Y_{1+\frac{\s}{2}}$.

In order to obtain \eqref{eq:step_2_regularization_NS} for all $\zeta<\infty$, we can again apply \eqref{eq:implication_Step_2_proof_regularization} but this time with $\delta=0$.
\end{proof}

Next we prove Part (C) of Theorem \ref{t:NS_critical_localgeneralkappa}, i.e.\ we show \eqref{eq:H_regularization_NSgeneralkappa} in the case $q=p=2$ and $\a=0$. Here we follow the arguments in \cite[Section 7]{AV19_QSEE_2} employing the results in \cite[Section 6]{AV19_QSEE_2}, see \cite[Roadmap 7.4]{AV19_QSEE_2} for a strategy summary.
\begin{proof}[Proof of Theorem \ref{t:NS_critical_localgeneralkappa} Part (C) --  Proof of \eqref{eq:H_regularization_NSgeneralkappa} in the case $q=p=d=2$, $\s=0$]
Let $(u,\sigma)$ be the $(2,0,2,0)-$solution to \eqref{eq:generalSEENS} provided by Part (A). It is enough to prove that for all $\varepsilon\in (0,\frac{1}{2})$ for all $r\in (2,\infty)$,
\begin{equation}
\label{eq:step_regularization_NS_pq_2}
u\in \bigcap_{\theta\in [0,1/2)} H^{\theta,r}_{\loc} (0,\sigma;\Hs^{1-\varepsilon-2\theta}) \ \text{ a.s.\ }
\end{equation}
where, to shorthand the notation, we wrote $\Hs^{1-\varepsilon-2\theta}$ instead of $\Hs^{1-\varepsilon-2\theta,2}$. Indeed, afterwards we can apply Steps 2 and 3 of Part (B) with $\delta = -\varepsilon$ and $q=2$.

\textit{Step 1: \eqref{eq:step_regularization_NS_pq_2} holds for some $\varepsilon>0$ and $r=4$.}
Let $\varepsilon_0>0$ be such that the Assumptions \ref{ass:NS} holds for $\delta=-\varepsilon_0<0$, cf.\ Remark \ref{r:ass_NS}. Without loss of generality we assume $\varepsilon_0\leq \frac{1}{2}$.
For all $\varepsilon\in (0,\varepsilon_0)$, let $\alpha\in [ 0, 1)$ be such that
\begin{equation}
\begin{aligned}
\label{eq:frac_12_alpha_grather_than_0}
 \frac12 = \frac{1+\alpha}{4} + \frac{\varepsilon}{2}.
\end{aligned}
\end{equation}
Since $\varepsilon<\frac{1}{2}$, \eqref{eq:frac_12_alpha_grather_than_0} yields $\alpha\in (0,1)=(0,\frac{r}{2}-1)$ where $r=4$.   To conclude Step 1, it is enough  to check the assumptions of \cite[Proposition 6.8]{AV19_QSEE_2} with $p=2$, $
Y_i=\Hs^{-1+2i-\varepsilon} $ and $  X_i=\Hs^{-1+2i}$. However, these follow from Step 1 of Part (A), Theorem \ref{t:maximal_Stokes_Tord}, and the fact that $\varphi_j=\beta_j=\frac{1}{2}(1-\frac{d}{2q})=\frac{3}{4}$ for $j\in\{1,2\}$.

\textit{Step 2: Let $\varepsilon_0>0$ be as in Step 1.
For each $\varepsilon\in (0,\varepsilon_0\wedge \frac{1}{2})$, \eqref{eq:step_regularization_NS_pq_2} holds with arbitrary $r>2$.}
Let $r=4$ and $\alpha\in (0,\frac{r}{2}-1)$ be as in Step 1. Let $\wh{r}\in [4,\infty)$ be arbitrary and let $\wh{\alpha}\geq 0$ such that
\begin{align*}
Y_i =\Hs^{-1-\varepsilon+2i },
\
X_i =\Hs^{-1+2i},
\ \wh{r}\in (4,\infty), \ \frac{1+\wh{\alpha}}{\wh{r}}<\frac{1+\alpha}{r}.
\end{align*}
Then $\wh{\alpha}\in [0,\tfrac{\wh{r}}{2}-1)$. It remains to apply \cite[Corollary 6.5]{AV19_QSEE_2}. Note that the assumptions follow from Step 1 of Part (A), and
$$
\Yr= \Bs^{1-\varepsilon-\frac{2}{r}}_{2,r} \embed\Ls^2= X^{\mathsf{Tr}}_{0,2},
$$
where we used $1-\varepsilon>\frac{2}{r}$.
\end{proof}

\begin{proof}[Proof of Theorem \ref{t:NS_critical_local}]
It suffices to set $\a = \a_{\crit}$ in Theorem \ref{t:NS_critical_localgeneralkappa}. One can check that this is admissible under the conditions on $p, q, \delta$ in \eqref{eq:critical_equation_NS_lemma_hypothesis_H}.
\end{proof}
In Theorem \ref{t:NS_critical_local} we omitted $\delta\in (-1, -\frac{1}{2})$. The reason for this  is that it does not enlarge the class of critical spaces with smoothness $>-\frac12$ for the initial data for which we can treat \eqref{eq:Navier_Stokes_generalized}. Indeed, in case $\s\in (-1,-\frac{1}{2})$, then the restrictions $q<\frac{d}{-\s}$ and $2\frac{1+\kappa}{p}+\frac{d}{q}\leq 2+\s$ (see \eqref{eq:critical_equation_NS_lemma_hypothesis_H}) implies that the smoothness of the space of the initial data satisfies $1+\s-2\frac{1+\a}{p}> \frac{d}{q}-1 >-1-\s>-\frac{1}{2}$.

\subsection{Higher order regularity and the proof of Theorem \ref{t:NS_high_order_regularity}}

For the proof of Theorem \ref{t:NS_high_order_regularity}  we need the following lemma, which is the analogue of Lemma \ref{l:hyp_H_NS} but with $\delta$ replaced by a higher order smoothness $s>0$.
\begin{lemma}
\label{l:hyp_H_NS_high_order}
Suppose that Assumption \ref{ass:NS}$(p,\a,q,\delta)$ and Assumption \ref{ass:high_order_nonlinearities} holds.
Let $X_j=\Hs^{-1+2j+s,q}$ with $s\in (0,\eta]$. Assume that
\begin{equation}\label{eq:condparaNS}
q\in (d, \infty), \ p\in (2,\infty),\  \kappa\in \Big[0,\frac{p}{2}-1\Big)
\ \text{ satisfy }  \ 1+s-\frac{2(1+\kappa)}{p}-\frac{d}{q}>0.
\end{equation}
Then for all $n\geq 1$ there exists a constant $C_n$ such that for all $v,v'\in X_1$ satisfying $\|v\|_{\Xap},\|v'\|_{\Xap}\leq n$,
\begin{align*}
\|F(\cdot,v)\|_{X_{0}}&\leq C_n(1+\|\vone\|_{\Xap}), \\
\|F(\cdot,v)-F(\cdot,v')\|_{X_{0}} &\leq C_n\|\vone-\vtwo\|_{\Xap},
\\ \|G(\cdot,\vone)\|_{\gamma(\ell^2,X_{1/2})}
&\leq  C_n(1+\|\vone\|_{\Xap}), &  \\ \|G(\cdot,\vone)-G(\cdot,\vtwo)\|_{\gamma(\ell^2,X_{1/2})}
&\leq  C_n\|\vone-\vtwo\|_{\Xap},
\end{align*}
\end{lemma}
Note that this lemma will also imply non-criticality of $F$ and $G$ since using the notation of \cite[Section 4.1]{AV19_QSEE_1} we can take $F = F_{\Tr}$ and $G = G_{\Tr}$ and hence the critical parts $F_c$ and $G_c$ vanish.

\begin{proof}
Note that $X_{\beta} = \Hs^{s+2\beta,q}$, $\Xap = \Bs^{1+s-2(1+\a)/p}_{q,p}$.
Due to the last condition in \eqref{eq:condparaNS} and Sobolev embedding we have
\begin{equation}
\label{eq:boundedness_high_order_regularity}
\Xap\embed C^{\varepsilon}\cap H^{s,q},\qquad  \text{ for some }\varepsilon>0.
\end{equation}
We prove the claimed estimate for $F$. The estimate for $G$ follows similarly. Fix $n\geq 1$ and $v,v'\in X_1$ satisfying $\|v\|_{\Xap},\|v'\|_{\Xap}\leq n$.
As in \eqref{eq:defFi} we write $F=F_1+F_2+F_3$. Note that
\begin{align*}
\|F_1(\cdot,v)-F_1(\cdot,v')\|_{\Hs^{-1+s,q}}
&\lesssim_{s,q} \|(v-v')\otimes v\|_{H^{s,q}}+\|v'\otimes (v-v')\|_{H^{s,q}} \\
&\stackrel{(i)}{\lesssim}_{s,q} \|v-v'\|_{L^{\infty}}\|v\|_{H^{s,q}} + \|v\|_{H^{s,q}}\|v-v'\|_{L^{\infty}} \\ & \stackrel{\eqref{eq:boundedness_high_order_regularity}}{\lesssim_n} \|v-v'\|_{\Xap}
\end{align*}
where in $(i)$ we used \cite[Proposition 4.1(1)]{AV21_SMR_torus} and in $(ii)$ Sobolev embedding. The estimate for $F_1(v)$ follows as well since $F_1(0) = 0$.

Next we prove the estimates for $F_2$. We only consider the $\div(f(\cdot,v))$ part, since the other term is similar. Let us write
\begin{align*}
\Fd(\cdot,v)-\Fd(\cdot,v')
=\Phi(v,v')(v-v'),&\\
\text{ where }\ \ \mathrm{\Phi}(v,v'):=\int_{0}^1 \partial_y \Fd(\cdot,(1-t) v+ t v')\,\dd t.&
\end{align*}
By repeating the argument for $F_1$, it is enough to show that, for all $v,v'\in X_1$ satisfying $\|v\|_{\Xap},\|v'\|_{\Xap}\leq n$,
$$
\|\Phi(v,v')\|_{H^{s,q}}\leq C_n.
$$
The latter follows from Assumption \ref{ass:high_order_nonlinearities}, \cite[Theorem 1]{BouMouSic} or \cite[Proposition 10.2, Chapter 13]{TayPDE3} and \eqref{eq:boundedness_high_order_regularity}. The estimate for $F_2(\cdot,v)$ can be proved more directly. The estimates for $F_3,F_4$ and $G$ can be proved similarly.
\end{proof}

After this preparation we will now be able to prove the higher order regularity by an iteration argument based on \cite[Theorem 6.3]{AV19_QSEE_2}.
\begin{proof}[Proof of Theorem \ref{t:NS_high_order_regularity}]
As usual we only prove \eqref{eq:H_regularization_NS_improved}, since \eqref{eq:C_regularization_NS_improved} follows from \eqref{eq:H_regularization_NS_improved} and Sobolev embeddings.

Let $\eta,\xi$ be as in Assumption \ref{ass:NS}. Let $\eta_0\geq 0$ and $\xi_0\in (d,\infty)$ be given by
\begin{equation*}
\xi_0
=\left	\{\begin{aligned}
d+1 \ \ &\text{ if }\xi\vee q\leq d,\\
\xi\vee q \ \ & \text{ otherwise},
\end{aligned}
\right.
\quad \text{ and }\quad
\eta-\frac{d}{\xi}= \eta_0-\frac{d}{\xi_0}.
\end{equation*}
Note that $\eta_0\in (0,\eta]$ since $\xi_0\geq \xi$ and $\eta>\frac{d}{\xi}$.
%%%%%%%
We split the proof into two steps. In the following steps, without further mentioning it, we apply Theorem \ref{t:maximal_Stokes_Tord}  several times in order to check stochastic maximal regularity of $(A,B)=(A_{\Stok},B_{\Stok})$. This is needed in order to check the conditions of our bootstrapping result \cite[Theorem 6.3]{AV19_QSEE_2}.
In case $\xi_0=\xi$ we also have $\eta_0=\eta$ and Step 2 below is not needed.

\emph{Step 1: For all $r\in (2,\infty)$,}
\begin{equation*}
u\in \bigcap_{\theta\in [0,1/2)} H^{\theta,r}_{\loc}(0,\sigma;\Hs^{1+\eta_0,\xi_0})\quad \text{ a.s.}
\end{equation*}
To show this fix $N\in \N$ so large that $\eta_0<N$. Set $s_k= k\frac{\eta_0}{N} $ for $k\in \{0,\dots,N\}$. By induction and \eqref{eq:NS_regularity_near_0} in Theorem \ref{t:NS_critical_local}, it is enough to prove that for each $k\in \{1,\dots,N\}$,
\begin{equation}
\label{eq:implication_high_order_regularity}
u\in \bigcap_{\theta\in [0,1/2)} H^{\theta,r}_{\loc}(0,\sigma;\Hs^{1+s_{k-1}-2\theta,\xi_0})\text{ a.s.}
 \Rightarrow
u\in \bigcap_{\theta\in [0,1/2)} H^{\theta,r}_{\loc}(0,\sigma;\Hs^{1+s_k-2\theta,\xi_0})\text{ a.s.}
\end{equation}
To prove \eqref{eq:implication_high_order_regularity} we apply \cite[Theorem 6.3]{AV19_QSEE_2} with $X_j = \Hs^{-1+2j+\s,q}$, $Y_j=\Hs^{-1+s_{k-1},\xi_0}$, $\wh{Y}_j=\Hs^{-1+s_{k},\xi_0}$ for $j\in \{0,1\}$ and $\alpha=0$, $r=\wh{r}$, $\wh{\alpha}$ to be chosen below.
Note that $\wh{Y}_j \embed Y_j \embed X_j$ for all $j\in \{0,1\}$.  Since $\xi_0>d$ and $\eta_0<N$, one can find $r\geq p$ such that $1-\frac{d}{\xi_0}>\frac{2}{r}$ and $\frac{\eta_0}{N}<1-\frac{2}{r}$. Set $\wh{\alpha}=\frac{r}{2}\frac{\eta_0}{N} \in [0,\frac{r}{2}-1)$.
It remains to check the conditions (1)--(3) of \cite[Theorem 6.3]{AV19_QSEE_2}. Conditions (1) and (2) are satisfied by the induction hypothesis and Lemma \ref{l:hyp_H_NS_high_order} which is applicable since
$$
1+s_{k}-\frac{d}{\xi_0}-2\frac{1+\wh{\alpha}}{r}
=
1+s_{k-1}-\frac{d}{\xi_0}-\frac{2}{r}
>0, \ \ \text{ for all }k\in \{1,\dots,N\},
$$
where we used $s_{k-1}\geq 0$ and $1-\frac{d}{\xi_0}>\frac{2}{r}$.  To check condition (3) note that \cite[(6.1)]{AV19_QSEE_2} follows from \cite[Lemma 6.2(4)]{{AV19_QSEE_2}}, $\wh{Y}_{1-\frac{\eta_0}{2N}}= Y_1$, $Y_{\frac{\eta_0}{2N}}= \wh{Y}_0$, and the choice of the parameters. For (3) it remains to observe that
$$
Y^{\Tr}_{r}=\Bs^{1+s_{k-1}-\frac{2}{r}}_{\xi,r} = \Bs^{1+s_{k}-2\frac{1+\wh{\alpha}}{r}}_{\xi,r}
=
\wh{Y}^{\Tr}_{\wh{\alpha},r}.
$$
Therefore, we can conclude \eqref{eq:implication_high_order_regularity}.

\emph{Step 2: Proof of \eqref{eq:H_regularization_NS_improved}}.
Fix $L\in \N$ such that $\eta-\eta_0<L$. Let $\eta_k=\eta_0+ k\frac{\eta-\eta_0}{L}$ where $k\in \{1,\dots,L\}$. Define $\xi_0\geq \xi_1\geq\ldots \xi_{L} = \xi$ by
\begin{equation}
\label{eq:s_k_choice}
\eta_k -\frac{d}{\xi_k}=\eta-\frac{d}{\xi},
\end{equation}
By induction and Step 1 it is enough to prove that for each $k\in\{1,\dots,L\}$,
\begin{equation}
\label{eq:implication_high_order_regularity_step_2}
u\in \bigcap_{\theta\in [0,1/2)} H^{\theta,r}_{\loc}(0,\sigma;\Hs^{1+\eta_{k-1}-2\theta,\xi_{k-1}})\text{ a.s.}
 \Rightarrow
u\in \bigcap_{\theta\in [0,1/2)} H^{\theta,r}_{\loc}(0,\sigma;\Hs^{1+\eta_k-2\theta,\xi_k})\text{ a.s.}
\end{equation}
Choose $r$ so large that $\frac{\eta-\eta_0}{L}<1-\frac{2}{r}$ (this is possible since $\eta-\eta_0<L$). Set $\wh{\alpha}:=\frac{\eta-\eta_0}{L} \frac{r}{2}\in [0,\frac{r}{2}-1)$.
As in the previous step, \eqref{eq:implication_high_order_regularity_step_2} follows from \cite[Theorem 6.3]{AV19_QSEE_2} applied with $X_j=\Hs^{-1+2j +\s,q }$, $Y_j =\Hs^{-1+2j+\eta_{k-1},\xi_{k-1}}$, $\wh{Y}_j= \Hs^{-1+2j+\eta_{k},\xi_{k}}$ and $r=\wh{r},\alpha=\wh{\alpha}$ as above.
Note that by \eqref{eq:s_k_choice} and Sobolev embeddings, $\wh{Y}_j\embed Y_j\embed X_j$ for $j\in \{0,1\}$. It remains to check the conditions (1)--(3) of \cite[Theorem 6.3]{AV19_QSEE_2}. Conditions (1)-(2) follows from Lemma \ref{l:hyp_H_NS_high_order}, which is applicable since
$$
1+\eta_k-2\frac{1+\wh{\alpha}}{r}-\frac{d}{\xi_k}\geq \eta_k-\frac{d}{\xi_k}\stackrel{\eqref{eq:s_k_choice}}{>}0,
$$
where we used $\frac{1+\wh{\alpha}}{r}<\frac{1}{2}$.
To check condition (3) note that \cite[(6.1)]{AV19_QSEE_2} follows from \cite[Lemma 6.2(1)]{AV19_QSEE_2}, and the fact that
\[Y^{\Tr}_r =\Bs^{1+\eta_{k-1}-\frac{2}{r}}_{\xi_{k-1},r}\embed \Bs^{1+\eta_{k-1}-\frac{2}{r}}_{\xi_{k},r}\embed \Bs^{1+\eta_{k}-2\frac{1+\wh{\alpha}}{r}}_{\xi_k,r} =\wh{Y}^{\Tr}_{\wh{\alpha},r},\]
where we used the definition of $\wh{\alpha}$
and the fact that $\xi_{k-1}\geq \xi_k$.
\end{proof}

\subsection{Blow-up criteria and proof of Theorem \ref{t:NS_serrin}}
\label{ss:NS_proofs_blow_up}
The proof of Theorem \ref{t:NS_serrin} is based on the blow-up criteria given in \cite[Section 4]{AV19_QSEE_2} and the extrapolation result in \cite[Lemma 6.10]{AV19_QSEE_2}.
We begin by specializing the \cite[Lemma 6.10]{AV19_QSEE_2} to the present situation.
Let $(u,\sigma)$ be the $(p,\a_{\crit},q,\s)$-solution \eqref{eq:Navier_Stokes_generalized} provided by Theorem \ref{t:NS_critical_local}. Consider the following stochastic Navier-Stokes equations starting at $\varepsilon>0$:
\begin{equation}\label{eq:Navier_Stokes_generalized_epsilon}
\left\{\begin{aligned}
\dd v + A v \,\dd t &= F(\cdot, v)\, \dd t + (B v +G(v)) \, \dd W_{\ell^2},\\
v(\varepsilon) & =\one_{\{\sigma>\varepsilon\}} u(\varepsilon).
\end{aligned}\right.
\end{equation}
Definition \ref{def:sol_NS} extends naturally to the problem \eqref{eq:Navier_Stokes_generalized_epsilon}. By either \eqref{eq:H_regularization_NS} or \eqref{eq:C_regularization_NS},
\begin{equation}\label{eq:uvarepsilreg}
\one_{\{\sigma>\varepsilon\}}u(\varepsilon)\in L^0_{\F_{\varepsilon}}(\O;\Bs^{\frac{d}{q_0}-1}_{q_0,p_0}(\Tor^d))\text{ a.s.\ for all }p_0,q_0\in (2,\infty).
\end{equation}
Thus Theorem \ref{t:NS_critical_localgeneralkappa} gives the existence and uniqueness of a $(p_0,0,q_0,\delta_0)$-solution $(v,\tau)$ under conditions on the parameters.

\begin{lemma}[Extrapolating life-span]
\label{l:NS_extrapolation}
Let Assumption \ref{ass:NS}$(p,\kappa_{\crit},q,\s)$ hold with
\[-\frac12\leq \delta\leq 0, \ \ \frac{d}{2+\s}<q<\frac{d}{1+\s}, \   \
\frac{2}{p}+\frac{d}{q}\leq 2+\s,\  \ \a=\a_{\crit}=-1+\frac{p}{2}\big(2+\s-\frac{d}{q}\big).\]
Let Assumption \ref{ass:NS}$(p_0, \kappa_0, q_0, \delta_0)$ hold with
\[\delta_0\in (-1,0],  \ \ \ \frac{d}{2+\s_0}<q_0<\frac{d}{-\s_0} \  \ \ \text{ and }\ \  \
2\frac{1+\kappa_0}{p_0}+\frac{d}{q_0}\leq 2+\s_0.
\]
Let $(u,\sigma)$ be the $(p,\a_{\crit},q,\s)$-solution to \eqref{eq:Navier_Stokes_generalized}.
Fix $\varepsilon>0$ and let $(v,\tau)$ be the $(p_0,\a_0,q_0,\delta_0)$-solution to \eqref{eq:Navier_Stokes_generalized_epsilon}. Then
$$
\tau=\sigma \text{ a.s.\ on }\{\sigma>\varepsilon\} \quad \text{ and }\quad
u=v\text{ a.e.\ on }[ \varepsilon,\sigma)\times \O.
$$
\end{lemma}

\begin{proof}
The claim follows by applying \cite[Lemma 6.10]{AV19_QSEE_2} with $X_j=\Hs^{-1+2j+\s,q}(\Tor^d)$, $Y_j =\Hs^{-1+2j+\delta_0,q_0}(\Tor^d)$, $\wh{Y}_j = \Hs^{-1+2j , \wh{q}_0}(\Tor^d)$, $r=p_0$, $\wh{r}=\wh{p}_0$, $\alpha=\a_0$ and $\wh{\alpha}=0$ where $\wh{p}_0,\wh{q}_0>2$ are large enough. To check this, note that the required regularity assumptions are satisfied thanks to \eqref{eq:H_regularization_NS} applied to $(u,\sigma)$ and $(v,\tau)$. The required stochastic maximal regularity also holds in the $(Y_0, Y_1, p_0, \kappa_0)$-setting by Theorem \ref{t:maximal_Stokes_Tord} and the conditions on the parameters. Moreover, by Lemma \ref{l:hyp_H_NS} the conditions (HF)-(HG) of \cite[Section 4.1]{AV19_QSEE_2} hold in the $(X_0, X_1, p, \kappa_{\crit})$-setting, $(Y_0, Y_1, p_0, \kappa_0)$-setting, and the $(\wh{Y}_0, \wh{Y}_1, \wh{p}_0, 0)$-setting if $\wh{p}_0, \wh{q}_0$ are large enough. The required embeddings $\wh{Y}_{\wh{p_0}}^{\Tr}\hookrightarrow X^{\mathsf{Tr}}_{\a_{\crit},p}$ and $\wh{Y}_1\hookrightarrow X_{1-\frac{\a}{p}}$ also hold if $\wh{p}_0, \wh{q}_0$ are chosen large enough.
\end{proof}
Next we prove Theorem \ref{t:NS_serrin} by applying our blow-up criteria from \cite{AV19_QSEE_2}.
\begin{proof}[Proof of Theorem \ref{t:NS_serrin}]
By Definition \ref{def:sol_NS}, \eqref{eq:Navier_Stokes_generalized} is understood as a stochastic evolution equation \cite[(4.1)]{AV19_QSEE_2} with the choice \eqref{eq:ABFG_choice_NS}, $X_j=\Hs^{-1+2j+\reg,q}$, $\a_{\crit}=-1+\frac{p}{2}(2+\reg-\frac{d}{q})$, and $H=\ell^2$. As in the proof of Lemma \ref{l:NS_extrapolation}, Lemma \ref{l:hyp_H_NS} implies that (HF)-(HG) in \cite[Section 4.1]{AV19_QSEE_2} hold in the $(X_0, X_1, p, \kappa_{\crit})$-setting with $\varphi_j=\beta_j = \frac{1}{2}(1-\frac{\s}{2}+\frac{d}{2q})$, $\rho_j=1$ for $j\in \{1,2\}$. The same holds with $(p,\a_{\crit},q, \delta)$ replaced by $(p_0,\a_{0,\crit},q_0, \delta_0)$ where $\a_{0,\crit}:=-1+\frac{p_0}{2}(2+\s_0-\frac{d}{q_0})$ for the translated problem \eqref{eq:uvarepsilreg}. By the condition $q_0<\frac{d}{1+\delta_0}$, one can check that $\a_{0,\crit}<\frac{p_0}{2} - 1$.

Finally note that Theorem \ref{t:maximal_Stokes_Tord} can be applied in these two different settings, which will be used in Step 1 below when we apply the blow-up criteria. Below $\varepsilon,T\in (0,\infty)$ are fixed.

\emph{Step 1: Proof of \eqref{eq:Serrin_NS_H}}.
\cite[Theorem 4.11]{AV19_QSEE_2} applied to the $(p,\a_{\crit,0},q_0,\delta_0)$-solution $(v,\tau)$ to \eqref{eq:Navier_Stokes_generalized_epsilon} gives
\begin{equation*}
\P\Big(\varepsilon<\tau<T,\, \|v\|_{L^{p_0}(\varepsilon,\tau; \Hs^{\gamma_0,q_0}(\Tor^d;\R^d))}<\infty\Big)=0,
\end{equation*}
where we used $[\Hs^{-1+\delta_0,q_0},\Hs^{1+\delta_0,q_0}]_{1- \frac{\a_{0,\crit}}{p_0}} = \Hs^{\gamma_0,q_0}$ and that $\div \, v = 0$ on $[\varepsilon,\sigma)\times \O$.
Therefore, the required result follows from Lemma \ref{l:NS_extrapolation}.

\emph{Step 2: Proof of  \eqref{it:blow_up_NS_critical_space_not_sharp_B}--\eqref{it:Serrin_NS_L}}.

\eqref{it:blow_up_NS_critical_space_not_sharp_B}: Let $q_1>q_0$. Since $B^{\frac{d}{q_0}-1}_{q_1,\infty}(\T^d;\R^d)\hookrightarrow B^{\frac{d}{q_0}-1}_{q_2,\infty}(\T^d;\R^d)$ for $q_2<q_1$, replacing $q_1$ by a smaller number if necessary, we may assume $q_1\in (q_0,\frac{d}{1+\s})$. Pick $p_1\in (2,\infty)$ so large that
$$
\frac{2}{p_1}+\frac{d}{q_1} -1 < \frac{d}{q_0}-1.
$$
Note that $B^{d/q_0-1}_{q_1,\infty}\embed H^{\gamma_1,q_1}$ where $\gamma_1=
\frac{2}{p_1}+\frac{d}{q_1} -1$ and therefore
\begin{align*}
\P\Big(\varepsilon<\sigma<T,\,\sup_{t\in [\varepsilon,\sigma)}\|u(t)\|_{B^{\frac{d}{q_0}-1}_{q_1,\infty}}<\infty\Big)
\leq
\P\Big(\varepsilon<\sigma<T,\, \|u\|_{L^{p_1}(\varepsilon,\sigma;H^{\gamma_1,q_1})}<\infty\Big)=0,
\end{align*}
where in the last equality we apply \eqref{eq:Serrin_NS_H} with $(p_0,q_0,\gamma_0)$ replaced by $(p_1,q_1,\gamma_1)$.

\eqref{it:Serrin_NS_L}: Follows from \eqref{eq:Serrin_NS_H}.
\end{proof}

\subsection{Global well-posedness: Proofs of Theorems \ref{t:global_small_data} and  \ref{t:NS_initial_data_large_two_dimensional}}

We begin by deriving
Theorem \ref{t:NS_initial_data_large_two_dimensional} from the blow-up criterion \eqref{eq:Serrin_NS_H} in Theorem \ref{t:NS_serrin} and the energy estimates provided by Theorem \ref{t:2D_global_rough_noise}.
\begin{proof}[Proof of Theorem \ref{t:NS_initial_data_large_two_dimensional}]
%%%%%%
Let $(u,\sigma)$ be the $(p,\a_{\crit},q,\s)$-solution provided by Theorem \ref{t:NS_critical_local}, where $\a_{\crit}$ is as in Theorem \ref{t:NS_initial_data_large_two_dimensional}.
By combining Lemma \ref{l:NS_extrapolation} with $p_0=q_0=2$, $\delta_0=0$ and Theorem \ref{t:2D_global_rough_noise}\eqref{it:NS_2D_global_weak_solution_estimate} from the appendix, we obtain that, for all $\varepsilon>0$,
\begin{equation}
\label{eq:u_pathwise_L_2_integrability}
u\in L^2(\varepsilon,\sigma;H^{1,2}(\Tor^2;\R^2)) \ \  \text{ a.s.\ on }\{\sigma>\varepsilon\}.
\end{equation}
Thus, for all $0<\varepsilon<T<\infty$,
\begin{equation}
\label{eq:P_0_d_2}
\P(\varepsilon<\sigma<T) \stackrel{\eqref{eq:u_pathwise_L_2_integrability}}{=}
\P\Big(\varepsilon<\sigma<T,\, \|u\|_{L^2(\varepsilon,\sigma;H^{1,2}(\Tor^2;\R^2))}<\infty\Big)
 =0
\end{equation}
where the last equality follows from \eqref{eq:Serrin_NS_H} of
Theorem \ref{t:NS_serrin} with $p_0=q_0=2$ and $\gamma_0=1$. Since $\sigma>0$ a.s.\ by Theorem \ref{t:NS_critical_local}, from \eqref{eq:P_0_d_2} and the arbitrariness of $0<\varepsilon<T<\infty$, it follows that $\P(\sigma<\infty)=0$. Hence $\sigma=\infty$ a.s.\ as desired.

The remaining assertions follow from Theorems \ref{t:NS_critical_local} and \ref{t:NS_high_order_regularity}.
\end{proof}

Next we prove Theorem \ref{t:global_small_data}. As in the proof of \cite[Theorem 1.4]{FS64_NS}, the main idea is to exploit the quadratic growth of the nonlinearity $\p[\div(u\otimes u)]$ to obtain an estimate where the LHS and the RHS have different scaling in $u$. Then the conclusion follows by choosing the data sufficiently small so that the quadratic growth of the RHS keeps sufficiently strong Sobolev norms of $u$ bounded over time. The nonlinearities $f_0, f$ and $g_{n}$ might contain lower order terms (see the estimate \eqref{eq:def_N_C0}), which need to be dealt with as well. This can be achieved by first working on small time intervals and then combine solutions as in \cite[Proposition 3.1]{AV21_SMR_torus}. The case of small time intervals is the content of the following result.

\begin{proposition}
\label{prop:global_small_data_small_interval}
Let Assumption \ref{ass:NS}$(p,\kappa,q,\s)$ be satisfied and suppose that one of the following conditions holds:
\begin{itemize}
\item $ \delta\in [-\frac{1}{2} ,0]$, $\frac{d}{2+\s}<q<\frac{d}{1+\s}$, $\frac{2}{p}+\frac{d}{q}\leq 2+\s$, and $\a=\a_{\crit}=-1+\frac{p}{2}\big(2+\s-\frac{d}{q}\big)$;
\item $\delta = \kappa =\kappa_{\crit}=0$ and $p=q=d=2$.
\end{itemize}
%%%
Assume that \eqref{eq:condfggrowtheq} holds for some constants $M_1,M_2>0$. Then there exists $T_0\in (0,T]$ depending only on the parameters in Assumption \ref{ass:NS} and $M_2$ for which the following assertion holds:
For all $\varepsilon_0\in (0,1)$ there exists a constant $C_{\varepsilon_0}>0$ such that, for all $t\in [0,T]$ and $v \in L^p_{\F_t}(\Omega;\Bs^{d/q-1}_{q,p})$ with
$$
\E\|v\|_{B^{d/q-1}_{q,p}}^p+M_1^p\leq C_{\varepsilon_0},
$$
the $(p,\a_{\crit},q,\s)$-solution $(u_t,\sigma_t)$ to \eqref{eq:Navier_Stokes_generalized} with initial data $v$ at time $t$ satisfies:
\begin{enumerate}[{\rm(1)}]
\item\label{it:life_taut_largeT}
$
\P(\sigma_t\geq  T_t)>1-\varepsilon_0$ where $ T_t:=(t+T_0)\wedge T.
$
\item\label{it:estimate_taut_largeT} There exists a stopping time $\tau_t\in (t,\sigma_t]$ a.s.\ such that $\P(\tau_t\geq T_t)>1-\varepsilon_0$ and
\begin{align*}
\E\Big[\one_{\{\tau_t\geq   T_t\}} \|u_t\|_{H^{\theta,p}(t,t+T_0,w_{\a_{\crit}}^{t};H^{1+\s-2\theta,q})}^p\Big]
\leq K_{\theta}(
\E\|v\|_{B^{d/q-1}_{q,p}}^p+M_{1}^p)
\end{align*}
for all $\theta\in [0,\tfrac{1}{2})$, where $K_{\theta}$ is a constant independent of $(t,v)$. Here in case $p=q=d=2$ and $\theta>0$ the {\normalfont{LHS}} of the above estimate should be replaced by $\E\big[\one_{\{\tau_t\geq   T_t\}} \|u_t\|_{C([t,\tau_t];\Ls^2)}^2\big]$.
\end{enumerate}
\end{proposition}

The key point in the above is the independence of $T_0$ on $t\in [0,T]$ and $\varepsilon_0>0$.

Note that $(p,\a_{\crit},q,\s)$-solutions to \eqref{eq:Navier_Stokes_generalized} at a time $t>0$ can be defined as in Definition \ref{def:sol_NS} and their existence is ensured by Theorem \ref{t:NS_critical_local} by a shift argument.
In the case of pure transport noise (i.e.\ $f_j = 0$ and $g_n=0$), the proof of Proposition \ref{prop:global_small_data_small_interval} shows that $T_0=T$, and thus the iteration argument needed to deduce Theorem \ref{t:global_small_data}  from Proposition \ref{prop:global_small_data_small_interval} can be avoided.

Next we first show how Proposition \ref{prop:global_small_data_small_interval} implies Theorem \ref{t:global_small_data}. We postpone the proof of Proposition \ref{prop:global_small_data_small_interval} to Subsection \ref{sss:global_small_data_small_interval}.

\begin{proof}[Proof of Theorem \ref{t:global_small_data}]
We begin by collecting some useful facts.
Fix $T\in (0,\infty)$ and $\varepsilon\in (0,1)$.
Let $(T_0,C_{\varepsilon_0})$ be as in Proposition \ref{prop:global_small_data_small_interval}, set $N_0:=\lceil \frac{T}{T_0}\rceil$ and fix $\theta_0\in (\frac{1+\a_{\crit}}{p},\frac{1}{2})$. Let $c_0$ be the norm of the embedding (see \cite[Theorem 1.2]{ALV21} and \cite[Proposition 2.5]{AV19_QSEE_1})
\begin{equation}
\label{eq:trace_embedding_proof_iteration}
H^{\theta_0,p}(0,T_0;w_{\a_{\crit}};H^{1+\s-2\theta_0,q})\cap L^p(0,T_0,w_{\a_{\crit}};H^{1+\s,q})
\embed C([0,T_0];B^{d/q-1}_{q,p}),
\end{equation}
and set $K_0:=(K_{\theta_0}+K_0)c_0$ where $K_{\theta_0}$ is as Proposition \ref{prop:global_small_data_small_interval}\eqref{it:estimate_taut_largeT}. Without loss of generality we assume $K_0\geq 1$.

Below we prove the assertions of Theorem \ref{t:global_small_data} if $C_{\varepsilon,T}$ is chosen as:
\begin{equation}
\label{eq:choice_smallness_global_3D}
 C_{\varepsilon,T}:=\frac{C_{\varepsilon_0}}{2K_0^{N_0}}
 \qquad \text{ where }\qquad
\varepsilon_0:= \frac{\varepsilon}{2^{N_0}} .
\end{equation}
Note that $C_{\varepsilon,T}$ is independent of $u_0$  and $M_1$, since $(C_{\varepsilon_0},T_0)$ depends only on the parameters in Assumption \ref{ass:NS} and $M_2$.

We prove Theorem \ref{t:global_small_data} by an iteration argument. Let $(t_n)_{n=0}^{N_0}$ be a partition of $[0,T]$ with step $\leq T_0/2$, i.e.\ $0=t_0<t_1<\dots<t_{N_0-1}<t_{N_0}=T$ and $|t_n-t_{n-1}|\leq  T_0/2$
for all $0\leq n\leq N_0$. We claim that for all $n\in \{1,\dots,N_0-1\}$ there exist a stopping time $\tau_{n}\in [0,\sigma)$ such that, for all $\theta\in[0,\frac{1+\a_{\crit}}{p})$,
\begin{align}
\label{eq:inductive_assumption_proof_global_1}
\P\big(\tau_n> t_{n}\big)&>1-2^{n-N_0} \varepsilon ,\\
\label{eq:inductive_assumption_proof_global_2}
\E\big[\one_{\{\tau_n> t_n\}}\|u(t_n)\|_{B^{d/q-1}_{q,p}}^p\big]+M_1^p&\leq K_0^{n-N_0} C_{\varepsilon_0},\\
\label{eq:inductive_assumption_proof_global_3}
\E\big[\one_{\{\tau_{n}>  t_n\}} \|u\|_{H^{\theta,p}(0,t_n,w_{\a_{\crit}};\Hs^{1+\s-2\theta,q})}^p\big]
&\lesssim_{\theta,n}
\E\|u_0\|_{B^{d/q-1}_{q,p}}^p+M_1^p.
\end{align}
Indeed, the claims of Theorem \ref{t:global_small_data} are equivalent to the case $n=N_0$ of \eqref{eq:inductive_assumption_proof_global_1}-\eqref{eq:inductive_assumption_proof_global_3}.
Let us begin by noticing that \eqref{eq:inductive_assumption_proof_global_1}-\eqref{eq:inductive_assumption_proof_global_3} for $n\in \{1,2\}$ hold due
Proposition \ref{prop:global_small_data_small_interval}, \eqref{eq:choice_smallness_global_3D} and $t_1\leq t_2\leq T_0$.
Hence it remains to show that \eqref{eq:inductive_assumption_proof_global_1}-\eqref{eq:inductive_assumption_proof_global_1} hold for $n+1$ whenever they hold for $n\in \{2,\dots,N_0-1\}$.
To this end, set
$$
v:=\one_{\{\tau_{n-1}>t_{n-1}\}} u(t_{n-1})\in L^{p}_{\F_{t_{n-1}}}(\O;B^{d/q-1}_{q,p}).
$$
Note that, up to a time shift, Theorem \ref{t:NS_critical_local} ensures the existence of a $(p,\a_{\crit},q,\s)$-solution
Let $(u_{n-1},\sigma_{n-1})$ to \eqref{eq:Navier_Stokes_generalized} with initial data $v$ at time $t_{n-1}$. Thus,
Proposition \ref{prop:global_small_data_small_interval} provides a stopping time $\lambda_{n-1}\in (t_{n-1},\sigma_{n-1}]$ a.s.\ such that
\begin{align}
\label{eq:lambda_large_probability_epsilon0}
\P(\lambda_{n-1}>  t_{n+1})&>1-\varepsilon_0=1-2^{-N_0}\varepsilon,
\end{align}
and for  $(c_0,K_0)$ as described around \eqref{eq:trace_embedding_proof_iteration},
\begin{align}
\label{eq:lambda_large_probability_epsilon01}
&\E\Big[\one_{\{\lambda_{n-1}\geq   t_{n+1}\}} \sup_{t\in [\sigma_{n-1},\lambda_{n-1})}\|u_{n-1}(t)\|_{B^{d/q-1}_{q,p}}^p\Big]\\
\nonumber
&\stackrel{\eqref{eq:trace_embedding_proof_iteration}}{\leq}
 c_0\E\Big[\one_{\{\lambda_{n-1}\geq   t_{n+1}\}} \max_{\theta\in \{0,\theta_0\}}
\|u_{n-1}(t)\|^p_{H^{\theta,p}(t_{n-1},t_{n+1},w_{\a}^{t_{n-1}};H^{1+\s-2\theta})}\Big]\\
\nonumber
&\leq
K_0\Big( \E\big[\one_{\{\tau_{n-1}> t_{n-1}\}} \|u(t_{n-1})\|_{B^{d/q-1}_{q,p}}^p\big]+ M_1^p\Big)
\leq K_0^{n+1-N_0} C_{\varepsilon_0}.
\end{align}

Arguing as in the proof of Theorem \ref{t:NS_serrin} proved in Subsection \ref{ss:NS_proofs_blow_up}, by maximality and instantaneous regularization of $(p,\a_{\crit},q,\s)$-solution, we have
\begin{equation}
\label{eq:sigma_equal_sigma_k}
\begin{aligned}
\sigma&=\sigma_{n-1}\  \text{ on } \ \{\tau_{n-1}>t_{n-1}\},
\\ u&=u_{n-1}\  \text{ on }\ [t_{n-1},\tau_{n-1})\times \{\tau_{n-1}>t_{n-1}\}.
\end{aligned}
\end{equation}
Set $\tau_{n+1}:=\one_{\{\tau_{n-1}>t_{n-1}\}} \lambda_{n-1}+ t_{n-1}\one_{\{\tau_{n-1}\leq t_{n-1}\}}$. Note that $\tau_{n+1}$ is a stopping time since $\lambda_{n-1} > t_{n-1}$. Moreover, \eqref{eq:inductive_assumption_proof_global_1} and  \eqref{eq:lambda_large_probability_epsilon0} imply
$$
\P(\tau_{n+1}>t_{n+1})> 1-2^{n-N_0} \varepsilon -  2^{-N_0}\varepsilon\geq 1-2^{n+1-N_0}\varepsilon_0.
$$
In addition \eqref{eq:lambda_large_probability_epsilon01} and \eqref{eq:sigma_equal_sigma_k} yield
\begin{align*}
&\E\Big[\one_{\{\tau_{n+1}> t_{n+1}\}}\sup_{s\in [t_{n},t_{n+1}]} \|u(s)\|_{B^{d/q-1}_{q,p}}^p\Big]
\leq K_0^{n+1-N_0} C_{\varepsilon_0}.
\end{align*}
Hence \eqref{eq:inductive_assumption_proof_global_1}-\eqref{eq:inductive_assumption_proof_global_2} for $n+1$ are proved. Finally, \eqref{eq:inductive_assumption_proof_global_3} with $n+1$ follows by combining Proposition \ref{prop:global_small_data_small_interval}\eqref{it:estimate_taut_largeT} and the fact that the restriction operator $f\mapsto f|_{[t_n,t_{n+1}]}$ maps  $H^{\theta,p}(t_{n-1},t_{n+1},w_{\a_{\crit}}^{t_{n-1}};H^{1+\s-2\theta,q})$ into $H^{\theta,p}(t_{n},t_{n+1};H^{1+\s-2\theta,q})$ with norm depending only on $p, \kappa_{\crit}$ and $|t_{n-1}-t_n|=T_0/2$ see \cite[Proposition 2.1(1)]{AV19_QSEE_2}.
\end{proof}

\subsubsection{Proof of Proposition \ref{prop:global_small_data_small_interval}}
\label{sss:global_small_data_small_interval}
In this subsection we use the notation introduced in \eqref{eq:ABFG_choice_NS}. Also recall from Lemma \ref{l:hyp_H_NS} that $\beta = \frac{1}{2}(1+\frac{\s}{2}-\frac{d}{2q})$ and $X_{\beta} = \Hs^{\frac{d}{2q}+\frac{\delta}{2},q}$. To simplify the notation we consider only the case $t=0$. The general situation $t>0$ follows verbatim as the constants in Theorem \ref{t:maximal_Stokes_Tord} do not depend on $s\in [0,T]$.
To prove Theorem \ref{t:global_small_data} we need a suitable $L^p$-space which allows us to bound the nonlinearities  in \eqref{eq:Navier_Stokes_generalized}. For each $t\in (0,\infty)$ it is defined as
\begin{equation}
\label{eq:def_X_space_NS}
\X(t):=L^{2p} (0,t,w_{\a_{\crit}};X_{\beta}), \ \text{ where } \  \a_{\crit}=-1+\frac{p}{2}\Big(2+\s-\frac{d}{q}\Big).
\end{equation}
By Lemma \ref{l:hyp_H_NS} the above space coincide with the abstract one introduced in \cite[Subsection 4.3]{AV19_QSEE_1}, where we also showed
\[\bigcap_{\theta\in [0,1/2)} H^{\theta,p}(0,t;w_{\a_{\crit}};X_{1-\theta})\subseteq \X(t).\]
In particular, the solution $(u, \sigma)$ provided by Theorem \ref{t:NS_critical_local} satisfies a.s.\ for all $t\in (0,\sigma)$,  $u\in \X(t)$.

Arguing as in Lemma \ref{l:hyp_H_NS} and using \eqref{eq:condfggrowtheq} one can check that we can find a constant $C_0$ (depending on $M_2$) such that for all $v\in X_{\beta}$
\begin{equation}
\label{eq:def_N_C0}
N(v):=\|F(\cdot,v)\|_{X_0}+\|G(\cdot,v)\|_{\gamma(\ell^2,X_{1/2})}
\leq
C_0 M_1+ C_0(\|v\|_{X_{\beta}}+\|v\|_{X_{\beta}}^2).
\end{equation}
Therefore, by H\"older's inequality there exists a constant $m_{t}>0$ (depending only on $t, p,\a_{\crit}$ and $m_t$ in \eqref{eq:def_N_C0}) such that $\lim_{t\downarrow 0} m_t=0$ and for all $v\in \X(t)$ a.s.
\begin{equation}
\label{eq:estimate_F_G_stability_NS}
\begin{aligned}
\|N(v)\|_{L^p(0,t,w_{\a_{\crit}})}^p
\leq
C_0^pM_1^p + m_t \|v\|_{\X(t)}^p + C_0^p  \|v\|_{\X(t)}^{2p}.
 \end{aligned}
\end{equation}
Without loss of generality we assume that $t\mapsto m_t$ is non-decreasing.
This estimate motivates once more the definition of $\X(t)$ and plays a crucial  role in the analysis below.

The next result is a special case of \cite[Lemma 5.3]{AV19_QSEE_2}, and provides an a priori estimate for the linear part of the equation based on maximal $L^p$-regularity.
\begin{lemma}
\label{l:linear_estimate_NS}
Let Assumption \ref{ass:NS}$(p,\kappa_{\crit},q,\s)$ be satisfied. Fix $T\in (0,\infty)$.
Let $A=A_{\Stok}$ and $B=B_{\Stok}$ where $(A_{\Stok},B_{\Stok})$ are as in \eqref{eq:generalizaed_Stokes_couple_divergence}.
Then there exists $K> 0$ depending only on $p,q,\s$ and the constant $C_2$ in Theorem \ref{t:maximal_Stokes_Tord} such that the following holds:

For any stopping time $\tau:\O\to[0,T]$, and any
$$
v_0\in L^0_{\F_0}(\O;X^{\mathsf{Tr}}_{\a_{\crit},p}), \  f\in L^p_{\Progress}(( 0,\tau)\times \O,w_{\a_{\crit}};X_0), \ g\in L^p_{\Progress}(( 0,\tau)\times \O,w_{\a_{\crit}};\gamma(\ell^2,X_{1/2})),
$$ and any $(p, \a_{\crit},q,\delta)$-solution $v\in L^p_{\Progress}(( 0,\tau)\times \O ,w_{\a_{\crit}};X_1)$ to
\begin{equation*}
\left\{
\begin{aligned}
\dd v + A v \,\dd t&=f \, \dd t+ \big(Bv +  g\big)\, \dd W_{\ell^2},
\\ v(0)&=v_0,
\end{aligned}\right.
\end{equation*}
on $[ 0,\tau]\times \O$ satisfies the following estimate
\begin{align*}
\|v\|_{L^p(\Omega;\X(\tau))}^p \leq K^p\big( \|v_0\|_{L^p(\Omega;X^{\mathsf{Tr}}_{\a_{\crit},p})}^p
&+\|f\|_{L^p(( 0,\tau)\times \O,w_{\a_{\crit}};X_{0})}^p\\
& +\|g\|_{L^p(( 0,\tau)\times \O,w_{\a_{\crit}};\gamma(\ell^2,X_{1/2}))}^p\big).
\end{align*}
\end{lemma}

Next we turn to the proof of Proposition \ref{prop:global_small_data_small_interval}.

\begin{proof}[Proof of Proposition \ref{prop:global_small_data_small_interval}]
Throughout the proof we fix $\varepsilon_0\in (0,\infty)$. 
Recall that for notational convenience we are assuming $t=0$.
We claim that there exist $T_0,C_{\varepsilon_0},r_0>0$ independent of $u_0$ such that
\begin{equation}
\label{eq:alternative_implication_global_perturbation}
\E\big\|u_0\big\|_{B^{\frac{d}{q}-1}_{q,p}(\Tor^d;\R^d)}^p
\leq C_{\varepsilon_0}
\quad
\Longrightarrow \quad
\P(\W)>1-\varepsilon_0,
\end{equation}
where
$$\W = \big\{\|u\|_{\X(\sigma\wedge T_{0})}\leq r_{0}\big\}.
$$
Moreover, we obtain that $T_0$ is independent of $\varepsilon_0$, see \eqref{eq:choice_T0_proof_small} below.

We now split the proof into several steps.
In Step 1 we obtain Proposition \ref{prop:global_small_data_small_interval}\eqref{it:life_taut_largeT} from  \eqref{eq:alternative_implication_global_perturbation} by applying the Serrin criterion \eqref{eq:Serrin_NS_H} of Theorem \ref{t:NS_serrin}, while in Steps 2 and 3 we prove the claim \eqref{eq:alternative_implication_global_perturbation}. In Step 4 we show Proposition \ref{prop:global_small_data_small_interval}\eqref{it:estimate_taut_largeT}.

{\em Step 1: \eqref{eq:alternative_implication_global_perturbation} implies Proposition \ref{prop:global_small_data_small_interval}\eqref{it:life_taut_largeT} }. By \eqref{eq:estimate_F_G_stability_NS} on $\W$ we find
\begin{equation*}
\|N(u)\|_{L^p(0,\sigma\wedge T_0,w_{\a_{\crit}})}^p \leq m_{T_0}^p(1+ r_{0}^p)  + C_0 r_0^{2p}=:C_1.
\end{equation*}
Define the stopping time $\tau$ by $\tau = \inf\{t\in [0,\sigma): \|N(u)\|_{L^p(0,t,w_{\a_{\crit}})}^p\geq C_1+1\}\wedge T_0$, where we set $\inf\emptyset =\sigma\wedge T_0$. Then $\tau = \sigma\wedge T_0$ on $\W$.

By Theorem \ref{t:maximal_Stokes_Tord}, $(A,B)=(A_{\Stok},B_{\Stok})$ has stochastic maximal $L^p$-regularity on $[0,T_0]$. Since $(u,\sigma)$ is a $(p,\a_{\crit},q,\s)$-solution to \eqref{eq:Navier_Stokes_generalized}, as in \cite[Proposition 3.12(2)]{AV19_QSEE_1} it follows that a.s. on $[0,\tau)$
\begin{align*}
\dd u +A u \, \dd t = \one_{[0,\tau)} F(\cdot, u)\, \dd t + \big(Bu + \one_{[0,\tau)} G(\cdot, u)\big)\, \dd W_{\ell^2}
\end{align*}
and $u(0) = u_0$.
Now Theorem \ref{t:maximal_Stokes_Tord} with $\theta = \frac{\a_{\crit}}{p}$ gives
$u\in L^p(\O;H^{\frac{\a_{\crit}}{p},p}(0,\tau,w_{\a_{\crit}};X_{1-\frac{\a_{\crit}}{p}}))$.
Using $\tau=\sigma\wedge T_0$ on $\W$, by Sobolev embedding \cite[Proposition 2.7]{AV19_QSEE_1} we obtain
\begin{equation}
\begin{aligned}
\label{eq:upathproprtylocO}
u\in H^{\frac{\a_{\crit}}{p},p}(0,\sigma\wedge T_0,w_{\a_{\crit}};X_{1-\frac{\a_{\crit}}{p}})
 &\hookrightarrow L^p(0,\sigma\wedge T_0;X_{1-\frac{\a_{\crit}}{p}}) \\
&= L^p(0,\sigma\wedge T_0;\Hs^{\gamma,q}(\Tor^d)) \ \ \text{a.s.\ on $\W$},
\end{aligned}
\end{equation}
where $\gamma = 1+\delta-2\frac{\a_{\crit}}{p}=\frac{2}{p}+\frac{d}{q}-1$.  It follows that
\begin{equation}
\label{eq:sigma_equal_T_on_W}
\begin{aligned}
\P\big(\{\sigma<T_0\}\cap \W\big)&\stackrel{(i)}{=}
\lim_{s\downarrow 0}\P\Big(\Big\{s<\sigma<T_0,\, \|u\|_{L^p(s,\sigma;H^{\gamma,q})}<\infty\Big\}\cap\W \Big)\\
&\leq  \limsup_{s\downarrow 0}
\P\Big(s<\sigma<T_0, \, \|u\|_{L^p(s,\sigma;H^{\gamma,q})}<\infty\Big)\stackrel{(ii)}{=} 0.
\end{aligned}
\end{equation}
Here in $(i)$ we used $\sigma>0$ a.s.\ (see Theorem \ref{t:NS_critical_local}) and \eqref{eq:upathproprtylocO}. In $(ii)$ we used \eqref{eq:Serrin_NS_H} of Theorem \ref{t:NS_serrin} with $p_0 = p$, $q_0 = q$ and $\gamma_0 = \gamma$. Therefore, $\sigma \geq  T_0$ on $\W$ and this gives the result of the Proposition \ref{prop:global_small_data_small_interval}\eqref{it:life_taut_largeT}. Thus Step 1 is proved.

Next we turn to the proof of \eqref{eq:alternative_implication_global_perturbation}. Before we continue it is important to recall that from the discussion below \eqref{eq:def_X_space_NS} a.s.\ for all $t\in (0,\sigma)$, $u\in \X(t)$.

\emph{Step 2: Let $K$ be as in Lemma \ref{l:linear_estimate_NS}. Let $m_t$ be as in \eqref{eq:estimate_F_G_stability_NS}. Fix $T_0>0$ such that
\begin{equation}
\label{eq:choice_T0_proof_small}
K^p  m_{T_0}^p  \leq \tfrac{1}{2}.
\end{equation}
Then there exists $\Constant>0$ depending only on $K$ and $C_0$ such that for any $N\geq 1$ and any stopping time $\mu$ satisfying $0\leq  \mu \leq \sigma\wedge T_0$ and $\|u\|_{\X(\mu)}\leq N$ a.s.,}
\begin{align*}
\E \big[\psi_R(\|u\|_{\X(\mu)}^p)\big]\leq \E\|u_0\|_{B^{\frac{d}{q}-1}_{q,p}}^p+M_1^p, \quad \text{with} \quad \psi_R(x)=\frac{x}{R}-x^{2}.
\end{align*}

In order to prove the assertion in Step 2 we use a localization argument. Set
\begin{equation*}
\begin{aligned}
\sigma_n:=\inf\big\{t\in [0,\sigma)\,:\,\|u\|_{L^p(0,t,w_{\a_{\crit}};X_1)}+ \|u\|_{\X(t)}\geq n\big\}, \ \ n\geq 1,
\end{aligned}
\end{equation*}
where we set $\inf\emptyset =\sigma$, and let $\mu_n:=\mu\wedge\sigma_n$. Then $\lim_{n\to \infty} \sigma_n= \sigma$ a.s., $\lim_{n\to \infty} \mu_n=\mu$, and for all $n\geq 1$,
\begin{equation*}
\|u\|_{\X(\sigma_n)} + \|u\|_{L^p(0,\sigma_n,w_{\a_{\crit}};X_1)}\leq n.
\end{equation*}
%%%%%%
As in Step 1 it follows from Lemma \ref{l:linear_estimate_NS}, $\mu_n\leq T_0\leq T$ and \eqref{eq:estimate_F_G_stability_NS} that
\begin{align*}
\E \|u\|_{\X(\mu_n)}^p
&\leq K^p\Big( \E \|u_0\|_{B^{d/q-1}_{q,p}}^p+ \E\|F(\cdot,u)\|_{L^p(( 0,\mu_n)\times\O,w_{\crit};X_0)}^p  \\ & \qquad +\E\|G(\cdot,u)\|_{L^p(( 0,\mu_n)\times \O,w_{\crit};X_{1/2})}^p
\Big)
\\ &\leq K^p(\|u_0\|_{B^{d/q-1}_{q,p}}^p + C_0^pM_1^p)\\
& \qquad + K^p m_{T_0}^p\E \|u\|_{\X(\mu_n)}^{p} +K^p C_0^p\E \|u\|_{\X(\mu_n)}^{2p}
\end{align*}
The choice \eqref{eq:choice_T0_proof_small} and the above estimate imply
\[\E \|u\|_{\X(\mu_n)}^p \leq 2 K^p(\|u_0\|_{B^{d/q-1}_{q,p}}^p + C_0^pM_1^p)
+ 2K^p C_0^p\E \|u\|_{\X(\mu_n)}^{2p}.
\]
Letting $n\to \infty$, the desired estimate follows after division by $R=2K^p (1\vee C_0^p)$.

\textit{Step 3:  \eqref{eq:alternative_implication_global_perturbation} holds with $T_0$ satisfying \eqref{eq:choice_T0_proof_small} and 
\begin{equation}
\label{eq:choice_C_varepsilon_0_r_0}
C_{\varepsilon_0}:=\frac{\varepsilon_0}{8R^2} \qquad \text{ and }\qquad r_0:= \frac{1}{(2R)^{1/p}},
\end{equation}
where $R$ is as in Step 2.} Let $\psi_R$ be as in Step 2. It is easy to check that
$\psi_R$ has a unique maximum on $\R_+$ given by $\frac{1}{4R^2}$ which is attained at $x=\frac{1}{2R}$. With the choice in \eqref{eq:choice_C_varepsilon_0_r_0}, we have
$\W=\{\|u\|_{\X(\sigma\wedge T_0)}\leq (2R)^{-1/p}\}$ and set
\begin{align}
\label{eq:def_lambda_proof_small}
\mu&:=
\inf\{t\in [0, \sigma)\,:\, \|u\|_{\X(t)}\geq (2R)^{-1/p}\}\wedge T_0,
\end{align}
where  $\inf\emptyset:=\sigma\wedge T_0$. 
To derive a contradiction suppose that
\begin{equation}
\label{eq:smallness_initial_data}
\text{(a)} \ \ \ \E\|u_0\|_{B^{d/q-1}_{q,p}}^p+M_1^p\leq C_{\varepsilon_0}, \qquad
\text{and} \qquad \text{ (b)}\ \ \  \P(\W)\leq 1-\varepsilon_0.
\end{equation}
From the definition of $\W$ and \eqref{eq:def_lambda_proof_small}, we find that $\mu< \sigma\wedge T_0$ a.s.\ on $ \O\setminus \W$. Moreover,
\begin{align}
\label{eq:f_reach_max_1}
\psi_R(\|u\|_{\X(\mu)}^p)&=\frac{1}{4R^2} \  \text{ a.s.\ on }\O\setminus \W, \ \  \text{and} \ \
  \psi_R(\|u\|_{\X(\mu)}^p) \geq 0 \  \text{ a.s.\ on } \W.
 \end{align}
Therefore,
\begin{align*}
\E\Big[\psi_R(\|u\|_{\X(\mu)}^p)\Big]
&=\E\Big[\psi_R(\|u\|_{\X(\mu)}^p)\one_{ \W}\Big]
+\E\Big[\psi_R(\|u\|_{\X(\mu)}^p)\one_{\O\setminus  \W}\Big]\\
&\stackrel{\eqref{eq:f_reach_max_1}}{\geq}
\P(\O\setminus  \W) \frac{1}{4R^2} \\
&\stackrel{\eqref{eq:smallness_initial_data}_{\text{(b)}}}{\geq } \frac{\varepsilon_0}{4 R^2}\stackrel{\eqref{eq:smallness_initial_data}_{\text{(a)}}}{\geq }  2\big(\E\|u_0\|_{B^{d/q-1}_{q,p}}^p+M_1^p\big).
\end{align*}
Since without loss of generality we can assume $M_1>0$, the latter contradicts Step 2. Thus $\P(\W)>1-\varepsilon_0$ as desired.

{\em Step 4: Proof of Proposition \ref{prop:global_small_data_small_interval}\eqref{it:estimate_taut_largeT}}.
Let $\tau$ be stopping time defined by the following variation of \eqref{eq:def_lambda_proof_small}:
\begin{equation}
\label{eq:def_tau_small_T_0}
\tau:=
\inf\Big\{t\in [0, \sigma)\,:\, \|u\|_{\X(t)}\geq \Big(\frac{4}{3}R\Big)^{-1/p}\Big\},
\end{equation}
where  $\inf\emptyset:=\sigma$. Then from Step 3, one can check that $\tau\wedge T_0=\sigma\wedge T_0$ on $\W$, and thus
\[\P(\tau\geq T_0) \geq \P(\{\tau\geq T_0\}\cap \W) = \P(\{\sigma\geq T_0\}\cap \W) \stackrel{\eqref{eq:sigma_equal_T_on_W}}{=} \P(\sigma\geq T_0)>1-\varepsilon_0.\]
Also note that $\|u\|_{\X(\tau)}^p\leq(\frac{4}{3}R)^{-1}$ a.s.\ and therefore $\E\|u\|_{\X(\tau)}^{2p}\leq (\frac{4}{3}R)^{-1} \E\|u\|_{\X(\tau)}^p$.
Thus Step 2 ensures that
$$
\E\|u\|_{\X(\tau)}^p \leq R (\E\|u_0\|_{B_{q,p}^{d/q-1}}^p +  M_1^p )+ \frac{3}{4}\E\|u\|_{\X(\tau)}^p
$$
and hence $
\E\|u\|_{\X(\tau)}^p\leq 4R(  \E\|u_0\|_{B_{q,p}^{d/q-1}}^p+M_1^p)$. Moreover,
$$
\E\|u\|_{\X(\tau)}^{2p}=\E\big[\|u\|_{\X(\tau)}^p \|u\|_{\X(\tau)}^p\big]
\stackrel{\eqref{eq:def_tau_small_T_0}}{\leq} 3\big( \E\|u_0\|_{B_{q,p}^{d/q-1}}^p+ M_1^p\big).
$$
Now the claim follows from Theorem \ref{t:maximal_Stokes_Tord} and \eqref{eq:estimate_F_G_stability_NS}.
\end{proof}

\begin{appendix}
\section{Global well-posedness for $d=2$ with $L^{\infty}$--noise}
\label{s:global_L_2_rough_noise}
The aim of this appendix is to prove global well--posedness for the stochastic Navier-Stokes equations \eqref{eq:Navier_Stokes_generalized} in two dimensions with $\Ls^2$--data and $L^\infty$-noise of transport type.
Although such result is essentially known, our approach based on maximal $L^2$-regularity techniques seems new and provides an improvement of several of the known results. The reader is referred to Remark \ref{r:comparison_2D_other_results} below for a comparison with other approaches.

The results presented in this appendix are used to prove Theorems \ref{t:NS_initial_data_large_two_dimensional} and will be presented more or less independently of the proofs given in the paper.

Here we consider the following stochastic Navier-Stokes equations on a domain $\Dom\subseteq \R^d$
\begin{equation}
\label{eq:Navier_Stokes_appendix}
\left\{
\begin{aligned}
\dd u&=\Big[\div(a\cdot\nabla u)  - \div(u\otimes u ) +f_0(\cdot,u)+\div(f(\cdot,u))\\
&\qquad \qquad \qquad \qquad\qquad \qquad \qquad\qquad -\nabla p + \partial_h \wt{p}+ \partial_{\hp}^2 \wt{p} \Big] \, \dd t \\
& \qquad \quad +\sum_{n\geq 1}\big[(\btwod_{n}\cdot\nabla) u -\nabla \wt{p}_n +\Gforce_n(\cdot,u)\big] \, \dd {w}_t^n, \\
u&=0\  \ \text{ on }\partial\Dom,\\
\div \,u& =0,
\\
u(\cdot,0)&=u_0,&
\end{aligned}\right.
\end{equation}
where $u=(u^k)_{k=1}^d:[0,\infty)\times \O\times \Dom\to \R^d$ is unknown velocity field, $p,\wt{p}_n:[0,\infty)\times \O\times \Dom\to \R$ are the unknown pressures, $(w^n)_{n\geq 1}$ is a sequence of standard independent Brownian motions on a filtered probability space $(\O,\MeasurableP,(\F_t)_{t\geq 0}, \P)$, $(u\otimes u)_{i,j} =u^{i}u^{j}$ and $\div(a\cdot \nabla u)$, $(b_n\cdot \nabla u)$, $\div(f(\cdot,u))$, $\partial_h \wt{p}$, $\partial_{\hp}^2 \wt{p}$ are as in \eqref{eq:def_main_operators_Navier_Stokes_generalized}. Here $\Dom$ can be either an open set in  $\R^d$ (not necessarily bounded) or a compact manifold without boundaries  (e.g.\ the torus $\T^d$). In the latter case the boundary conditions in \eqref{eq:Navier_Stokes_appendix} have to be omitted.

If $\Dom\subseteq\R^d$, then under suitable regularity assumptions $\Dom$ other boundary conditions can be considered,   e.g.\ perfect-slip or Navier boundary conditions (see Remark \ref{r:other_boundary_conditions} below). For brevity, this will not be pursued here. Let us remark that the no-slip condition allows us to avoid any further regularity assumptions on $\Dom$.

\subsection{Assumptions and main result}
We begin by listing the assumption needed in this appendix.

\begin{assumption} Let $d\geq 2$ be an integer.
\label{ass:NS_2D}
\begin{enumerate}[{\rm(1)}]
\item\label{it:NS_2D_ellipticity} For all $n\geq 1$ and $i,j\in \{1,\dots,d\}$, the maps $a^{i,j},b_n^j,h^{j,k}_n,  \hp^j_n:\R_+\times\O\times\Dom\to \R$ are $\Progress\otimes \Borel(\Dom\times \R^d)$-measurable and there exist $M,\ellip>0$ such that, a.s.\ for all $t\in \R_+$, $x\in \Dom$ and $\Upsilon=(\Upsilon_i)_{i=1}^d\in \R^d$,
\begin{align*}
|a^{i,j}(t,x)|+ \sum_{n\geq 1}\Big( |b^j_n(t,x)|^2+|h^{i,j}_n(t,x)|^2 +|\hp^{j}_n(t,x)|^2\Big)&\leq M,\\
\sum_{i,j=1}^d\Big[a^{i,j}(t,x)
-\frac{1}{2}\Big(
\sum_{n\geq 1}\btwod_n^i(t,x) \btwod_n^j(t,x)\Big)\Big]\Upsilon_i\Upsilon_j &\geq \ellip |\Upsilon|^2.
\end{align*}
\item\label{it:NS_2D_growth} For all $j\in \{0,\dots,d\}$ and $n\geq 1$, $\Fd_j,\Gforce_n:\R_+\times\O\times \Dom\times \R^d\to \R^d$ are $\Progress\otimes \Borel(\Dom\times \R^d)$-measurable. Moreover, $\Fd_j(\cdot,0)\in L^{2}_{\loc}([0,\infty);L^2( \O\times \Dom;\R^d))$, $(\Gforce_n(\cdot,0))_{n\geq 1}\in L^{2}_{\loc}([0,\infty);L^2(\O\times \Dom;\ell^2))$ and, a.s.\ for all $t\in \R_+$, $x\in \Tor^d$, $y,y'\in \R^d$,
\begin{align*}
|\Fd_j(t,x,y)-\Fd_j(t,x,y')|
&+\|(\Gforce_n(t,x,y)-\Gforce_n(t,x,y'))_{n\geq 1}\|_{\ell^2}\\
&\lesssim (c_{\Dom}+|y|+|y'|)|y-y'|,
\end{align*}
where $c_{\Dom}=0$ if $\Dom$ is unbounded and $c_{\Dom}=1$ otherwise.
\end{enumerate}
\end{assumption}

To formulate our main result we introduce the needed function spaces.  Set
\begin{equation*}
H^1_0(\Dom):=\overline{C^{\infty}_0(\Dom)}^{H^1(\Dom)}, \quad \text{ and }\quad
H^1_0(\Dom;\R^d):=(H^1_0(\Dom))^d.
\end{equation*}
Next we introduce the Helmholtz projection. For each $F\in L^2(\Dom;\R^d)$, let $\psi_{F}$ be the (up to a constant) unique distribution such that $\psi_F \in L^2_{\loc}(\Dom)$, $\nabla\psi_{F}\in L^2(\Dom;\R^d)$ and
\begin{equation}
\label{eq:variational_solution_helmholtz_weak}
\int_{\Dom}\nabla \psi_{F} \cdot \nabla \varphi \,\dd x=\int_{\Dom} F\cdot \nabla\varphi \,\dd x \ \ \text{ for all }\varphi \in L^2_{\loc}(\Dom) \text{ s.t.\ }\nabla \varphi\in L^2(\Dom;\R^d).
\end{equation}
The existence of a solution $\psi_F\in L^2_{\loc}(\Dom)$ to \eqref{eq:variational_solution_helmholtz_weak} with $\|\nabla \psi_F\|_{L^2(\Dom;\R^d)}\leq \|F\|_{L^2(\Dom;\R^d)}$ follows from the Riesz representation theorem. In case $\Dom$ is a bounded Lipschitz domain, then one can also obtain $\psi_F\in L^2(\Dom)$. Note that $\nabla\psi_F\in L^2(\Dom;\R^d)$ is uniquely determined by $F$, and that \eqref{eq:variational_solution_helmholtz_weak} is the weak formulation of
\begin{equation}
\label{eq:Helmholtz_projection_variational}
\left\{\begin{aligned}
\Delta \psi_{F} &=\div F & \ \text{ on }&   \Dom,\\
\partial_{\nu} \psi_{F} & =F\cdot \nu  & \ \text{ on } & \partial\Dom.
\end{aligned}\right.
\end{equation}
The Helmholtz projection is given by
$
\p F:= F- \nabla \psi_{F}.
$
Using \eqref{eq:variational_solution_helmholtz_weak}, one sees that $\p$ is an orthonormal projection on $L^2(\Dom;\R^d)$. Finally, we let
$$
\Ls^2(\Dom):=\p(L^2(\Dom;\R^d)), \ \  \
\Hs_0^1(\Dom):=H_0^1(\Dom;\R^d)\cap \Ls^2(\Dom),
\ \ \
 \Hs^{-1}(\Dom):=(\Hs_0^1(\Dom))^*.
$$
Next we introduce $L^2$-solutions to \eqref{eq:Navier_Stokes_generalized} which, roughly speaking, are \emph{strong} in the probabilistic sense and \emph{weak} in the analytic sense. As in Subsection \ref{ss:main_assumption_definition}, we view \eqref{eq:Navier_Stokes_appendix} as a problem for the unknown $u$. To this end, for $h^{j,k}_n,\hp^j_n$ as in Assumption \ref{ass:NS_2D}\eqref{it:NS_2D_ellipticity}, let us set $
\wt{f}(u):= (\wt{f}^{\,k}(u))_{k=1}^d$ where
\begin{align}
\label{eq:def_ftilde_L2}
\wt{f}^{\,k}(u)
&=
\sum_{n\geq 1} \Big( \Big[(I-\p)\big[(b_n \cdot \nabla )u +\Gforce_n(\cdot,u)\big]\Big]\cdot h^{\cdot,k}_n\\
\nonumber
& + \sum_{n\geq 1}\div \Big(\hp_{n}(I-\p)\big[(b_n \cdot \nabla )u +\Gforce_n(\cdot,u)\big]^k\Big).
\end{align}
As below \eqref{eq:wtfbg}, we have $
\wt{f}(\cdot,u) = \partial_h \wt{p}+\partial_{\hp}^2 \wt{p}$.

\begin{definition}
\label{def:L_2_solutions}
Let Assumption \ref{ass:NS_2D} be satisfied.
\begin{enumerate}[{\rm(1)}]
\item Let $\sigma$ be a stopping time and let $u:[0,\sigma)\times \O\to \Hs_0^{1}(\Dom)$ be a stochastic process. We say that $(u,\sigma)$ is an \emph{$L^2$-local solution} to \eqref{eq:Navier_Stokes_appendix} if there exists a sequence of stopping times $(\sigma_{\ell})_{\ell\geq 1}$ for which the following hold:
\begin{itemize}
\item $\sigma_\ell\leq \sigma$ a.s.\ for all $\ell\geq 1$ and $\lim_{\ell\to \infty} \sigma_{\ell}=\sigma$ a.s.;
\item For all $\ell\geq 1$, the process $
\one_{[0,\sigma_{\ell}]\times \O}u
$ is progressively measurable;
\item $u\in L^2(0,\sigma_{\ell};\Hs_0^{1}(\Dom))\cap C([0,\sigma_{\ell}];\Ls^2(\Dom))$ a.s.\ and for all $i,j\in \{1,\dots,d\}$
\begin{equation}
\label{eq:integrability_condition_L_2}
\begin{aligned}
f_0^i(u) +f^i_j(u)+\wt{f}^{\,i}(u) - u^i u^j\in L^2(0,\sigma_{\ell};L^{2}(\Dom;\R^d))\text{ a.s.,}&\\
\big((\Gforce_{n}^j(\cdot,u)\big)_{n\geq 1} \in  L^2(0,\sigma_{\ell};L^2(\Dom;\ell^2))\text{ a.s.\ }&
\end{aligned}
\end{equation}
\item a.s.\ for all $n\geq 1$, $t\in [0,\sigma_{\ell}]$ and $\psi\in \Hs_0^{1}(\Dom)$ the following identity holds
\begin{equation}
\label{eq:integral_equation_Navier_Stokes_weak}
\begin{aligned}
&\int_{\Dom} (u^k(t)-u^k_{0}) \psi^k\,\dd x
+ \int_0^t \int_{\Dom} a^{i,j} \partial_j u^k \partial_i \psi^k\,\dd x\, \dd s\\
&= \int_0^t \int_{\Dom} (u^j u^k +f^k_j(u))\partial_j \psi^k \,\dd x \,\dd s
+ \int_0^t \int_{\Dom} (f_0^k(u) +\wt{f}^{\,k}(u))\psi^k \,\dd x\,\dd s \\
&+\sum_{n\geq 1}
\int_0^t \one_{[0,\sigma_n]}\int_{\Dom}  \Big(\btwod_n^j \partial_j u^k  +g_{n}^k(\cdot,u)\Big)\psi^k\, \dd x \,\dd w^n_s,
\end{aligned}
\end{equation}
where we used the Einstein summation convention for the sums over $i,j,k\in \{1, \ldots, d\}$.
\end{itemize}
%%%%%%
\item An $L^2$-local solution $(u,\sigma)$ to \eqref{eq:Navier_Stokes_appendix} is called an \emph{$L^2$-maximal solution} if for any other local solution $(v,\tau)$ to \eqref{eq:Navier_Stokes_appendix} one has $\tau\leq \sigma$ a.s.\ and $v=u$ on $[0,\tau)\times \O$.
%%%%
\end{enumerate}
\end{definition}

Note that $L^2$-maximal solutions are \emph{unique}. Moreover, \eqref{eq:integrability_condition_L_2} ensures that all the integrals and series appearing in \eqref{eq:integral_equation_Navier_Stokes_weak} are well-defined and convergent.

Next we present the main result of this appendix.
\begin{theorem}[Global solutions in 2D]
\label{t:2D_global_rough_noise}
Assume that $d=2$ and that  $\Dom\subseteq \R^2$ is an open set, or a compact $d$-dimensional manifold without boundary (e.g.\ the torus $\T^d$).
Let Assumption \ref{ass:NS_2D} be satisfied and that for all $n\geq 1 $, $ j\in \{1,\dots,d\}$
\begin{equation}
\label{eq:theta_equal_b}
 \hp^j_n=\alpha_n b_n^j \ \  \text{ for some } \ \alpha_n\in [- \tfrac{1}{2},\infty) .
\end{equation}
Then for each
$
u_0\in L^0_{\F_0}(\O;\Ls^2(\Dom))
$
there exists an $L^2$-maximal solution $(u,\sigma)$ to \eqref{eq:Navier_Stokes_appendix} such that $\sigma>0$ a.s.\ and
\begin{equation}
\label{eq:L_2_regularity_up_to_0}
u\in C([0,\sigma);\Ls^2(\Dom))\cap L^2_{\loc}([0,\sigma);\Hs_0^{1}(\Dom))\ \ \ \text{a.s. }
\end{equation}
Moreover, if there exist $\Xi\in L^2_{\loc}([0,\infty);L^2(\O\times \Dom))$ and $C>0$ such that, a.s.\ for all $j\in \{0,1,2\}$, $t\in \R_+$, $x\in \Dom$ and $y\in\R$,
\begin{equation}
\label{eq:NS_2D_sublinearity}
|\Fd_j(t,x,y)|+
\|(\Gforce_n(t,x,y))_{n\geq 1}\|_{\ell^2}\leq \Xi (t,x)+C |y|,
\end{equation}
then the following hold:
\begin{enumerate}[{\rm(1)}]
\item\label{it:NS_2D_global_weak_solution} The $L^2$-maximal solution $(u,\sigma)$ to \eqref{eq:Navier_Stokes_appendix} is \emph{global} in time, i.e.\ $\sigma=\infty$ a.s.
\item\label{it:NS_2D_global_weak_solution_estimate} If $u_0\in L^2(\O;\Ls^2(\Dom))$, then for each $T\in (0,\infty)$ there exists $C_T>0$ independent of $u,u_0$ such that
$$
\E\Big[\sup_{t\in [0,T]}\|u(t)\|_{L^2(\Dom)}^2\Big] +
\E\int_0^T \|\nabla u(t)\|_{L^2(\Dom)}^2 \,\dd t
\leq C_T(1+ \E\|u_0\|_{L^2(\Dom)}^2).
$$
\end{enumerate}
\end{theorem}

The proof of Theorem \ref{t:2D_global_rough_noise} will be given in Subsection \ref{ss:proof_theorem_appendix} below.
The case $\alpha_n\equiv -\frac{1}{2}$ of \eqref{eq:theta_equal_b} corresponds to the Stratonovich formulation of transport noise, see the comments around \eqref{eq:stratonovich_correction}. The proof of the above result shows that the range $\alpha_n\in [ -\frac{1}{2},\infty)$ can be improved to $\alpha_n \in (\alpha_0,\infty)$ for some $\alpha_0<- \frac{1}{2}$ depending only on $(\ellip,M)$ in Assumption \ref{ass:NS_2D}\eqref{it:NS_2D_ellipticity}. We leave the details to the reader.

Note that under further regularity conditions on the coefficients $a^{ij}$, $b^j$ and $h^{ij}$, the above result can be improved and this is the content of Theorem \ref{t:NS_initial_data_large_two_dimensional} combined with Theorems \ref{t:NS_critical_local} and \ref{t:NS_high_order_regularity}.
In the following remark we give a comparison with the existing results.

\begin{remark}
\label{r:comparison_2D_other_results}
The existence of $L^2$-maximal solutions to \eqref{eq:Navier_Stokes_appendix} with nonlinearities of quadratic growth (see Assumption \ref{ass:NS_2D}\eqref{it:NS_2D_growth}) seems new. \eqref{it:NS_2D_global_weak_solution_estimate}  should be compared with \cite[Theorem 2.2]{MR05} if $\Dom=\R^d$ and \cite[Section 6 and Corollary 7.7]{BM13_unbounded} if $\Dom$ is an unbounded domain. Some earlier related results appeared in \cite{BCF91,BCF92} using some smoothness of $b$.
Unlike in the above references we do not need any regularity assumption on $b$ besides boundedness and measurability.
The regularity assumption on $b$ are needed in \cite{MR05,BM13_unbounded} to construct martingale solutions to \eqref{eq:Navier_Stokes_appendix} for general $d\geq  2$. The case $d=2$ is only considered afterwards.

In case $f$ and $g$ are globally Lipschitz, Theorem \ref{t:2D_global_rough_noise} also follows from the monotone operator approach (see e.g.\ \cite[Chapter 4]{LR15}). The extensions to the non-Lipschitz setting in \cite[Chapter 5]{LR15} are not applicable since transport noise is not considered there.
\end{remark}

Below, we write $L^2$, $\Hs_0^1$, $L^2(\ell^2)$ instead of $L^2(\Dom)$, $\Hs_0^1(\Dom)$, $L^2(\Dom;\ell^2)$ etc.\ for brevity.

\subsection{Maximal $L^2$-regularity for the turbulent Stokes system}
In this subsection we consider the linear part of \eqref{eq:Navier_Stokes_appendix} and prove stochastic maximal $L^2$-regularity estimates which will be used for Theorem \ref{t:2D_global_rough_noise}. We consider the \emph{turbulent Stokes system} on $\Dom\subseteq \R^d$:
\begin{equation}
\label{eq:turbulent_stokes_couple_appendix}
\left\{
\begin{aligned}
\dd u +A_{\Stok}u \,\dd t &=f\, \dd t +\sum_{n\geq 1} (B_{\Stok,n} u+ g_n )\, \dd w_t^n, \\
u(s)& =0.
\end{aligned}\right.
\end{equation}
Here $s\in [0,\infty)$, $(A_{\Stok},B_{\Stok})=(A_{\Stok},(B_{\Stok,n})_{n\geq 1}):\R_+\times \O\to \calL\big(\Hs_0^{1},
\Hs^{-1}\times \Ls^2(\ell^2)\big)$ is the \emph{turbulent Stokes couple}, which for all $v\in \Hs_0^1$ and $\psi\in \Hs^{1}_0$, is given by
\begin{align*}
\l \psi, A_{\Stok} v \r&:=  \sum_{i,j=1}^d\int_{\Dom} a^{i,j} \partial_j v^k \partial_i \psi^k\,\dd x
+\sum_{n\geq 1}\sum_{j=1}^d \int_{\Dom} \hp^j_n \big((I-\p)[(b_n\cdot\nabla) v]\big)^k \partial_j \psi^k\,\dd x 
\\
&\qquad - \sum_{n\geq 1}
\sum_{j=1}^d \int_{\Dom} \Big[(I-\p)\big[(b_n \cdot \nabla )v ]\Big] \cdot h^{\cdot,j}_n \psi^j \,\dd x,\\
B_{\Stok}(t) v&:=\big(\p[(b_n\cdot \nabla ) v]\big)_{n\geq 1}, \qquad \qquad\qquad h^{\cdot,j}_n :=(h^{i,j}_n)_{i=1}^d,
\end{align*}
where $\l \cdot,\cdot\r$ is the pairing in the duality $(\Hs_0^{1})^*=\Hs^{-1}$ and $\Ls^2(\ell^2):=\p(L^2(\Dom;\ell^2(\N_{\geq 1};\R^d)))$. Note that, if Assumption \ref{ass:NS_2D}\eqref{it:NS_2D_ellipticity} holds, then
$$
\|A_{\Stok}v\|_{\Hs^{-1}}
+
\|(B_{\Stok,n} v)_{n\geq 1}\|_{\Ls^2(\ell^2)} \lesssim M \|v\|_{\Hs^1_0}.
$$
Recall that $\Progress$ denotes the progressive $\sigma$-algebra.
For a stopping time $\tau$, we say that $u\in L^2_{\Progress}(( 0,\tau)\times \O;\Hs_0^1)$ is a \emph{strong solution} to  \eqref{eq:turbulent_stokes_couple_appendix} (on $[0,\tau]\times \O$) if a.s.\ for all $t\in [s,\tau)$
$$
u(t)+\int_s^{t} A_{\Stok}(r)u(r)\,\dd r=\int_{s}^t f(r)\,\dd r
+\sum_{n\geq 1}\int_{s}^t (B_{\Stok,n}(r)u(r)+g_n(r))\, \dd w_r^n.
$$

Next we show that the turbulent Stokes couple $(A_{\Stok},B_{\Stok})$ has stochastic maximal $L^2$-regularity on $[s,T]$.
\begin{proposition}[Maximal $L^2$-regularity]
\label{prop:maximal_L_2_regularity}
Let $d\geq 2$. Let Assumption \ref{ass:NS_2D}\eqref{it:NS_2D_ellipticity} be satisfied.
Assume that  $\Dom\subseteq \R^d$ is an open set, or a compact $d$-dimensional manifold without boundary (e.g.\ the torus $\T^d$).
Let $0\leq s<T<\infty$. Then for all
$$
f\in L^2_{\Progress}(( s,T)\times \O;\Hs^{-1}), \quad \text{ and }\quad
g\in L^2_{\Progress}(( s,T)\times \O;\Ls^2(\ell^2)),
$$
there exists a unique strong solution $u\in L^2_{\Progress}(( s,T)\times \O;\Hs_0^1)\cap L^2(\O;C([s,T];\Ls^2))$. Moreover, for any stopping time $\tau:\O\to [s,T]$ and any strong solution $u$ to \eqref{eq:turbulent_stokes_couple_appendix} on
$[ s,\tau]\times \O$ one has
\begin{equation*}
\|u\|_{L^2(( s,\tau)\times \O;\Hs_0^1)}+\|u\|_{L^2(\O;C([s,\tau];\Ls^2))}\lesssim
\|f\|_{ L^2(( s,\tau)\times \O ;\Hs^{-1})}+
\| g\|_{ L^2((s,\tau)\times \O;\Ls^2(\ell^2))}
\end{equation*}
where the implicit constant is independent of $(f,g)$.
\end{proposition}

\begin{proof}
This follows from the monotone operator approach to SPDEs (see \cite[Theorem 4.2.4]{LR15}). Indeed, it suffices to check the coercivity assumption. To this end, note that, a.s.\ for all $t\in \R_+$ and $v\in \Hs^1_0$,
\begin{align*}
 \int_{\Dom} \hp^j_n \big((I-\p)[(b_n\cdot\nabla) v]\big)^k \partial_j v^k\,\dd x
\stackrel{\eqref{eq:theta_equal_b}}{=}\alpha_n   \big\|(I-\p) [(b_n\cdot\nabla) v]\big\|_{L^2}^2.
\end{align*}
Since $\inf_{n\geq 1}\alpha_n\geq - \frac{1}{2}$, we have, a.s.\ for all $t\in \R_+$ and $v\in \Hs^1_0$,
\begin{align*}
&-\l v, A_{\Stok} v \r + \frac{1}{2}
\sum_{n\geq 1}\|B_{\Stok,n}v\|_{L^2}^2 \\
&\stackrel{(i)}{\leq} -\sum_{i,j=1}^d\int_{\Dom} a^{i,j}\partial_i v\cdot \partial_j v \, \dd x
+\frac{1}{2} \sum_{n\geq 1}\|(I-\p) [(b_n\cdot\nabla) v]\|_{L^2}^2\\
&+C\|h\|_{L^{\infty}(\ell^2)}\|\nabla v\|_{L^2}\|v\|_{L^2}+
\frac{1}{2}\sum_{n\geq 1}\| \p (b_n \cdot \nabla )v\|_{L^2}^2
\\
&\stackrel{(ii)}{\leq }-\ellip \|\nabla v\|_{L^2}^2 + \frac{\ellip}{2}  \|\nabla v\|_{L^2}^2 + C_{\ellip} M\|v\|_{L^2}^2
=- \frac{\ellip}{2}  \|\nabla v\|_{L^2}^2 + C_{\ellip} M\|v\|_{L^2}^2
\end{align*}
where in $(i)$ we used that $\p$ is an orthogonal projection and in $(ii)$ we used Assumption \ref{ass:NS_2D}\eqref{it:NS_2D_ellipticity}.
%%%%
Hence the coercivity assumption in \cite[Theorem 4.2.4]{LR15} is satisfied and this concludes the proof.
\end{proof}

We conclude by pointing out another proof of the above result.
\begin{remark}
\label{r:SMR_2_appendix}
Instead of the monotone operator approach one can also prove Proposition \ref{prop:maximal_L_2_regularity} via the method of continuity for stochastic maximal regularity (see \cite[Proposition 3.13]{AV19_QSEE_2}). Indeed, one can reason as in the proof of Theorem \ref{t:maximal_Stokes_Tord} and the required a priori estimate can be obtained from It\^o's formula for $u\mapsto \|u\|_{L^2}^2$.
\end{remark}

\subsection{Proof of Theorem \ref{t:2D_global_rough_noise}}
\label{ss:proof_theorem_appendix}
To estimate the nonlinearities in \eqref{eq:Navier_Stokes_appendix} we need the following consequence of Sobolev embedding. Here $[\cdot,\cdot]_{\theta}$ denotes the complex interpolation functor, see e.g.\ \cite{BeLo}.

\begin{lemma}[Sobolev embeddings]
\label{l:Sob_emb}
Let $d\geq 2$.
Assume that  $\Dom\subseteq \R^d$ is an open set, or a compact $d$-dimensional manifold without boundary (e.g.\ the torus $\T^d$).
Let $\theta\in (0,1)$ and $q\in[2,\infty)$ be such that $\theta-\frac{d}{2}\geq -\frac{d}{q}$. Then
$$
[\Hs^{-1}(\Dom),\Hs_0^1(\Dom)]_{\frac{1+ \theta}{2}}\embed L^q(\Dom).
$$
\end{lemma}

\begin{proof}
From general theory on coercive forms (see \cite[Section 5.5.2]{ArendtHandbook}) we obtain $[\Hs^{-1},\Hs_0^1]_{1/2}=\Ls^2$.
Therefore, by reiteration for complex interpolation (see \cite[Theorem 4.6.1]{BeLo})
\begin{align*}
[\Hs^{-1},\Hs_0^1]_{(1+\theta)/2}=[\Ls^2,\Hs_0^1]_{\theta} \hookrightarrow  [L^2,H^1_0]_{\theta}.
\end{align*}
Using the boundedness of the zero-extension operator ${\rm E}_0:L^2(\Dom)\to L^2(\R^d)$ and ${\rm E}_0:H^1_0(\Dom)\to H^1(\R^d)$, and Sobolev embedding (see \cite[Theorem 6.4.5]{BeLo}),
$$
\|f\|_{L^q} \leq \|{\rm E}_{0}f\|_{L^q(\R^d)}\lesssim  \|{\rm E}_{0} f\|_{[L^2(\R^d);H^1(\R^d)]_{\theta}}\leq
\|f\|_{[L^2;H^1_0]_{\theta}}
\leq \|f\|_{[\Ls^2,\Hs^1_0]_{\theta}}.
$$
\end{proof}

\begin{proof}[Proof of Theorem \ref{t:2D_global_rough_noise}]
The proof is divided into several steps.

\emph{Step 1: There exists an $L^2$-maximal solutions $(u,\sigma)$ to \eqref{eq:Navier_Stokes_appendix} such that $\sigma>0$ a.s.\ and that \eqref{eq:L_2_regularity_up_to_0} holds}. By Definition \ref{def:L_2_solutions} the claim of this step follows from \cite[Theorem 4.9]{AV19_QSEE_1}  with the choice $p=2$, $\a=0$, $X_0=\Hs^{-1}$, $X_1=\Hs_0^{1}$, $H=\ell^2$ and for $v\in \Hs_0^1$, $\psi\in \Hs_0^1$,
\begin{equation}
\label{eq:choice_ABFG_L_2}
\begin{aligned}
A v=A_{\Stok} v, \qquad B v=\B_{\Stok} v,\qquad G(v)=\big(\p [\Gforce_n(\cdot,v)]\big)_{n\geq 1},&\\
\l \psi, F(\cdot,v)\r= \int_{\Dom} (u^j u^i +f^i_j(u))\partial_j \psi^i \, \dd x\, \dd s
+ \int_{\Dom} (f_0^i(u) +\wt{f}^{\,i }(u))\psi^i \,\dd x\, \dd s,&
\end{aligned}
\end{equation}
where $A_{\Stok},B_{\Stok}$ and $\wt{f}$ are as below \eqref{eq:turbulent_stokes_couple_appendix} and in \eqref{eq:def_ftilde_L2}, respectively. Indeed, it remains to check the assumptions of \cite[Theorem 4.9]{AV19_QSEE_1}. The stochastic maximal $L^2$-regularity follows from Proposition \ref{prop:maximal_L_2_regularity}. Next we check hypothesis \cite[(HF)-(HG)]{AV19_QSEE_1}.
To this end, note that by Assumption \ref{ass:NS_2D}\eqref{it:NS_2D_growth} we have, for all $v,v'\in \Hs_0^1$,
\begin{equation}
\begin{aligned}
\label{eq:NS_2D_L_4_estimates}
&\|F(\cdot,v)-F(\cdot,v')\|_{\Hs^{-1}}
+ \|G(\cdot,v)-G(\cdot,v')\|_{\Ls^2(\ell^2)}\\
&\qquad
\lesssim (c_{\Dom}+\|v\|_{L^4}+\|v'\|_{L^4})\|v-v'\|_{L^4}\\
&\qquad
\lesssim (c_{\Dom}+\|v\|_{[\Hs^{-1},\Hs_0^1]_{\frac{3}{4}}}+\|v'\|_{[\Hs^{-1},\Hs_0^1]_{\frac{3}{4}}})
\|v-v'\|_{[\Hs^{-1},\Hs_0^1]_{\frac{3}{4}}}
\end{aligned}
\end{equation}
where in the last step we used Lemma \ref{l:Sob_emb} with $d=2$.
Thus \cite[(HF)-(HG)]{AV19_QSEE_1} are satisfied with $m_G=m_F=1$, $\rho_j=1$, $\beta_j=\varphi_j =\frac{3}{4}$ where $j\in \{1,2\}$. Hence the claim of this step follows from \cite[Theorem 4.9]{AV19_QSEE_1}.

%%%%%%%

It remains to prove \eqref{it:NS_2D_global_weak_solution}-\eqref{it:NS_2D_global_weak_solution_estimate} in Theorem \ref{t:2D_global_rough_noise} in case \eqref{eq:NS_2D_sublinearity} holds. To this end, it is enough to show that if $(u,\sigma)$ is an $L^2$-maximal solution to \eqref{eq:Navier_Stokes_appendix} and $u_0\in L^2(\O;L^2)$, then for each $T\in (0,\infty)$
\begin{equation}
\label{eq:energy_estimate_L_2}
\E\Big[\sup_{t\in [0,T\wedge\sigma)}\|u(t)\|_{L^2}^2\Big] +
\E\int_0^{T\wedge\sigma} \|\nabla u(t)\|_{L^2}^2 \,\dd t
\lesssim_T 1+\E\|u_0\|_{L^2}^2.
\end{equation}
Indeed, if \eqref{eq:energy_estimate_L_2} is true, then in Step 2 we prove that \eqref{it:NS_2D_global_weak_solution} holds and therefore \eqref{it:NS_2D_global_weak_solution_estimate} follows from \eqref{it:NS_2D_global_weak_solution} and \eqref{eq:energy_estimate_L_2}.

\emph{Step 2: If \eqref{eq:energy_estimate_L_2} holds, then \eqref{it:NS_2D_global_weak_solution} is true}.
Recall that \eqref{eq:Navier_Stokes_generalized} is understood as a stochastic evolution equations on $\Hs^{-1}$ with the choice  \eqref{eq:choice_ABFG_L_2}. Thus by localization we may assume $u_0\in L^{\infty}(\O;\Ls^2)$ (see \cite[Proposition 4.13]{AV19_QSEE_2}). Hence, for all $0<\varepsilon<T<\infty$,
$$
\P(\varepsilon<\sigma<T)\stackrel{\eqref{eq:energy_estimate_L_2}}{=}
\P\Big(\varepsilon<\sigma<T,\,\sup_{t\in [0,\sigma)} \|u(t)\|_{\Ls^2}+\|u\|_{L^2(0,\sigma;\Hs_0^{1})}<\infty\Big)
=0
$$
where in the last equality we used \cite[Theorem 4.10(3)]{AV19_QSEE_2}.
Since $0<\varepsilon<T<\infty$ are arbitrary and $\sigma>0 $ a.s.\ by Step 1, the previous yields $\sigma=\infty$ as desired.

%%%%%
To prove \eqref{eq:energy_estimate_L_2}, we need a localization argument. Fix $T\in (0,\infty)$.
For each $j\geq 1$, let $\sigma_j$ be the stopping time given by
\begin{equation}
\label{eq:sigma_j_energy_estimate}
\sigma_j:=\inf\big\{t\in [0,\sigma)\,:\,\|\nabla u\|_{L^2(0,t;L^2)}+\|u(t)\|_{L^2}\geq j\big\}\wedge T,
\end{equation}
where $\inf\emptyset:=\sigma$. By Step 1, \eqref{eq:L_2_regularity_up_to_0} holds and therefore $\lim_{j\to \infty} \sigma_j=\sigma$ a.s. Thus to obtain \eqref{eq:energy_estimate_L_2}, by Gronwall and Fatou's lemmas, it is enough to show the existence of a constant $C>0$ independent of $j,u,u_0$ such that
\begin{equation}
\label{eq:Gronwall_estimate_NS_2D}
\E y(t)\leq C(1+\E \|u_0\|_{L^2}^2)+C \int_0^t \E y(s)\, \dd s, \ \ t\in [0,T],
\end{equation}
where
\begin{equation}
\label{eq:def_y_NS_2D}
 y(t)=\sup_{r\in [0,t\wedge\sigma_j)}\|u(t)\|_{L^2}^2+\int_{0}^{t\wedge \sigma_j}\|\nabla u(s)\|^2_{L^2} \, \dd s.
\end{equation}
In the remaining steps $j\geq 1$ is fixed and we set $\stopp:=\sigma_j$.

\emph{Step 3: We apply It\^o's formula to obtain the identity \eqref{eq:Ito_energy_inequality} below}.
Let us begin by collecting some useful facts.
Note that $\|\nabla u\|_{L^2(0,\stopp;L^2)}+\sup_{t\in [0,\stopp)}\|u(t)\|_{L^2}\leq j$ a.s.\ by \eqref{eq:sigma_j_energy_estimate} (recall $\stopp=\sigma_j$), and $\|u\|_{L^4(0,\stopp;L^4)}\lesssim j$ a.s.\ by Lemma \ref{l:Sob_emb}  with $d=2$ and standard interpolation inequalities
\begin{equation*}
L^2(0,t;\Hs_0^1)\cap L^{\infty}(0,t;\Ls^2)\hookrightarrow L^4(0,t;[\Ls^2,\Hs_0^{1}]_{1/2})
\hookrightarrow L^4(0,t; L^4), \ \text{ for } t>0.
\end{equation*}
Note that the embedding constants in the above can be made independent of $t>0$.

Combining the previous observations with \eqref{eq:NS_2D_L_4_estimates}, we have
\begin{equation}
\label{eq:integrability_F_G_L_2}
\one_{[ 0,\stopp]\times \O}  F(\cdot,u)
\in L^2((0,T)\times \O;\Hs^{-1}),\quad
\one_{[0,\stopp]\times \O }  G(\cdot,u)
\in L^2((0,T)\times \O;\Ls^2(\ell^2)),
\end{equation}
where $F$ and $G$ are as in \eqref{eq:choice_ABFG_L_2}.
By \eqref{eq:integrability_F_G_L_2}, It\^o's formula \cite[Theorem 4.2.5]{LR15}, and Assumption \ref{ass:NS_2D}\eqref{it:NS_2D_ellipticity}, arguing as in the proof of Proposition \ref{prop:maximal_L_2_regularity}, one can check that, a.s.\ for all $t\in [0,T]$,
\begin{equation}
\begin{aligned}
\label{eq:Ito_energy_inequality}
&\|\usigma(t)\|_{L^2}^2- \|u_0\|_{L^2}^2+\frac{\ellip}{2} \int_0^{t} \one_{[0,\stopp]}\|\nabla u\|_{L^2}^2\, \dd s\\
&\leq \int_0^t \one_{[0,\stopp]} \Big[2\big|(\Fd_0(u),u)_{L^2}|+2\sum_{k=1}^2|(\Fd_k (u),\partial_k u)_{L^2}|
+\|\Gforce(u)\|_{L^2(\ell^2)}^2\Big]\, \dd s \\
&+\sum_{n\geq 1} \int_0^t \one_{[0,\stopp]} \big((\btwod_n\cdot \nabla) u, \p\Gforce_n(u)\big)_{L^2} \, \dd s \\
&+ 2\sum_{n\geq 1}\int_0^t \one_{[0,\stopp]}(\Gforce_n(u),u)_{L^2}\, \dd w^n_s
+2\sum_{n\geq 1}\int_0^t \one_{[0,\stopp]} \big((\btwod_n\cdot \nabla) u,u\big)_{L^2}\, \dd w^n_s\\
&=: I_t+II_t+III_t+IV_t
\end{aligned}
\end{equation}
where $\usigma:=u(\cdot\wedge \stopp)$,
$(f,g)_{L^2}=\sum_{k=1}^2\int_{\Dom} f^k g^k\,\dd x$ and we used the cancellation
\begin{equation}\label{eq:boundarycondition}
\sum_{i,j=1}^2 \int_{\Dom} u^i u^j \partial_{j} u^i\,\dd x=0, \ \ \text{ for all } u\in \Hs_0^{1}.
\end{equation}
The above follows from $\div(u) = 0$ and by integrations by parts and using that $u\cdot \nu=0$ on $\partial\Dom$ due to $\p u= u$ and \eqref{eq:Helmholtz_projection_variational}, where $\nu$ denotes the outer normal vector on $\partial \Dom$.

\textit{Step 4: Proof of \eqref{eq:Gronwall_estimate_NS_2D} with $C$ independent of $u,u_0,j$}. Recall that $\stopp:=\sigma_j$.
Due to \eqref{eq:NS_2D_sublinearity}, one can readily check that there exists $C>0$ independent of $j,u,u_0$ such that,
for all $t\in [0,T]$,
\begin{equation}
\label{eq:estimate_I_II_NS_2D}
\E\Big[ \sup_{s\in [0,t]}| I_s|\Big]\leq \frac{\ellip}{4}\E\int_{0}^t \one_{[0,\stopp]} \|\nabla u(s)\|_{L^2}^2 \,\dd s+
C\Big(1+ \E\int_{0}^t \one_{[0,\stopp]}\|u(s)\|_{L^2}^2 \, \dd s\Big).
\end{equation}
By Assumption \ref{ass:NS_2D}\eqref{it:NS_2D_ellipticity}, we have
\begin{equation}
\begin{aligned}
\label{eq:estimate_III_NS_2D}
\E \Big[ \sup_{s\in [0,t]}| II_s|\Big]
&\leq\E\int_{0}^{t}\one_{[0,\stopp]} \|(\btwod_n(s)\cdot \nabla )u(s)\|_{L^2(\ell^2)}\|u(s)\|_{L^2} \, \dd s\\
&\leq M \E \int_{0}^{t}\one_{[0,\stopp]} \|\nabla u(s)\|_{L^2}\|u(s)\|_{L^2} \, \dd s\\
&\leq \frac{\ellip}{4} \E \int_{0}^{t}\one_{[0,\stopp]} \|\nabla u(s)\|_{L^2}^2 +
C_{M,\ellip}\E\int_{0}^{t}\one_{[0,\stopp]}\|u(s)\|_{L^2}^2 \, \dd s.
\end{aligned}
\end{equation}
Therefore by taking expectations in \eqref{eq:Ito_energy_inequality}, using $\E[III_{t}]=\E[IV_t]=0$, and \eqref{eq:estimate_I_II_NS_2D}-\eqref{eq:estimate_III_NS_2D}, one obtains
\begin{equation}
\label{eq:intermediate_estimate_nabla_L2}
\E\int_{0}^{t} \one_{[0,\stopp]} \|\nabla u(s)\|_{L^2}^2 \, \dd s \leq
C\Big(1+t+\E\int_{0}^{t} \one_{[0,\stopp]}\|u(s)\|_{L^2}^2 \, \dd s\Big)
\end{equation}
where $C$ is independent of $u,u_0,j$.

Next we take $\E\big[\sup_{s\in [0,t]}|\cdot|\big]$ in \eqref{eq:Ito_energy_inequality}. Due to \eqref{eq:estimate_I_II_NS_2D}-\eqref{eq:estimate_III_NS_2D}, it remains to estimate $III$ and $IV$.
By the Burkholder-Davis-Gundy inequality, for all $t\in [0,T]$,
\begin{align*}
\E\Big[\sup_{s\in [0,t]}| III_s|\Big]
&\lesssim
\E\Big[\int_{0}^t \one_{[0,\stopp]} \|(\btwod_n(s)\cdot \nabla )u(s)\|_{L^2(\ell^2)}^2
\|u(s)\|_{L^2}^2 \, \dd s\Big]^{1/2}\\
&\leq  M
\E\Big[\Big(\sup_{s\in [0,\stopp\wedge t)} \|u(s)\|^2_{L^2}\Big)^{1/2}
\Big(\int_{0}^t \one_{[0,\stopp]} \|\nabla u(s)\|_{L^2}^2 \, \dd s\Big)^{1/2}\Big]\\
&\leq \frac{1}{4}\E y(t)+C_{M}\E \int_{0}^t \one_{[0,\stopp]} \|\nabla u(s)\|_{L^2}^2 \, \dd s\\
&\stackrel{\eqref{eq:intermediate_estimate_nabla_L2}}{\leq} \frac{1}{4}\E y(t)
+ C_{M}'\Big(1+\E\int_{0}^t \one_{[0,\stopp]} \| u(s)\|_{L^2}^2 \, \dd s\Big).
\end{align*}
where $y$ is as in \eqref{eq:def_y_NS_2D}. A similar argument applies to $IV$. Indeed, by \eqref{eq:NS_2D_sublinearity}, for all $t\in [0,T]$,
\begin{align*}
\E\Big[\sup_{s\in [0,t]}| IV_s|\Big]
&\leq
\E\Big[\Big(\sup_{s\in [0,\stopp\wedge t)} \|u(s)\|_{L^2}^2\Big)^{1/2}
\Big(\int_{0}^t\one_{[ 0,\stopp]\times \O} \|\Gforce(s,u(s))\|_{L^2(\ell^2)}^2 \, \dd s\Big)^{1/2}\Big]\\
&\leq \frac{1}{4}\E y(t)+ {C}\Big(1 +\E\int_{0}^t \one_{[ 0,\stopp]\times \O} \|u(s)\|_{L^2}^2 \, \dd s\Big).
\end{align*}
Taking $\E\big[\sup_{s\in [0,t]}|\cdot|\big]$ in \eqref{eq:Ito_energy_inequality} and collecting the previous estimates as well as using that $\usigma=u$ on $[ 0,\stopp)\times \O$ and $\stopp=\sigma_j$ we get \eqref{eq:Gronwall_estimate_NS_2D}.
\end{proof}

\begin{remark}
\label{r:other_boundary_conditions}
As already mentioned at the beginning of this appendix several other boundary conditions can be used in \eqref{eq:Navier_Stokes_appendix}. We mention two possible choices in case $a^{i,j}=\delta^{i,j}$ (the Kronecker's delta): \emph{perfect slip} boundary conditions
$$
u\cdot \nu=0,\ \ \   P_{\partial\Dom}\big[(\nabla u-(\nabla u)^{\top})\nu \big]=0,\qquad \text{ on }\partial\Dom,
$$
or \emph{Navier} boundary  conditions
$$
u\cdot \nu=0, \ \ \  P_{\partial\Dom}\big[(\nabla u+(\nabla u)^{\top})\nu \big]+\alpha u=0,\qquad \text{ on }\partial\Dom,
$$
where $\alpha\geq 0$, $\nu$ is the outer normal vector field on $\partial\Dom$ and $P_{\partial\Dom}:=I-\nu\otimes \nu$ is the orthogonal projection on the tangent plane of $\partial\Dom$. In case $\Dom$ is a bounded $C^3$-domain, the weak formulation of the Stokes operators subject to one of the previous boundary conditions has been given in Subsection 2.3 and 3.2 of \cite{PW18}, respectively. Using the methods in the latter one can check that in case $a^{i,j}=\delta^{i,j}$ and $\Dom$ is a $C^3$-bounded domain, then the results of Theorem \ref{t:2D_global_rough_noise} extends to the case of perfect-slip or Navier boundary conditions. Indeed, \eqref{eq:boundarycondition} still holds since $u\cdot \nu=0$ on $\partial\Dom$ in both cases and one can replace the use of Lemma \ref{l:Sob_emb} by standard Sobolev embeddings due to the regularity of $\Dom$.
\end{remark}

\end{appendix}

\subsubsection*{Acknowledgements}

The authors thank the anonymous referees and Max Sauerbrey for careful reading and helpful comments.

\medskip
\noindent
{\bf Data Availability.} This manuscript has no associated data.

\bibliographystyle{alpha-sort}
\bibliography{literature}

\end{document}